\documentclass[a4paper,11pt]{article} %

\usepackage{amsmath} %
\usepackage{amsthm} %
\usepackage{amssymb} %
\usepackage{epsfig} %
\usepackage{psfrag} %
\usepackage[mathscr,mathcal]{eucal} %
\usepackage{enumerate} %
\usepackage{dsfont} %
\usepackage[margin=1in]{geometry} %
\usepackage{graphicx}

\theoremstyle{plain}

\newtheorem{theorem}{Theorem}[section]
\newtheorem{lemma}[theorem]{Lemma}

\newtheorem{proposition}[theorem]{Proposition}
\newtheorem{claim}[theorem]{Claim}

\newtheorem{question}[theorem]{Question}
\newtheorem{assumption}[theorem]{Assumption}

\theoremstyle{definition}

\newtheorem{definition}[theorem]{Definition}
\newtheorem{notation}[theorem]{Definition}
\newtheorem{remark}[theorem]{Remark}

\numberwithin{equation}{section}

\newcommand{\N}{\mathbb{N}}
\newcommand{\Z}{\mathbb{Z}}

\newcommand{\R}{\mathbb{R}}

\newcommand{\ind}{\mathds{1}} 
\newcommand{\veps}{\varepsilon}
\newcommand{\fsubset}{\subset \subset}

\newcommand{\Falg}{\mathcal{F}}

\newcommand{\ball}{\mathcal{B}}

\newcommand{\prob}{\mathbb{P}}
\newcommand{\Var}{\mathrm{Var}}

\newcommand{\energy}{\mathcal{E}}
\newcommand{\capacity}{\mathrm{cap}}

\newcommand{\POI}{\mathrm{POI}}

\newcommand{\Dlow}{\underline{\Delta}}
\newcommand{\Dup}{\overline{\Delta}}
\newcommand{\Rs}{R_s}
\newcommand{\Rb}{R_b}

\newcommand{\pinkwalk}{Z}
\newcommand{\worms}{\mathcal{X}}

\newcommand{\hitonly}[1]{{}^{\star}\!{#1}}

\newcommand{\CLowerFugaHitDELTALOWER}{\underline{\delta}}	
\newcommand{\CLowerFugaHitDELTAUPPER}{\overline{\delta}}	
\newcommand{\CLowerFugaHitLOWER}{c_1}			
\newcommand{\CPrfLowerFugaFirstLOWER}{c_2}	
\newcommand{\CPrfLowerFugaHitEXACTSTARUPPER}{c_3}	

\newcommand{\CLowerNOVisitedBoxesLOWER}{c_4}			
\newcommand{\CLowerNOVisitedBoxesRATIOLOWER}{C_1}	
\newcommand{\CCLAIMConsecutiveVisits}{C_2}		
\newcommand{\CPrfEasyLEDBoundOnExpETAKS}{C_3}			

\newcommand{\CSmallEnergyUPPER}{C_4}							

\newcommand{\CBigTotalEnergyDENOMINATOR}{c_5}					
\newcommand{\CBigTotalEnergyPARAMETERUPPER}{c_6}				
\newcommand{\CBigTotalEnergyUPPER}{C_5}						
\newcommand{\CBigTotalEnergyPARAMETERLOWER}{C_6}				

\newcommand{\CPrfDoublingGammaAUX}{c_7}						

\newcommand{\CCapaHyUUPER}{c_8}								

\newcommand{\CCondExpRestrictMeasureTotalLOWER}{c_9}			

\newcommand{\CCondVarRestrictMeasureTotalUPPER}{C_7}			

\newcommand{\CExpectBoundInteractionsUPPER}{C_8}				

\newcommand{\CGrimmettMarstrand}{c_{0}}  			

\begin{document}


\title{Percolation of worms}
\author{Bal\'azs R\'ath\footnote{Department of Stochastics, Institute of Mathematics, Budapest University of Technology and  Economics, M\H{u}egyetem rkp. 3., H-1111 Budapest, Hungary; MTA-BME Stochastics Research Group, M\H{u}egyetem rkp. 3., H-1111 Budapest, Hungary;
  Alfr\'ed R\'enyi Institute of Mathematics, Re\'altanoda utca 13-15, 1053 Budapest, Hungary.
 rathb@math.bme.hu}, \; S\'andor Rokob\footnote{MTA-BME Stochastics Research Group, M\H{u}egyetem rkp. 3., H-1111 Budapest, Hungary
 ;
 Alfr\'ed R\'enyi Institute of Mathematics, Re\'altanoda utca 13-15, 1053 Budapest, Hungary.
   roksan@math.bme.hu }}

\maketitle

\begin{abstract}

We introduce a new correlated percolation model on the $d$-dimensional lattice $\mathbb{Z}^d$ called the \emph{random length worms model}.
Assume given a probability distribution on the set of positive integers (the length distribution) and $v \in (0,\infty)$ (the intensity parameter). From each site of $\mathbb{Z}^d$ we start $\mathrm{POI}(v)$ independent simple random walks
 with this length distribution. We investigate the connectivity properties of the set $\mathcal{S}^v$ of sites visited by this cloud of random walks. It is easy to show that if the second moment of the length distribution is finite then $\mathcal{S}^v$ undergoes a percolation phase transition as $v$ varies.
Our main contribution is a sufficient condition on the length distribution which guarantees that $\mathcal{S}^v$
percolates for all $v>0$ if $d \geq 5$. E.g., if the probability mass function of the length distribution is
 \[ m(\ell)= c \cdot \ln(\ln(\ell))^{\varepsilon}/ (\ell^3 \ln(\ell)) \mathds{1}[\ell \geq \ell_0] \]
  for some $\ell_0>e^e$ and $\varepsilon>0$ then $\mathcal{S}^v$
percolates for all $v>0$. Note that the second moment of this length distribution is only ``barely'' infinite. In order to put our result in the context of earlier results about similar models (e.g., finitary random interlacements, loop percolation, Bernoulli hyper-edge percolation, Poisson Boolean model, ellipses percolation, etc.), we define a natural family of percolation models called the \emph{Poisson zoo} and argue that the percolative behaviour of the random length worms model is quite close to being ``extremal'' in this family of  models.

\bigskip

\medskip


\end{abstract}

\newpage

\tableofcontents

\section{Introduction}

 The mathematical theory of percolation studies the connectivity properties of certain random subgraphs of an infinite graph $G$. One typically looks at a family of probability measures on the space of subgraphs of $G$ parametrized by the density of the random subgraph and says that the percolation model exhibits phase transition if there exists a non-trivial value of the density parameter (called the critical threshold) such that if the density is smaller than the threshold then the connected clusters of the random subgraph are all finite (subcritical phase), but if the density is bigger than the threshold then there exists an infinite cluster (supercritical phase).

The prime example of a percolation model which exhibits phase transition is Bernoulli site percolation on $\mathbb{Z}^d, d \geq 2$, where
the random subgraph is created by retaining the vertices of the $d$-dimensional lattice independently from each other with the same probability $p$, cf.\ \cite{Gr99}.

However, there is no recipe to decide whether or not a general correlated percolation model exhibits phase transition. The goal of this paper is to study a natural class of correlated percolation models and explore the conditions under which a model \emph{does not} exhibit phase transition in the sense that it is supercritical for all positive values of the density parameter.

\medskip

In Section \ref{subsection_def_poi_zoo} we introduce a  general correlated site percolation model $\mathcal{S}^v$ on  $\Z^d$ that we call the  \emph{Poisson zoo} at level $v$.
Loosely speaking, $\mathcal{S}^v$ is the trace of a Poisson point process (PPP) of connected subsets  of $\mathbb{Z}^d$ (i.e., lattice animals). The intensity measure of this PPP is of form $v \cdot \mu$, where $v \in \mathbb{R}_+$ and $\mu$ is a measure on lattice animals which is invariant under the translations of $\mathbb{Z}^d$. We state two basic results about the Poisson zoo: loosely speaking, Lemma \ref{lemma_infinite_m1_perc} states that if the expected number of animals that contain the origin is infinite for any $v>0$ then  $\mathcal{S}^v$ percolates for all  $v>0$, while
Lemma \ref{lemma_m_2_finite_no_perc} states that if the expected total cardinality of animals that contain the origin is finite for any $v>0$ then $\mathcal{S}^v$ exhibits percolation phase transition as $v$ varies. One wonders whether the conditions of Lemmas \ref{lemma_infinite_m1_perc} and \ref{lemma_m_2_finite_no_perc} can be sharp for specific choices of $\mu$. We will briefly argue that the condition of Lemma \ref{lemma_infinite_m1_perc} is sharp if one chooses the animals to be balls (this follows from the main result of \cite{Gouere2008} on the Poisson Boolean model). The main results of our paper are related to the question of sharpness of Lemma \ref{lemma_m_2_finite_no_perc}.

In Section \ref{subsection_def_worms_rl} we introduce a special case of the Poisson zoo model that we call the \emph{random length worms model} and
state Theorem \ref{main_thm}, the main result of this paper.  Theorem \ref{main_thm} gives a condition on the length distribution of worms under which the occupied set percolates  for all positive values of  $v$ when $d \geq 5$. We will see that the gap between the condition of Lemma \ref{lemma_m_2_finite_no_perc} (which guarantees phase transition) and the condition of Theorem \ref{main_thm} (which guarantees that there is no phase transition) is very small. The question of  sharpness of the condition of Lemma \ref{lemma_m_2_finite_no_perc} for the random length worms model remains open for $d\geq 5$, but the results of \cite{ChangSapozhnikov2016} about loop percolation make us conjecture that the condition of Lemma \ref{lemma_m_2_finite_no_perc} is not sharp for the random length worms model if $d \leq 4$.

In Section \ref{subsection_related_literature} we discuss some results about related percolation models from the literature, namely the Poisson Boolean model, ellipses percolation, finitary random interlacements, loop percolation, Bernoulli hyper-edge percolation and Wiener sausage percolation.

In Section \ref{subsection_methods} we explain the ideas of the proof of our main result.

In Section \ref{sub_conj_d234} we formulate some conjectures about the percolation of worms in $d=2,3,4$.

In Section \ref{sub_structure_of_paper} we outline the structure of the rest of the paper.

\subsection{Poisson zoo}\label{subsection_def_poi_zoo}

Denote by $\N = \{ 1, 2, \ldots \}$ the set of natural numbers, by $\N_0 = \N \cup \{ 0 \}$ the set of non-negative integers, and by $\Z$ the set of integers. $\mathbb{Z}^d$ denotes
 the $d$-dimensional integer lattice (i.e., the space of $d$-dimensional vectors with integer coordinates). Let $o$ denote the origin of $\mathbb{Z}^d$.

We think about $\mathbb{Z}^d$ as a graph where nearest neighbour vertices are connected by an edge. We say that a set $H \subseteq \mathbb{Z}^d$ is connected
if the subgraph of $\mathbb{Z}^d$ spanned by $H$ is connected.

\begin{definition}[Lattice animals]\label{def_lattice_animal}
Let us denote by $\mathcal{H}$ the set of finite, connected, nonempty subsets of $\Z^d$ that contain the origin. We say that $\mathcal{H}$ is the collection of \emph{lattice animals} rooted at the origin $o$. A pair $(x,H) \in \mathbb{Z}^d \times \mathcal{H}$ is called a lattice animal rooted at $x$, and we define the trace of such an animal to be the translated set $x+H$.
\end{definition}

\begin{definition}[Poisson zoo]
	\label{poisson_zoo}
	Let $\nu$ denote a probability measure on $\mathcal{H}$. Given some $v \in \R_{+}$, let us consider independent Poisson random variables
	$N_{x,H}^{v} \sim \POI( v \cdot \nu(H) )$ indexed by $x \in \Z^d$ and $H \in \mathcal{H}$.
	We say that $N_{x,H}^{v}$ is the number of copies of the animal $(x,H)$.
	We call the collection of random variables
	\begin{equation}
		\widehat{\mathcal{N}}^v := \left\{ N^v_{x,H} \, : \, x \in \mathbb{Z}^d, H \in \mathcal{H} \right\}
	\end{equation}
	 the \emph{Poisson zoo} at level $v$.
	We call the random set
	\begin{equation}\label{def_eq_trace_of_zoo}
	\mathcal{S}^v := \bigcup_{x \in \Z^d} \bigcup_{H \in \mathcal{H}} \bigcup_{i = 1}^{ N_{x,H}^{v} } \big( x + H \big)
	\end{equation}
	the	\emph{trace of the Poisson zoo} at level $v$.
\end{definition}
Note that if $N_{x,H}^{v}=0$ then $\bigcup_{i = 1}^{ N_{x,H}^{v} } \big( x + H \big)=\emptyset$, but if  $N_{x,H}^{v}\geq 1$ then  $\bigcup_{i = 1}^{ N_{x,H}^{v} } \big( x + H \big)=x+H$.

\begin{remark}\label{remark_other_def_oz_zoo}
We may alternatively define $\mathcal{S}^v$ by considering i.i.d.\ random variables $(N^v_x, \, x \in \mathbb{Z}^d)$ with $\POI(v)$ distribution
and i.i.d.\ random sets $(\mathrm{H}^{x,k}, \, x \in \mathbb{Z}^d, k \in \mathbb{N})$ with distribution $\nu$. Here $N^v_x$ denotes the number of animals rooted
at $x$  and if we define $N^v_{x,H}:=\sum_{k=1}^{N^v_x} \mathds{1}[ \mathrm{H}^{x,k} = H ] $ then $\{ N^v_{x,H} \, : \, x \in \mathbb{Z}^d, H \in \mathcal{H} \}$
will be a Poisson zoo at level $v$. Thus, we can alternatively define the trace of the Poisson zoo at level $v$  as $\mathcal{S}^v=\cup_{x \in \Z^d}  \cup_{k=1}^{N^v_x} (x+\mathrm{H}^{x,k})$.
\end{remark}

Note that
\begin{equation}\label{eq_ergodic}
\text{ the law of $\mathcal{S}^v$ is invariant and ergodic under the translations of $\mathbb{Z}^d$,}
 \end{equation}
 since $\mathcal{S}^v$ is a factor of i.i.d.\ process (cf.\ \eqref{def_eq_trace_of_zoo}) and Bernoulli shifts are ergodic.

We say that a subset $S$ of $\mathbb{Z}^d$ percolates if $S$ has an infinite connected component.

It follows from \eqref{eq_ergodic} that the value of $\mathbb{P}( \mathcal{S}^v \text{ percolates}  )$ is either zero  or one.

Under the monotone coupling of the Poisson zoos $\widehat{\mathcal{N}}^{v}$ and $\widehat{\mathcal{N}}^{v'}$ with the same $\nu$ we have
\begin{equation}\label{mon_coupl_Sv}
\mathcal{S}^{v} \subseteq \mathcal{S}^{v'} \quad \text{if} \quad
0 \leq v \leq v'.
\end{equation}

Let us define the critical threshold $v_c=v_c(\nu)$ of the Poisson zoo percolation model as
\begin{equation}\label{zoo_v_c_threshold}
v_c := \inf\{\, v >0 \, : \, \mathbb{P}( \mathcal{S}^v \text{ percolates}  )=1 \, \}.
\end{equation}
Using \eqref{mon_coupl_Sv} we see that $\mathbb{P}( \mathcal{S}^v \text{ percolates}  )=1$ if $v>v_c$, but $\mathbb{P}( \mathcal{S}^v \text{ percolates}  )=0$ if $v<v_c$.

Various correlated percolation models (e.g., finitary random interlacements and loop percolation) arise as a  special case of the Poisson zoo model by choosing a specific
  probability measure $\nu$ on $\mathcal{H}$ in Definition \ref{poisson_zoo}. We will give an overview of the literature of such models in Section \ref{subsection_related_literature}. The basic question for any such percolation model concerns the non-triviality of phase transition: for which choices of $\nu$ do we have $v_c \in (0,\infty)$?

Let us first argue that if $d \geq 2$ then $v_c<+\infty$ holds for any choice of $\nu$. Indeed, any $H \in \mathcal{H}$ contains the origin $o$ of $\mathbb{Z}^d$, therefore the trace $\mathcal{S}^v$ of the Poisson zoo at level $v$ can be coupled to an i.i.d.\ Bernoulli site configuration $\mathcal{B}^p$ with density $p=1-e^{-v}$ in such a way that $\mathcal{B}^p \subseteq \mathcal{S}^v$. This implies $v_c \leq -\ln\left( 1-p_c \right)$, where $p_c=p_c^{\text{site}}(\Z^d) \in (0,1)$ denotes the critical threshold of Bernoulli site percolation on the nearest neighbour lattice $\mathbb{Z}^d$.

Therefore the real question is whether the Poisson zoo model has a non-trivial subcritical phase (i.e., $v_c>0$) or is it supercritical for all $v >0$ (i.e., $v_c=0$).
Let us now formulate two general lemmas that give sufficient conditions for $v_c=0$ and $v_c>0$, respectively.

Let $|H|$ denote the cardinality of $H \in \mathcal{H}$.

\begin{definition}[Moments]
	Given $k \in \N$, let us denote the $k$'th moment of the cardinality of a random subset with law $\nu$ by
	\begin{equation}\label{def_eq_moments}
	m_k := m_k(\nu):= \sum_{H \in \mathcal{H}} |H|^k \cdot \nu(H).
	\end{equation}
\end{definition}

\begin{lemma}[Infinite first moment implies percolation for any $v>0$]\label{lemma_infinite_m1_perc}
	If $m_1 = \infty$, then for any $v \in (0,+\infty)$ we have $\mathcal{S}^v=\mathbb{Z}^d$. This implies in particular that  $\mathcal{S}^v$ percolates  for any $v \in \R_{+}$.
\end{lemma}
\noindent In short: if $m_1=+\infty$ then $v_c=0$.  We will prove Lemma \ref{lemma_infinite_m1_perc} in Section \ref{section_moment_cond_zoo}.

Let us note without a proof that if $d=1$ then $m_1<+\infty$ implies $v_c=+\infty$. Let us also note without a proof that for any $d \geq 1$ the condition
$m_1<+\infty$ implies that for any $v>0$ we have $\mathcal{S}^v \neq \mathbb{Z}^d$ and in fact the density of $\mathcal{S}^v$ is $1-\exp\left( -v \cdot m_1 \right)$,
which can be made arbitrarily small by making $v$ small. Nevertheless, this does not always imply that $\mathcal{S}^v$ only contains finite connected components when $v$ is small enough: examples of percolation models  for which $m_1<+\infty$ and $v_c=0$ is possible include \emph{planar ellipses percolation} \cite{TeixeiraUngaretti2017} (see Section \ref{subsub_ellipses}) and the \emph{random length worms model} on $\mathbb{Z}^d, d \geq 5$ (to be introduced in Section \ref{subsection_def_worms_rl}).

   Our next lemma gives a sufficient condition for a non-trivial percolation phase transition.

\begin{lemma}[Finite second moment implies existence of subcritical phase]\label{lemma_m_2_finite_no_perc}
	If $m_2 < \infty$ then for any $v \in \big(0,  1 \slash ((2d+1) \cdot m_2)\big)$ the set $\mathcal{S}^v$ does not percolate.
\end{lemma}
\noindent In short: if $m_2<+\infty$ then $v_c>0$.
We will prove Lemma \ref{lemma_m_2_finite_no_perc} in Section \ref{section_moment_cond_zoo}.

Note that one can prove Lemma \ref{lemma_m_2_finite_no_perc} by dominating the exploration of the animals that make up the connected cluster of $o$ in $\mathcal{S}^v$ by a subcritical branching process
with a compound Poisson offspring distribution. The reason why the second moment $m_2$ enters the picture is that the expected total cardinality of animals that contain a fixed vertex
is related to $m_2$ by size-biasing.
 If $m_2=+\infty$ then the expected number of offspring in the dominating branching process is infinite for all $v>0$, therefore it survives with positive probability. One might naively think that this survival property can be transferred back to the exploration of the animals that make up the connected cluster of the origin in $\mathcal{S}^v$, which would imply that $\mathcal{S}^v$ percolates. The problem with this approach is that the total population of a supercritical branching process  conditioned on survival grows much faster than the volume of the region occupied by the associated branching random walk on $\mathbb{Z}^d$, see e.g.\ \cite{biggins_shape}.
  Thus, on the long run, the exploration of $\mathcal{S}^v$ behaves very differently compared to the dominating branching process, since it inevitably  runs out of new animals to be explored in larger and larger spatial regions surrounding the origin.

Nevertheless, the gap between the conditions of Lemmas \ref{lemma_infinite_m1_perc} and \ref{lemma_m_2_finite_no_perc} is rather wide, so one wonders whether it is possible to provide weaker moment conditions that guarantee either $v_c=0$ or $v_c>0$ in the Poisson zoo percolation model. Let us now argue that one cannot tighten this gap by much, i.e., both Lemmas \ref{lemma_infinite_m1_perc} and \ref{lemma_m_2_finite_no_perc} are sharp (or at least very close to being sharp).

On the one hand,  the condition of Lemma \ref{lemma_infinite_m1_perc} is sharp if $\nu$ is the law of a ball $\ball(o,\mathrm{R})$ w.r.t.\ the $\ell_{\infty}$-norm on $\mathbb{Z}^d$ and the radius $\mathrm{R}$ of the ball is an $\mathbb{N}_0$-valued random variable. Indeed, in Section \ref{subsection_related_literature} we will show that it easily follows from the main result of \cite{Gouere2008} that
 $m_1<+\infty$ implies $v_c>0$ for this choice of $\nu$.

On the other hand, the theorem that we will state in Section \ref{subsection_def_worms_rl} shows that if $d \geq 5$ then we can choose $\nu$  in such a way that
 $m_2(\nu)$ is ``barely'' infinite, yet $v_c=0$ holds (i.e., the condition of Lemma \ref{lemma_m_2_finite_no_perc} is very close to being sharp).
   However, the question whether the condition of Lemma \ref{lemma_m_2_finite_no_perc} can be weakened remains open:

 \begin{question}\label{question_strenghten_lemma_m2}
  Given  $d \geq 2$, is there  a function $f: \mathbb{N} \to \mathbb{R}_+$ satisfying $\lim_{n \to \infty} f(n)/n^2 =0$ such that
  for any choice of $\nu$ the condition $\sum_{H \in \mathcal{H}} f(|H|) \cdot \nu(H)<+\infty$ implies $v_c(\nu)>0$?
  \end{question}

\subsection{Random length worms model}\label{subsection_def_worms_rl}

  The main percolation model of this paper can be defined by specifying the measure $\nu$ on $\mathcal{H}$ that appears in
   the definition of the Poisson zoo  (cf.\ Definition \ref{poisson_zoo}).

\begin{definition}[Random length worms model]\label{def_worms_nu}
	Let us consider a probability mass function $m: \N \to \mathbb{R} $ and an $\N$-valued random variable $\mathcal{L}$ satisfying $\mathbb{P}(\mathcal{L}=\ell)=m(\ell), \, \ell \in \mathbb{N}$. We call the distribution of $\mathcal{L}$ the \emph{length distribution} of worms. Let us also consider
 a nearest neighbour simple random walk $(X(t))_{t \in \N_0}$ on $\Z^d$ starting from the origin,  independent of $\mathcal{L}$. Let
	\begin{equation*}
	\mathcal{R} := \left\{\, X(0), X(1), \ldots, X(\mathcal{L}-1)\, \right\}
	\end{equation*}
	denote the set of sites visited by the random walk up to time $\mathcal{L}-1$.
	Let us define the probability measure $\nu$  on $\mathcal{H}$ by $\nu(H) = \mathbb{P}( \mathcal{R} = H )$ for each $H \in \mathcal{H}$.
	Given some $v \in \R_{+}$, the random set $\mathcal{S}^v$ obtained via Definition \ref{poisson_zoo} with this  $\nu$ is called the \emph{random length worms set} at level $v$.
\end{definition}

Now we are ready to state the main result of this paper.

\begin{theorem}[Supercritical worm percolation]\label{main_thm} Let $d \geq 5$.
Let $\varepsilon>0$ and $\ell_0 \geq e^e$. If
\begin{equation}\label{loglog_epsilon_main}
  m(\ell)= c \frac{\ln(\ln(\ell))^{\varepsilon}}{\ell^3 \ln(\ell)} \mathds{1}[\ell \geq \ell_0], \quad \ell \in \mathbb{N}, \quad
  \text{(where $c$ is chosen so that $\sum_{\ell} m(\ell)=1$)}
  \end{equation}
 then for any $v>0$ the random length worms model $\mathcal{S}^v$ is supercritical:
$\mathbb{P}(\, \mathcal{S}^v \text{ percolates}\, )=1$.
 \end{theorem}
\noindent In short, if $d \geq 5$, $\varepsilon>0$ and $m$ is of form \eqref{loglog_epsilon_main} then $v_c=0$.

\begin{remark}\label{remark_after_main_thm} $ $

\begin{enumerate}[(i)]
\item Observe that if we choose $\varepsilon < -1$ in \eqref{loglog_epsilon_main} then $\mathbb{E}[\mathcal{L}^2]= \sum_{\ell=1}^{\infty} \ell^2 m(\ell) <+\infty $, which implies $m_2<+\infty$ (cf.\ \eqref{def_eq_moments}) and  therefore $v_c>0$ by Lemma \ref{lemma_m_2_finite_no_perc}. This indicates that the gap between the conditions of
    Theorem \ref{main_thm} and Lemma \ref{lemma_m_2_finite_no_perc} is of ``loglog'' order.

\item Also note that if $\varepsilon>0$ in \eqref{loglog_epsilon_main} then  $\mathbb{E}[\mathcal{L}^2/\ln(\ln(\mathcal{L}))^{1+\alpha}] <+\infty $ if and only if $\alpha>\varepsilon$,
 which is just another way to see that the gap between
    Theorem \ref{main_thm} and Lemma \ref{lemma_m_2_finite_no_perc} is small.

\item   Theorem \ref{main_thm} can be strengthened:
 in Section \ref{subsection_good_seq_of_scales} we state Theorem \ref{thm_worm_perco} which allows us to conclude  $v_c=0$
 under weaker  conditions on $m$ than Theorem \ref{main_thm}. The reason why we decided to defer the statement of Theorem \ref{thm_worm_perco} to a later section is that the conditions of Theorem \ref{thm_worm_perco} are more complicated to formulate than those of
 Theorem \ref{main_thm}.  Let us also note  that the proof of Theorem \ref{thm_worm_perco} also implies percolation of worms of bounded length in thick enough slabs, i.e., for each $v>0$ there is a length threshold $L=L(v)$ and a thickness $R=R(v)$ such that if $\mathcal{S}^{v}_L$ denotes the union of the traces of worms of length at most $L$ then
 $\mathcal{S}^{v}_L \cap (\mathbb{Z}^2 \times \{0,1,\dots,R \}^{d-2})$ almost surely percolates.

\item \label{remark_stoch_dom} If the law of the length variable $\widetilde{\mathcal{L}}$ stochastically dominates the law of the length variable $\mathcal{L}$ then one can couple the corresponding
random length worms sets at the same level $v$ in such a way that $\mathcal{S}^v \subseteq \widetilde{\mathcal{S}}^v$ holds. This implies that the corresponding percolation thresholds satisfy $\widetilde{v}_c \leq v_c$, and thus $v_c=0$ implies $\widetilde{v}_c=0$.

\item The main reason why we assume $d \geq 5$ in Theorem \ref{main_thm} is that we must assume this if we want the length and the capacity of a worm to be comparable (see Proposition \ref{prop_lln_cap_rw}).

\item  If $d \geq 5$ then  $m_2=+\infty$ is equivalent to  $\mathbb{E}[\mathcal{L}^2]=+\infty$ by \cite[Theorems 1, 2]{dvoer}.

\end{enumerate}
\end{remark}

\begin{question}\label{question_worm_m2} Let $d \geq 5$.
Is it possible to conclude $v_c=0$ in the random length worms model if we only assume $m_2=+\infty$?
\end{question}

\noindent Note that if the answer to Question \ref{question_worm_m2} is affirmative then Lemma \ref{lemma_m_2_finite_no_perc} is sharp in the sense that
 the answer to Question \ref{question_strenghten_lemma_m2} is negative if $d=5$.

\noindent We formulate some conjectures about the percolation of worms when $d=2,3,4$ in Section \ref{sub_conj_d234}.

\subsection{Related results from the literature}\label{subsection_related_literature}

Let us compare our models, methods and results to earlier models, methods and results.

\subsubsection{Poisson Boolean model}\label{subsub_poi_boolean}

The \emph{Poisson Boolean model} is a continuum percolation model where every point of a homogeneous PPP on $\mathbb{R}^d$ is the center of a Euclidean ball with a random radius $\mathrm{R}_*$ (where it is also assumed that the radii corresponding to different points are identically distributed and independent of each other and the PPP).
Although the literature of this model is vast, now we will focus on one specific result which is the most relevant from our viewpoint (a great collection of properties of the Poisson Boolean model can be found in \cite{MeesterRoy1996}).

In \cite{Gouere2008} the author shows that for any $d \geq 2$ there is phase transition for the existence of an infinite connected cluster in the union of the balls
 if and only if the first moment of the volumes (or the $d$'th moment of the radii $\mathrm{R}_*$) are finite. Let us now discuss the implications of this result regarding our Poisson zoo percolation model.

  We want to deduce from the main result of \cite{Gouere2008} that the condition of Lemma \ref{lemma_infinite_m1_perc} is sharp if  $\nu$ (cf.\ Definition \ref{poisson_zoo}) is the law of a ball $\ball(o,\mathrm{R})$ w.r.t.\ the $\ell_{\infty}$-norm on $\mathbb{Z}^d$ and the radius $\mathrm{R}$ of the ball is an $\mathbb{N}_0$-valued random variable: we want to show that $m_1<+\infty$ implies $v_c>0$ for this choice of $\nu$. Indeed, given a radius variable $\mathrm{R}$ in the Poisson zoo, let us define the  radius variable
  $\mathrm{R}_*$ in the Poisson Boolean model to be $\mathrm{R}_*=\sqrt{d}\cdot ( \mathrm{R}+3 )$. Now if we are given a Poisson Boolean model with intensity $v$ and radius distribution $\mathrm{R}_*$, we can derive from it a Poisson zoo model at level $v$ with radius distribution $\mathrm{R}$ in the following fashion: define the number of animals rooted at $x \in \mathbb{Z}^d$ to be the number of
  points $\omega$ of the PPP of the  Boolean model whose coordinate-wise lower integer part  is equal to $x$, and if $R^\omega_*$ is the radius of the ball around the point $\omega$   in the Poisson Boolean model, then we define the corresponding radius in the Poisson zoo to be $\frac{1}{\sqrt{d}}R^\omega_*-3$. One can check that under this coupling the absence of percolation in the Poisson Boolean model implies absence of percolation in $\mathcal{S}^v$. It is also easy to check that $m_1<+\infty$ if and only if
   $\mathbb{E}[ \mathrm{R}^d] <+\infty $ if and only if $\mathbb{E}[ (\mathrm{R}_*)^d] <+\infty $, thus the main result of \cite{Gouere2008} implies $v_c>0$.


\subsubsection{Ellipses percolation}\label{subsub_ellipses}

The \emph{ellipses percolation} model is a continuum percolation model introduced in \cite{TeixeiraUngaretti2017} and is defined as follows.
One considers  ellipses centered on a homogeneous Poisson point process on $\mathbb{R}^2$
 with uniformly random direction, minor axis fixed to be equal to one and major axis with distribution $\rho$ supported on $[1, \infty)$.
 The authors only consider distributions $\rho$ satisfying $\rho[r,+\infty) \asymp r^{-\alpha}$ and find that (i)  if $0 <\alpha \leq 1$ then the set covered by the ellipses is $\mathbb{R}^2$ (this is similar to our Lemma \ref{lemma_infinite_m1_perc}, but in the continuum setting $m_1$ denotes the expected Lebesgue measure of an animal), (ii) if $\alpha \geq 2$ then the set covered by the ellipses does not percolate if the intensity of the PPP is low enough  and (iii) if $1<\alpha<2$ then the covered set is not equal to $\mathbb{R}^2$ but it percolates for any positive value of the intensity of the PPP. Thus, heuristically speaking, (the continuous analogue of) Lemma \ref{lemma_m_2_finite_no_perc} is almost sharp (but not exactly sharp) for the ellipses percolation model: the second moment $m_2$ of the area of one ellipse is finite if and only if $\alpha>2$, but if $\alpha=2$ then $m_2=+\infty$, yet the authors of \cite{TeixeiraUngaretti2017} prove a percolation phase transition for the covered set in the $\alpha=2$ case. See Section \ref{subsub_log_corr} for some further comparison of ellipses percolation and worm percolation in the $\alpha=2$ case.

  Note that more subtle results about the percolation of the covered set in the $\alpha=2$ case (which arises from the intersection of the 3 dimensional \emph{Poisson cylinder model} \cite{TykessonWindisch2012} with a plane) and the percolation of the vacant set are also proved in  \cite{TeixeiraUngaretti2017}. The geometry of the infinite cluster of ellipses percolation is further studied in \cite{HilarioUngaretti2021} in the case when $1<\alpha<2$.

\subsubsection{Finitary random interlacements}\label{subsub_FRI}

Although \emph{finitary random interlacements (FRI)} was defined in \cite{Bowen2019} in a more general setting,  here we will  only discuss FRI on $\mathbb{Z}^d, \, d \geq 3$, which can be viewed as the special case of our random length worms model when the length $\mathcal{L}$ of a worm has geometric distribution with expectation $T$. Thus, if we fix such a length distribution then the occupied set exhibits a percolation phase transition as the intensity parameter varies by our Lemma \ref{lemma_m_2_finite_no_perc}.
However, the parametrization of FRI is different than ours: finitary random interlacements $\mathcal{FI}^{u,T}$ at level $u$ is the same as our  $\mathcal{S}^v$ if $v=2du/(T+1)$.
The reason for this parametrization  is that by \cite[Theorem A.2]{Bowen2019} if we let $T \to \infty$ then $\mathcal{FI}^{u,T}$ weakly converges to random interlacements $\mathcal{I}^{2du}$  at level $2du$ (defined in \cite{Sznitman2010}). Note that $\mathcal{I}^u$ is connected for all $u>0$ by \cite[Corollary 2.3]{Sznitman2010}.

 The main result of \cite{ProcacciaYeZhang2021} is that for any $u > 0$ there exist $0 < T_0(u,d) \leq T_1(u,d) < \infty$ such that for all $T > T_1(u,d)$ the set $\mathcal{FI}^{u,T}$ percolates and for all $0 < T < T_0(u,d)$ the set $\mathcal{FI}^{u,T}$ does not percolate.  Currently it is conjectured (but not yet proved) that $T_0(u,d) = T_1(u,d)$, the main obstacle of a rigorous proof being the lack of stochastic monotonicity of $\mathcal{FI}^{u,T}$ in the variable $T$, see \cite{CaiXiongZhang2021}.
In \cite{CaiHanYeZhang2020} the authors study the large-scale geometry and the local connectivity properties of the infinite cluster of $\mathcal{FI}^{u,T}$ for large values of $T$.

Given some $T \in \mathbb{R}_+$, let $u_*(T)$ denote critical percolation threshold of $\mathcal{FI}^{u,T}$ (as $u$ varies).   In  \cite{CaiZhang2021b} it is proved that $u_*(T) \asymp T^{-1}$ if $d \geq 5$, $u_*(T) \asymp T^{-1}\ln(T)$ if $d=4$ and $u_*(T) \asymp T^{-1/2}$ if $d=3$ (see Section \ref{subsub_sausage} for further discussion).
In \cite{CaiProcacciaZhang2021} it is proved that various natural critical percolation thresholds characterizing local subcriticality and local supercriticality of $\mathcal{FI}^{u,T}$ all agree with $u_*(T)$, and it is shown that $u_*(T)$ is a continuous function of $T$ for any $d \geq 3$.

\subsubsection{Loop percolation, Bernoulli hyper-edge percolation}\label{subsub_loop}

Percolation of the \emph{random walk loop soup} on $\mathbb{Z}^d$ was initiated in \cite{LeJanLemaire2013}. Here we will only discuss the case when the killing parameter $\kappa$ is equal to zero and the underlying Markov process is the simple random walk on $\mathbb{Z}^d$. For any $n \geq 2$ an element $\dot{\ell}=(x_0,\dots,x_{n-1})$ of $(\mathbb{Z}^d)^n$
is called a non-trivial discrete based loop (based at $x_0$) if $x_{i+1}$ is a nearest neighbour of $x_i$ for all $0\leq i \leq n-2$, moreover $x_{n-1}$ and $x_0$ are also nearest neighbours.
The weight of a based loop $\dot{\ell}=(x_0,\dots,x_{n-1})$ is defined to be equal to
  \[\mu(\dot{\ell})=\frac{1}{n} (2d)^{-n} =P_{x_0}\big(\, X(1)=x_1,\, \dots,\, X(n-1)=x_{n-1},\, X(n)=x_0 \, \big),  \]
where $P_{x_0}$ denotes the law of a simple random walk $(X(t))_{t \in \mathbb{N}_0}$ starting from $x_0$.
 Given some intensity parameter $\alpha>0$, one considers a Poisson point process on the space of based loops with intensity measure $\alpha \mu$ and declares an edge of $\mathbb{Z}^d$ to be open if it is crossed by a based loop of the PPP. \cite[Theorem 1.1]{ChangSapozhnikov2016}
 states that if $d \geq 3$ then this bond percolation model undergoes a non-trivial phase transition as $\alpha$ varies.

  Let us mention that the corresponding site percolation model (where we declare a site to be open if it is visited by a loop of the PPP)
can be viewed as a special case of the Poisson zoo percolation model. Indeed, the total weight of loops based at the origin can be normalized to be a probability measure $\mathrm{P}$, and then we can define the measure $\nu$ (cf.\ Definition \ref{poisson_zoo}) to be the law of the trace of a loop with law $\mathrm{P}$.
Note that the probability  under $\mathrm{P}$ that a loop has (an even) length $n$ is asymptotically comparable to $n^{-(d+2)/2}$, thus,  using the notation of \eqref{def_eq_moments}, we have $m_1<+\infty$ if and only if $d \geq 3$ and $m_2<+\infty$ if and only if $d\geq 5$. Indeed, the $d \geq 5$ case of loop percolation is treated
in  \cite[Section 5.1]{ChangSapozhnikov2016} using ideas similar to our Lemma \ref{lemma_m_2_finite_no_perc}. The proof of the existence of a subcritical phase in loop percolation is substantially harder if $d=3,4$, see the proof of \cite[Theorem 1.1]{ChangSapozhnikov2016}.

  Heuristically, loop percolation on $\mathbb{Z}^d$ is  similar to our random length worms model on $\mathbb{Z}^d$ with length mass function $m(\ell) \asymp \ell^{-(d+2)/2}$ (cf.\ Definition \ref{def_worms_nu}). In Section \ref{sub_conj_d234} we conjecture that if $m(\ell) \asymp \ell^{-\beta}$ then the most interesting choice of the exponent is $\beta=(d+2)/2$ in both dimensions $3$ and $4$ (therefore, loop percolation is interesting in both dimesions $d=3,4$).

Note that some elements of our proof are similar to those used in the proof of the lower bounds stated in \cite[Theorems 1.5 and 1.6]{ChangSapozhnikov2016} on the one-arm probability of loop percolation in $d=3,4$,  see Section \ref{subsub_further_ideas_main_thm} for a discussion of similarities and differences.

Further results about the supercritical phase of loop percolation on the $d$-dimensional lattice  (and its vacant set)  are proved in \cite{AlvesSapozhnikov2019, Chang2017}

\medskip

\emph{Bernoulli hyper-edge percolation} on $\mathbb{Z}^d$, introduced in \cite{Chang2021}, is a generalization of the above-discussed loop percolation model.
Let us consider a measure $\mu$ on the set $\mathcal{H}^*$ of finite subsets (i.e., the set of hyper-edges) of $\mathbb{Z}^d$. We assume that $\mu$ is invariant under the translations of $\mathbb{Z}^d$. We consider a Poisson point measure $\mathcal{X}$ on $\mathcal{H}^*$ with intensity $v \cdot \mu$, $v \in \mathbb{R}_+$, and declare that $x,y \in \mathbb{Z}^d$ are connected if there exists $k \in \mathbb{N}$ and $h_1,\dots,h_k \in \mathcal{H}^*$ satisfying $\mathcal{X}(\{ h_i \})\geq 1$ for every $i=1,\dots,k$ with $x \in h_1$, $y \in h_k$ and $h_i \cap h_{i+1} \neq \emptyset$ for every $i=1,\dots,k-1$. This percolation model is quite similar to our Poisson zoo model, but (i)
hyper-edge percolation is a generalization of Bernoulli \emph{bond} percolation, while our Poisson zoo is a generalization of Bernoulli \emph{site} percolation and (ii) we assume
 that the animals in the zoo are connected subsets of $\mathbb{Z}^d$, but this is not necessarily the case for hyper-edges. \cite[Theorem 1.1]{Chang2021} provides sufficient conditions under which the percolation phase transition of hyper-edge percolation is non-trivial. The condition that implies the existence of a non-trivial subcritical phase concerns the $\mu$-measure of hyper-edges that cross large annuli:
\begin{equation}\label{annuli_crossing}
\text{there exists }  \lambda>1 \text{ such that } \sup_{n \to \infty} \mu\left( \{ h \, : \,  h \cap \ball(n)\neq \emptyset , \; h \cap \ball(\lambda n)^c \neq \emptyset    \} \right)<+\infty.
\end{equation}
We note that the analogous condition implies $v_c>0$ in our Poisson zoo model, so let us compare it to the condition $m_2<+\infty$, which also implies $v_c>0$ according to
our Lemma \ref{lemma_m_2_finite_no_perc}. In the case of four dimensional loop percolation we have $m_2=+\infty$,  while
\cite[Lemma 2.7]{ChangSapozhnikov2016} implies that \eqref{annuli_crossing} holds in this case. On the other hand, if $d \geq 5$ then one can check using the arguments of our Section \ref{subsection_rw_estimates} that in the random length worms model condition \eqref{annuli_crossing} implies $m_2<+\infty$.

Other results of \cite{Chang2021} include natural conditions on $\mu$ under which the infinite cluster is unique and conditions under which a Grimmett-Marstrand-type result holds in the supercritical regime.

\subsubsection{Wiener sausage percolation}\label{subsub_sausage}

\emph{Wiener sausage percolation} on $\mathbb{R}^d$ is a continuum percolation model introduced in
\cite{ErhardMartinezPoisat2017}. Let us consider Wiener sausages with radius $r$  that we run up to time $t$ and whose initial points are distributed according to a homogeneous PPP
on $\mathbb{R}^d$ with intensity $1$. In \cite{ErhardMartinezPoisat2017} it is proved that if $d\geq 2$ and $r \in (0,+\infty)$ is small enough then there exists $t_c(r)\in (0,+\infty)$ such that if $t<t_c(r)$ then the occupied set does not percolate, but if $t>t_c(r)$ then it percolates.

\cite[Theorem 1.3]{ErhardMartinezPoisat2017} states that if $d=2,3$ then $t_c(r)\nearrow t_c(0)\in (0,+\infty)$ as $r \to 0$.

The main result \cite{ErhardPoisat2016} is that $t_c(r) \asymp r^{(4-d)/2} $ if $d \geq 5$ and $t_c(r)\asymp \sqrt{\ln(1/r)}$ if $d=4$ as $r \to 0$.

One may use the scale invariance of Brownian motion to slightly rephrase these results:  let us consider Wiener sausages with radius $1$  that we run up to time $T$ and whose initial points form a PPP on $\mathbb{R}^d$ with intensity $v$. It follows from the above results that the critical intensity for percolation satisfies $v_c(T) \asymp T^{-d/2} $ if $d=2,3$, $v_c(T) \asymp \ln(T)/T^2 $ if $d=4$ and  $v_c(T) \asymp 1/T^2$ if $d \geq 5$
 as $T \to \infty$. Heuristically, one expects the percolation threshold  of the random length worms model (cf.\ \eqref{zoo_v_c_threshold}) to satisfy the same asymptotics if the length  distribution of worms is the Dirac mass concentrated on $T$. Asymptotic results analogous to those of \cite{ErhardPoisat2016} (with a slightly different parametrization) about the percolation threshold of worms with geometric length distribution (i.e., finitary random interlacements, see Section \ref{subsub_FRI}) are proved in \cite{CaiZhang2021b}.

  Let us stress that the above-described dimension-dependence of the asymptotic behaviour of $v_c(T)$
is intimately related to the fact that the capacity (to be defined in Section \ref{section_rw_capacity})
of the trace of a random walk of length $T$ is comparable to $T$ if $d\geq 5$,  it is comparable to $T/\ln(T)$ if $d=4$  and   it is comparable to $\sqrt{T}$ if $d=3$ (see  \cite{capacity_rw_range_d, capacity_rw_range_4} for finer results).

Note that some of our methods (coarse graining, dynamic renormalization, see Section \ref{subsub_powerlaw}) are similar to the method used to construct an infinite connected cluster in \cite[Section 4]{ErhardPoisat2016}.

\subsection{Remarks about our methods}\label{subsection_methods}

In Theorem \ref{main_thm} we claim that if $d \geq 5$ then $\mathcal{S}^v$ percolates for all $v>0$ under the assumption that
$ m(\ell) \asymp \ln(\ln(\ell))^{\varepsilon}/(\ell^3 \ln(\ell))$ for some $\varepsilon>0$. In Sections \ref{subsub_powerlaw} and \ref{subsub_log_corr} we sketch
how percolation of $\mathcal{S}^v$ for all $v >0$ could be proved under stronger assumptions on the decay of the tail of the length distribution of worms (note that by Remark \ref{remark_after_main_thm}\eqref{remark_stoch_dom} the assumption of  Section \ref{subsub_powerlaw} is indeed stronger than that of Section \ref{subsub_log_corr}, which is in turn stronger than that of   Section \ref{subsub_loglog_corr_heu}).
 In Section \ref{subsub_loglog_corr_heu} we
outline the heuristic idea behind the proof of Theorem \ref{main_thm} and in Section \ref{subsub_further_ideas_main_thm} we provide some further details of our proof.

\subsubsection{Sketch proof of percolation if $m(\ell) \asymp \ell^{\varepsilon}/ \ell^{3} $ with $\varepsilon>0$ }\label{subsub_powerlaw}
The proof uses coarse graining and dynamic renormalization (as defined by Grimmett and Marstrand in \cite[Section 3]{GrimmettMarstrand1990}, to be recalled in Section \ref{section_perco_worms_multiscale}) as well as the notion of
capacity (to be recalled in Section \ref{section_rw_capacity}), which can be used to estimate hitting probabilities.

We subdivide $\mathbb{Z}^d$ into disjoint boxes of side-length $R$. It is easy to show that we can find a worm of length greater than $R^2$, i.e., a \emph{long worm}, in the box that contains the origin with high probability. In this case we call this box \emph{good} and we call the trace of this long worm the \emph{seed} of the box that contains it.
Now we try to explore boxes one by one in a way that a newly explored box is a neighbour of an already explored good box with a seed in it. We call such new box \emph{good}
if there exists a long worm emanating from it which hits the seed of the neighbouring good box. We call this part of our construction \emph{target shooting}.

The expectation of the number $N$ of long worms emanating from the new box that hit the seed of the neighbouring good box is comparable to
 $  (v \cdot R^d \cdot \mathbb{P}[\mathcal{L}> R^2 ]) \cdot (R^{2-d} \cdot R^2)$, as we now explain:  $v \cdot R^d \cdot \mathbb{P}[\mathcal{L}> R^2 ]$ is the expected number of long worms emanating from the new box, and $ R^{2-d} \cdot R^2$ is roughly the probability that such a long worm hits the seed of the neighbouring good box. For the latter,  we used that $d \geq 5$ and thus the length and the  capacity of worms are comparable, hence the capacity of a seed is roughly $R^2$ (see the proof of Lemma \ref{lemma:good_shooting} for a similar rigorous calculation).

If $m(\ell) \asymp \ell^{\varepsilon}/ \ell^{3} $ then $ \mathbb{P}[\mathcal{L}> R^2 ] \asymp R^{2\varepsilon}/R^4$, thus $\mathbb{E}(N) \asymp v \cdot R^{2\varepsilon}$, thus for any $v>0$ one can choose $R$ big enough so that
$\mathbb{E}(N)$ is big enough to guarantee $\mathbb{P}(N\geq 1)\geq 1/2$, hence the newly explored box will be good  with probability at least $1/2$. With this uniform lower bound on the probability that the newly explored box is good, one can use dynamic renormalization to show that
the set of explored good boxes percolate,  which implies that $\mathcal{S}^v$ percolates.

\subsubsection{Sketch proof of percolation if $m(\ell) \asymp \ln(\ell)^{\gamma}/(\ell^{3} \ln(\ell) )$ with $\gamma>1/2$ }\label{subsub_log_corr}

We can boost the construction of Section \ref{subsub_powerlaw} by using the worms emanating from the new box that have length at most $R^2$, i.e., \emph{short worms}, to \emph{fatten} the long worm  that hits the seed of  the neighbouring good box and thus creating a new seed with a bigger capacity.

 The expected number of worms of length $\ell$ that hit $o$ is comparable to $v \cdot m(\ell)\cdot \ell$, thus
 the expected total length of short worms that hit $o$ is comparable to $v \cdot \mathbb{E}( \mathcal{L}^2 \cdot \mathds{1}[\mathcal{L} \leq R^2 ] ) \asymp v \cdot \ln(R)^\gamma$. With a bit of work one can show that the total length of short worms that hit a long worm is roughly $R^2 \cdot v \cdot \ln(R)^\gamma $ with high probability. With some further work,
 one  can also show that the capacity of the union of the traces of short worms that hit a long worm (i.e., the new seed) is also roughly $R^2 \cdot v \cdot \ln(R)^\gamma $ with high probability.

  Similarly to Section \ref{subsub_powerlaw}, we want $ \mathbb{E}(N) \asymp (v \cdot R^d \cdot \mathbb{P}[\mathcal{L}> R^2 ]) \cdot (R^{2-d} \cdot R^2 \cdot v \cdot \ln(R)^\gamma )$ to be big. Our assumption $m(\ell) \asymp \ln(\ell)^{\gamma}/(\ell^{3} \ln(\ell) )$ implies  $\mathbb{P}[\mathcal{L}> R^2 ] \asymp \ln(R)^{\gamma-1}/R^4$, therefore have $\mathbb{E}(N) \asymp v^2 \ln(R)^{2\gamma-1} \gg 1 $ if $R \gg 1$, since we assumed $\gamma>1/2$.

Note that the idea of fattening already appeared in the study of the connectivity properties of random interlacements in \cite[Section 4]{RathSapozhnikov2012}.

Also note that the $\gamma=1$ case gives  $m(\ell) \asymp \ell^{-3}$, thus we have $\mathbb{P}[ \mathcal{L} \geq \ell ] \asymp \ell^{-2}$ and our conclusion is that $\mathcal{S}^v$ percolates for all $v>0$. This behaviour is different from that of planar ellipses percolation (cf.\ Section \ref{subsub_ellipses}), which exhibits percolation phase transition in the case when the tail of the size of an animal decays with exponent $\alpha=2$. Thus five dimensional worm percolation gets closer to showing that Lemma \ref{lemma_m_2_finite_no_perc} is sharp (cf.\ Questions \ref{question_strenghten_lemma_m2}, \ref{question_worm_m2}) than planar ellipses percolation. However, we expect planar worm percolation
to perform much worse than planar ellipse percolation from this point of view, see Section \ref{sub_conj_d234}.

\subsubsection{Heuristic proof of percolation if $m(\ell)\asymp \ln(\ln(\ell))^{\varepsilon}/(\ell^3 \ln(\ell))$ with $\varepsilon>0$ }\label{subsub_loglog_corr_heu}

We can boost the construction of Section \ref{subsub_log_corr} if we perform the ``fattening'' a bit more effectively. Let $S_0$ denote the trace of our long worm that we want to fatten and for any $\ell=1,\dots,R^2$ and define $S_\ell$ to be the union of $S_{\ell-1}$ and the trace of the worms of length $\ell$ that hit $S_{\ell-1}$. Now our new seed will be $S_{R^2}$. The
``snowball effect'' caused by recursive fattening will make the capacity $\capacity(S_{R^2})$  bigger than that of the  seed constructed in  Section \ref{subsub_log_corr}.  Heuristically, we have
$\capacity(S_\ell)\approx \capacity(S_{\ell-1})\cdot(1+v \cdot m(\ell)\cdot \ell^2) $, thus
\begin{equation*}
\capacity(S_{R^2}) \asymp R^2 \prod_{\ell=1}^{R^2} \left(1+v \cdot m(\ell)\cdot \ell^2\right) \asymp R^2 \exp\left( v \cdot \sum_{\ell=1}^{R^2} m(\ell)\cdot \ell^2  \right) \asymp
R^2 \exp\left( v \cdot \ln(\ln(R))^{1+\varepsilon} \right).
\end{equation*}
Similarly to Sections \ref{subsub_powerlaw} and \ref{subsub_log_corr}, we want $ \mathbb{E}(N) \asymp (v \cdot R^d \cdot \mathbb{P}[\mathcal{L}> R^2 ]) \cdot (R^{2-d} \cdot \capacity(S_{R^2}) )$ to be big.
Our assumption  $m(\ell)\asymp \ln(\ln(\ell))^{\varepsilon}/(\ell^3 \ln(\ell))$ gives  $\mathbb{P}[\mathcal{L}> R^2 ] \asymp \ln(\ln(R))^{\varepsilon}/( \ln(R) R^4)$, therefore we have $\mathbb{E}(N) \gg v \frac{\exp( v \cdot \ln(\ln(R))^{1+\varepsilon} )}{\ln(R)} \gg 1$ if $R \gg 1$, since we assumed $\varepsilon >0$.

\subsubsection{Further details of the proof of our main result} \label{subsub_further_ideas_main_thm}

In order to build an infinite connected component of $\mathcal{S}^v$, we will subdivide the PPP of worms into disjoint \emph{packages}, where a package contains worms of a certain  length scale that  emanate from a box of a certain spatial scale. Crucially, by the defining properties of Poisson point processes, the contents of different packages will be independent of each other. We will also make sure that
we use every package only once (for the purpose of target shooting or fattening), since it is hard to give a lower bound on the number of useful leftover worms in a package once it has already been used.

In order to make the idea of recursive fattening of Section \ref{subsub_loglog_corr_heu} rigorous, we will use a method that we call \emph{multi-scale recursion}.
Instead of adding the packages of worms of length $\ell=1,\dots,R^2$ one by one (as in Section \ref{subsub_loglog_corr_heu}), we will  define a sequence of scales $(R_n)_{n=0}^{N+1}$ (where $R_{N+1}=R$) and we will fatten with all of the worms of length between
$R^2_{n}$ and $R^2_{n+1}$ in one round. We will choose  $(R_n)_{n=0}^{N+1}$ in a way that one such round of fattening \emph{doubles the capacity} of the set being fattened with high probability. This grouping of worm lengths may seem like a step back in the direction of the fattening method of Section \ref{subsub_log_corr} from the more advanced method of  Section \ref{subsub_loglog_corr_heu}, but  (i) it makes it easier for us to prove that the size of the fattened set is well concentrated around its mean, (ii) it does not weaken the snowball effect by much.

Note that a similar idea (recursive capacity doubling using a sequence of rapidly growing scales followed by target shooting) also appears in the proof of the lower bounds stated in \cite[Theorems 1.5 and 1.6]{ChangSapozhnikov2016} on the one-arm probability of loop percolation in $d=3,4$   (see \cite[Section 6]{ChangSapozhnikov2016}). However, the implementation of our capacity doubling involves an extra layer of difficulty compared to that of \cite{ChangSapozhnikov2016}, as we now explain. \cite[Theorems 1.5 and 1.6]{ChangSapozhnikov2016} are about the subcritical phase of loop percolation (thus one-arm probabilities are necessarily small), while our Theorem \ref{main_thm} is about the supercritical phase of worm percolation (thus we necessarily have to show that connection probabilities are bounded away from zero). Consequently, the proof of \cite[Theorems 1.5 and 1.6]{ChangSapozhnikov2016} can be achieved using a lower bound on the \emph{expected value} of the capacity of the inductively fattened loop cluster
(cf.\ \cite[Lemmas 6.1 and 6.2]{ChangSapozhnikov2016}), while the proof of our main result also involves \emph{concentration estimates} which guarantee that the capacity of the inductively fattened worm cluster is big enough with high enough probability. In order to prove such concentration estimates,  we need to control the correlation between the
amounts of fat produced at distant parts of the set being fattened. In order to effectively bound the sum of such correlations, we want to apply our induction hypothesis on multiple parts of the set being fattened which are well spread-out in space, and thus our capacity doubling scheme involves a fractal pattern which is not present in \cite[Section 6]{ChangSapozhnikov2016}.

\subsection{Conjectures about worm percolation in $d=2,3,4$}\label{sub_conj_d234}

\subsubsection{Conjectures about worm percolation on $\mathbb{Z}^2$}

If $d=2$ then it is reasonable to believe that a worm of length $\ell$  will behave like a ball of radius comparable to $\sqrt{\ell}$. In other words, heuristically,
 worm percolation with length distribution $\mathcal{L}$ will behave like the planar Poisson Boolean model with radius $\mathrm{R}_* \asymp \sqrt{\mathcal{L}}$ (cf.\ Section \ref{subsub_poi_boolean}). Thus we conjecture that  $v_c>0$ holds if and only if $\mathbb{E}[ (\mathrm{R}_*)^2 ] \asymp \mathbb{E}[ \mathcal{L} ]  <+\infty $. However, let us recall the definition of $m_1$ from \eqref{def_eq_moments} and note that if $d=2$ then $m_1 \asymp \mathbb{E}[ \mathcal{L} /\ln( \mathcal{L} ) ] $ holds in the random length worms model by \cite[Theorem 1]{dvoer}, thus one can choose the length distribution of worms in such a way that $m_1<+\infty$ (and therefore $\mathcal{S}^v \neq \mathbb{Z}^d$), but $v_c=0$ (this cannot happen in the case of the Poisson Boolean percolation model).

\subsubsection{Conjectures about worm percolation on $\mathbb{Z}^3$}

Let us assume that $d=3$ and $m(\ell) \asymp \ell^{-\beta}$. We conjecture that $v_c>0$ if and only if $\beta \geq 5/2$:

\begin{itemize}
\item If $\beta > 5/2$ then we replace each worm with the smallest ball that contains it. The radius of the resulting Poisson Boolean model satisfies $\mathrm{R}_* \asymp \sqrt{\mathcal{L}}$, thus  $ \mathbb{E}[ \mathcal{L}^{3/2} ] \asymp \mathbb{E}[ (\mathrm{R}_*)^3 ]   <+\infty $ implies $v_c>0$, where
    $\mathbb{E}[ \mathcal{L}^{3/2} ] =\sum_{\ell=1}^{\infty} \ell^{3/2} m(\ell) <+\infty$ holds if and only if $\beta > 5/2$.
\item If $\beta<5/2$ then one can prove $v_c=0$ using the argument of Section \ref{subsub_powerlaw}, taking into account that if $d=3$
then the capacity of a worm of length $R^2$ is comparable to $R$ (cf.\  \cite[Proposition 1.5]{capacity_rw_range_d}).
Note here that the capacity of a three-dimensional ball of radius $R$ is also comparable to $R$ (cf.\ Proposition \ref{capacity_of_ball}), therefore if $d=3$ then the capacity of a worm is comparable to the capacity of the smallest ball that contains it.
\item The $\beta=5/2$ case heuristically corresponds to loop percolation on $\mathbb{Z}^3$ (cf.\ Section \ref{subsub_loop}), which does exhibit percolation phase transition.
\end{itemize}

Note that if $\beta \leq 2$ then Lemma \ref{lemma_infinite_m1_perc} implies $v_c=0$ and if $\beta > 3$ then Lemma \ref{lemma_m_2_finite_no_perc}  implies $v_c>0$,
so (heuristically) the threshold $\beta=5/2$ for the exponent of $m(\ell) \asymp \ell^{-\beta}$ in $3$-dimensional worm percolation is ``halfway'' between the universal bounds of
 Lemmas \ref{lemma_infinite_m1_perc} and \ref{lemma_m_2_finite_no_perc}.

\subsubsection{Conjectures about worm percolation on $\mathbb{Z}^4$}

 Let us assume that $d=4$ and $m(\ell) \asymp \ell^{-\beta}$. We conjecture that $v_c>0$ if and only if $\beta \geq 3$:

 \begin{itemize}
\item If $\beta > 3$ then $m_2 \leq \mathbb{E}[\mathcal{L}^2] <+\infty$, thus $v_c>0$ by Lemma \ref{lemma_m_2_finite_no_perc}.
\item If $\beta < 3$ then one can prove $v_c=0$ using the argument of Section \ref{subsub_powerlaw}, taking into account that
 if $d=4$ then the capacity of a worm of length $R^2$ is of order $R^2/\ln(R)$.
 \item  The $\beta=3$ case  corresponds to loop percolation on $\mathbb{Z}^4$, for which $v_c>0$ holds.
 \end{itemize}
 Note that if $d=5$ and $\beta=3$ then $v_c=0$ (cf.\ Section \ref{subsub_log_corr}), which shows that worm percolation should behave differently in $d=4$ and $d=5$.
 Also note that the $d=4, \beta=3$ case indicates that the answer to the four dimensional analogue of Question \ref{question_worm_m2} should be negative.

\subsection{Structure of the paper}\label{sub_structure_of_paper}

In  Section \ref{section_basic_notation} we introduce some notation that we will use throughout the paper.

\noindent In Section \ref{section_moment_cond_zoo} we prove  Lemmas \ref{lemma_infinite_m1_perc} and  \ref{lemma_m_2_finite_no_perc} about the Poisson zoo.

\noindent In Section \ref{section_rw_capacity} we collect/prove some useful results about random walks and capacity on $\mathbb{Z}^d$.

\noindent In Section \ref{section_ppp_worms} we set up the notation related to Poisson point processes on the space of worms.

\noindent In Section \ref{section_perco_worms_multiscale} we formally introduce the scales $(R_n)_{n=0}^{N+1}$, state the capacity doubling lemma (Lemma \ref{lemma_doubling_gamma}) that we mentioned in Section \ref{subsub_further_ideas_main_thm} and derive the proof of our main result from it (i.e., we carry out the coarse graining, the target shooting and the dynamic renormalization).

\noindent In Section \ref{subsection_doubling_gamma} we prove the capacity doubling lemma (i.e., we carry out the fattening).

\section{Basic notation}\label{section_basic_notation}

Given $x=(x_1,\dots,x_d)\in \mathbb{Z}^d$, let us denote by $|x|=\max_{1 \leq i \leq d}|x_i|$ the sup-norm of $x$.

Let us use the shorthand $H \fsubset \mathbb{Z}^d$ for a finite subset $H$ of $\mathbb{Z}^d$.

 We use the shorthand $x \sim y$ to denote that $x \in \mathbb{Z}^d$ and $y \in \mathbb{Z}^d$ are nearest neighbours.

If $A,B,S \subseteq \mathbb{Z}^d$ then we say that $A \stackrel{S}{\longleftrightarrow} B$ (in words: $A$ is connected to $B$ through $S$) if there exist $n \in \mathbb{N}_0$ and $x_0,\dots, x_n \in S$ such that $x_0 \in A$, $x_n \in B$ and $x_i \sim x_{i+1}$ for any $0 \leq i \leq n-1$. If $x \in S$, we use the shorthand
$x \stackrel{S}{\longleftrightarrow} B$ to denote  $\{x\} \stackrel{S}{\longleftrightarrow} B$. Note that the relation $\cdot \stackrel{S}{\longleftrightarrow} \cdot $ is an equivalence relation on the set of vertices of $S$, and we call the corresponding equivalence classes the connected components (or clusters) of $S$. We use the shorthand
$x \stackrel{S}{\longleftrightarrow} \infty $ to denote the fact that the cardinality of the connected cluster of $x \in S$ is infinite. We say that $S$ is connected if $S$ contains exactly one connected component.
 We say that $S$ percolates if $S$ has
an infinite connected component.

Let us denote by $\ball(x,R):= \{ y \in \mathbb{Z}^d \, : \, |y-x|\leq R \} $ the ball of radius $R$ about $x$ and use the shorthand $\ball(R) := \ball(o,R)$ to denote a ball centered at the origin.

If $H  \fsubset \mathbb{Z}^d$ and $x \in \mathbb{Z}^d$, let  ${ x + H } = \{ x + y \, : \, y \in H \}$ denote the translate of $H$ by $x$.

If $H \fsubset \mathbb{Z}^d$ then the interior vertex boundary of $H$ is $\partial^{\text{int}} H := \{ x \in H \, : \, \exists \, y \in H^c, \, x\sim y \} $.

If $H \fsubset \Z^d$ then the external vertex boundary of $H$ is $\partial^{\text{ext}} H  := \{ x \in H^c \, : \, \exists \, y \in H, \, x \sim y \}$.

The indicator of the event $A$ is denoted by $\ind[A]$: if $A$ occurs then $\ind[A]=1$, but if $A$ does not occur then $\ind[A]=0$.

For any $a < b \in \mathbb{R}_+$ let $\sum_{\ell=a}^b f(\ell) $ denote $\sum_{\ell=\lceil a \rceil}^{\lfloor b \rfloor} f(\ell) $.

\medskip

 {\bf Notational conventions regarding constants:} We denote  small, but positive constants by $c$ or $c'$ and large, but finite constants by $C$ or $C'$. The constants can only depend on the dimension $d$ of the underlying lattice $\mathbb{Z}^d$. The value of the constants $c$, $c'$, $C$ and $C'$  may change from line to line. On the other hand, we also use numbered constants $c_1, c_2, \dots$ and $C_1, C_2, \dots$ whose value remains fixed for the rest of the paper once they are introduced.

\section{Moment conditions for sub/supercriticality in  Poisson zoo percolation}
\label{section_moment_cond_zoo}

Section \ref{section_moment_cond_zoo} is devoted to the proof of Lemmas \ref{lemma_infinite_m1_perc} and  \ref{lemma_m_2_finite_no_perc}. The ideas of these proofs are folclore, but we include them here for completeness.

Recall the definition of lattice animals and the Poisson zoo from Definitions \ref{def_lattice_animal} and \ref{poisson_zoo}.

\begin{proof}[Proof of Lemma \ref{lemma_infinite_m1_perc}]
	Let us introduce the number
	\begin{equation}
	N := \sum_{x \in \Z^d} \sum_{H \in \mathcal{H}} \sum_{i = 1}^{ N_{x, H}^v } \ind \left[ o \in x + H \right]
	\end{equation}
	 of animals in the Poisson zoo $\widehat{\mathcal{N}}^v$ (counted with multiplicity) whose trace contains the origin. We have
	\begin{align*}
	\mathbb{E} (N)
	=
	\mathbb{E} \left[ \sum_{x \in \Z^d} \sum_{H \in \mathcal{H} } \sum_{i = 1}^{ N_{x, H}^v } \ind \left[ o \in x + H \right]  \right]
	\stackrel{(*)}{=}
	v \cdot \sum_{H \in \mathcal{H}} \nu(H) \cdot \sum_{x \in \Z^d} \ind \left[ o \in x + H \right]
	=
	v \cdot m_1,	
	\end{align*}
	where in $(*)$ we used that $\mathbb{E}(N_{x, H}^v)=v \cdot \nu(H)$. Also note that $N$ is the sum of independent Poisson random variables, thus $N \sim \mathrm{POI}(v \cdot m_1)$. Hence, by the assumptions of the lemma we have $\mathbb{E}(N) = \infty$, which implies $\prob \left( N = \infty \right) = 1$. Thus $\mathbb{P}(o \in \mathcal{S}^v)=1$ and thus
$\mathbb{P}(x \in \mathcal{S}^v)=1$ also holds for any $x \in \mathbb{Z}^d$  by translation invariance.
\end{proof}

The proof of Lemma \ref{lemma_m_2_finite_no_perc} requires further notation.

Given a set $H \in \mathcal{H}$ we will denote its discrete closure by $\overline{H} := H \cup \partial^{\text{ext}}H$.

\begin{definition}[Zoo distance]\label{def:zoo_distance}
Given a Poisson zoo $\widehat{\mathcal{N}}^v$, we say that the \emph{zoo distance} of vertex $y \in \Z^d \setminus \{ o \}$ and the origin $o$ is $k$ if $k \in \N$ is the smallest index such that there exists a sequence of animals $(x_i, H_i) \in \Z^d \times \mathcal{H}$ ($i = 1, \ldots, k$) satisfying
\begin{enumerate}
	\item $N_{x_i, H_i}^{v} \geq 1$ for $i = 1, \ldots, k$;
	\item $o \in x_1 + H_1$ and $y \in x_k + H_k$;
	\item $(x_i + \overline{H}_i) \cap (x_{i+1} + H_{i+1}) \neq \emptyset$ for $i = 1, \ldots, k-1$.
\end{enumerate}
\end{definition}
Observe that since we take $k$ to be the smallest of such indices, it follows that if the zoo distance of $o$  and $y$ is $k$  and  $\left\{ (x_i, H_i) \right\}_{i = 1}^{k}$ is a sequence of animals satisfying Definition \ref{def:zoo_distance}, then
\begin{equation}\label{distinct_animals}
(x_i, H_i) \neq (x_j, H_j), \qquad i \neq j \in \{ 1, \ldots, k \}.
\end{equation}

\begin{definition}\label{def_dist_k_C_k_v}  Let $k \in \mathbb{N}$.
Let us denote by $C_k^v$ the set of vertices  whose zoo distance from the origin is equal to $k$.
\end{definition}

 Given some $x \in \mathbb{Z}^d$, let us denote by \[\mathcal{C}_x^v:= \{\, y \in \mathbb{Z}^d \,: \, x \stackrel{ \mathcal{S}^v }{\longleftrightarrow} y \, \}\]
 the connected cluster of $x$ with respect to $\mathcal{S}^v$. With the above definitions we have
\begin{equation}
\label{nota:cluster_of_the_origin}
\mathcal{C}_o^v := \bigcup_{k = 1}^{\infty} C_k^v.
\end{equation}

\begin{proof}[Proof of Lemma \ref{lemma_m_2_finite_no_perc}] First we show that if $v < 1 \slash ((2d+1) \cdot m_2)$ then
	 $\mathbb{E} [ \, | \mathcal{C}_o^v | \, ] < \infty$. We have
	\begin{align}
	& \mathbb{E} \left[ \, | C_k^v | \, \right]
	=
	\sum_{z \in \Z^d} \prob \left( z \in C_k^v \right)
	\stackrel{(\bullet)}{\leq}
	\sum_{z \in \Z^d} \sum_{x_1, \ldots, x_k \in \Z^d} \sum_{H_1, \ldots, H_k \in \mathcal{H}} \ind \left[ o \in x_1 + H_1 \right]
	\cdot \notag \\ & \quad \cdot
	\left( \prod_{i = 1}^{k - 1} \ind \left[ (x_i + \overline{H}_i) \cap (x_{i+1} + H_{i+1}) \neq \emptyset \right] \right) \cdot
	\ind \left[ z \in x_k + H_k \right]
	\cdot
	\prod_{i = 1}^{k} \left( v \cdot \nu(H_i) \right)
	\notag \\ & \stackrel{(*)}{\leq}
	\sum_{z \in \Z^d} \sum_{x_1, \ldots, x_k \in \Z^d} \sum_{H_1, \ldots, H_k \in \mathcal{H}} \sum_{y_1, \ldots, y_{k-1} \in \Z^d}
	\ind \left[ o \in x_1 + H_1 \right] \cdot
	\ind \left[ z \in x_k + \overline{H}_k \right]
	\cdot \notag \\ &
	\label{nota:M_k_sums}
	\quad \cdot
	\left( \prod_{i = 1}^{k - 1} \ind \left[ y_i \in x_i + \overline{H}_i \left] \cdot \ind \right[ y_i \in x_{i+1} + H_{i+1} \right] \right)
	\cdot
	\prod_{i = 1}^{k} \left( v \cdot \nu(H_i) \right) =: M_k.
	\end{align}
Let us now explain the inequalities $(\bullet)$ and $(*)$.
	In $(\bullet)$ we used  the union bound and
\begin{equation*}
\mathbb{P}\left(\,  N^v_{x_i,H_i} \geq 1, \, i=1,\dots,k \, \right)\stackrel{(\diamond)}{=}
\prod_{i = 1}^{k} \mathbb{P} \left( \, N^v_{x_i,H_i} \geq 1 \, \right)= \prod_{i = 1}^{k} \left(1-e^{-v\cdot \nu(H_i)} \right)
\leq \prod_{i = 1}^{k} \left( v \cdot \nu(H_i) \right),
\end{equation*}
where in $(\diamond)$ we used \eqref{distinct_animals}
  and the independence of the random variables $N^v_{x,H}, x \in \mathbb{Z}^d, H \in \mathcal{H}$.

In  $(*)$ we used that for any $i=1,\dots,k-1$ we have
\[  \mathds{1} \left[ (x_i + \overline{H}_i) \cap (x_{i+1} + H_{i+1}) \neq \emptyset \right] \leq
\sum_{y_i \in \mathbb{Z}^d} \mathds{1}\left[ y_i \in x_i + \overline{H}_i \right] \cdot \mathds{1} \left[ y_i \in x_{i+1} + H_{i+1} \right]
 \]
and also replaced $H_k$ with  $\overline{H}_k$: this will make the recursive formula \eqref{obs:M_k=Lambda_M_k-1}  simpler.
	
	Let us also introduce
	\begin{equation*}
	\Lambda := v \cdot \sum_{H \in \mathcal{H}} |H| \cdot | \overline{H} | \cdot \nu(H).
	\end{equation*}
	Since $| \overline{H} | \leq (2d+1) \cdot | H|$ holds for any $H \in \mathcal{H}$, we  have $\Lambda \leq (2d+1) \cdot v \cdot m_2$. An appropriate rearrangement of the sums that appear in \eqref{nota:M_k_sums} gives
	\begin{equation}
	\label{obs:M_k=Lambda_M_k-1}
	M_k = \Lambda \cdot M_{k-1}, \qquad k \in \N,
	\end{equation}
	 where we defined $M_0 := 1$.
		As a consequence, for any $0 < v < 1 \slash ((2d+1) \cdot m_2)$ we have
	\begin{equation*}
	\mathbb{E} \left[ \, | \mathcal{C}_o^v | \, \right]
	\stackrel{ \eqref{nota:cluster_of_the_origin} }{ \leq }
	\sum_{k = 1}^{\infty} \mathbb{E} \left[ \, | C_k^v | \, \right]
	\stackrel{ \eqref{nota:M_k_sums} }{ \leq }
	\sum_{k=1}^{\infty} M_k
	\stackrel{ \eqref{obs:M_k=Lambda_M_k-1} }{ = }
	\sum_{k = 1}^{\infty} \Lambda^k
	\leq
	\sum_{k=1}^{\infty} \left( (2d+1) \cdot v \cdot m_2 \right)^k < \infty.
	\end{equation*}
	 This implies $\mathbb{P}( o \stackrel{ \mathcal{S}^v }{\longleftrightarrow} \infty )=0$, and by translation invariance we also have $\mathbb{P}( x \stackrel{ \mathcal{S}^v }{\longleftrightarrow} \infty )=0$ for any $x \in \mathbb{Z}^d$. We conclude that if $0 < v < 1 \slash ((2d+1) \cdot m_2)$ then $\mathcal{S}^v$ almost surely does not percolate, finishing the proof of Lemma \ref{lemma_m_2_finite_no_perc}.
\end{proof}

\begin{remark} In this paper we defined the Poisson zoo model on the $d$-dimensional nearest neighbour lattice graph $\mathbb{Z}^d$, but the definition of the model can be naturally generalized to the case when the underlying graph is the Cayley graph  corresponding to a finite and symmetric generating set of a countable group. The statements and proofs of Lemmas \ref{lemma_infinite_m1_perc} and \ref{lemma_m_2_finite_no_perc} generalize to the Poisson zoo model defined on Cayley graphs (with obvious notational adjustments).
\end{remark}

\section{Preliminary results about random walks and discrete potential theory}\label{section_rw_capacity}
In Section \ref{section_rw_capacity} we collect/prove some facts about random walks on $\mathbb{Z}^d$. In later sections these intermediate results will serve as ingredients of the proof of our main result.

In Section \ref{subsection_random_walks} we introduce some notation (e.g.\ hitting times) and state some basic inequalities about the transition probabilities and Green function of simple random walk on $\mathbb{Z}^d$.

In Section \ref{subsection_capacity} we introduce the notions of discrete potential theory associated to simple random walk on $\mathbb{Z}^d$: the central quantity is called the capacity and the central question is how to bound it from above/below.
In Section \ref{subsection_pt_rw} we state some facts that will help us bound the capacity of worms (i.e., the trace of random walk trajectories of finite length).

In Section \ref{subsection_rw_estimates} we consider a subset $K$ of a box of scale $r$ and prove some estimates about  the expected number of worms of length $\ell$ emanating from a box of larger scale $R$ that hit $K$.

In Section \ref{subsection_rw_sub_boxes} we consider a random walk in a box of scale $R$ and prove some estimates about the number and location of
sub-boxes of scale $r$ visited by this random walk.

\subsection{Basic notation and results about random walks}\label{subsection_random_walks}

In Section \ref{subsection_random_walks} we collect some facts about random walks on  $\mathbb{Z}^d$. We will explicitly write down our assumptions regarding the value of the dimension  $d$ in the statement of each result.

Let $W$ be the space of $\Z^d$-valued infinite nearest-neighbour trajectories, indexed by non-negative times.
This space is endowed with the $\sigma$-algebra $\mathcal{W}$ generated by the canonical coordinate maps.
	
For $x \in \Z^d$, let $P_x$ denote the law of simple random walk $(X(t))_{t \in \N_0}$ on $\Z^d$, which starts from $X(0)=x$, and let $E_x$ be the corresponding expectation. The law $P_x$ will be considered as a probability measure on $(W, \mathcal{W})$, hence a simple random walk $(X(t))_{t \geq 0}$ is a random element of $W$.

Similarly, for $x, x' \in \Z^d$, let $P_{x,x'}$ denote the joint law of two independent nearest neighbour random walks $(X(t))_{t \in \N_0}$ and $(X'(t))_{t \in \N_0}$ starting from $x$ and $x'$, respectively. Let $E_{x,x'}$ denote the corresponding expectation.

There are some commonly used random times which we will need in what follows. For any $w \in W$ and $K \subseteq \Z^d$, let
\begin{align}
	\label{entrance_time}
	T_K & := T_K(w) := \inf \{\, t \geq 0 \, : \,  w(t) \in K\, \}, \qquad \text{the entrance time of } K; \\
	\label{hitting_time}
	\tilde{T}_K & := \tilde{T}_K(w) := \inf \{\, t \geq 1 \, : \, w(t) \in K \, \}, \qquad \text{the hitting time of } K; \\
	\label{last_visit}
	L_K & := L_K(w) := \sup \{ \, t \geq 0 \, : \, w(t) \in K \, \}, \qquad \text{the time of the last visit to } K,
\end{align}
where $\inf \emptyset = \infty$ and $\sup \emptyset = - \infty$.
	
Let us denote the \emph{transition probabilities} of the simple random walk on $\Z^d$ by
\begin{equation}\label{rw_transition_probab}
  p_t(x,y):=P_x(X(t) = y), \qquad x,y \in \mathbb{Z}^d, \; t \in \mathbb{N}_0.
\end{equation}
It is classical (see \cite[Theorem 2.1.1]{LawlerLimic2010}) that there exist constants $c > 0$ and $C<+\infty$ such that
\begin{align}
	\label{dumb_heat_kernel_bound}
	p_t(x,y) & \leq C \cdot t^{-d/2}, \qquad t \in \N_0, \, x,y \in \mathbb{Z}^d, \\
	\label{heat_kernel_lower_bound}
	p_{t-1}(x,y) + p_{t}(x,y) & \geq c \cdot t^{-d \slash 2}, \qquad t \in \N, \, x,y \in \Z^d, \, |x-y| \leq \sqrt{t}.
\end{align}

We define the \emph{Green function} $g \, : \, \Z^d \times \Z^d \rightarrow [0, \infty)$ of the  random walk by
\begin{equation}\label{green}
	g(x,y) := \sum_{t = 0}^{\infty} p_t(x,y), \qquad x,y \in \mathbb{Z}^d.
\end{equation}
Note that $d \geq 3$ implies that $g(x,y)<+\infty$ for any $x,y \in \mathbb{Z}^d$.  We have $g(x,y) = g(y,x)$ and $g(x,y) = g(o, y-x)$ by symmetry and translation invariance. Moreover, it follows from \cite[Theorem 4.3.1]{LawlerLimic2010} that for any $d \geq 3$ there exist constants $c > 0$ and $C < \infty$ such that
\begin{equation}
	\label{green_bounds}
	c \cdot (|x-y| + 1)^{2 - d} \leq g(x,y) \leq C \cdot (|x-y| + 1)^{2 - d}, \qquad x, y \in \Z^d.
\end{equation}

The proof of the next result follows from the Azuma-Hoeffding inequality.
\begin{proposition}[Random walk does not travel too far] Let $d \geq 1$.
	\label{staying_inside}
	 For any $R, L \in \mathbb{N}$ we have
	\begin{align}
	\label{eq:staying_inside}
	P_o \big(\, \max_{0 \leq t \leq L} | X(t) | \geq R \, \big) \leq  2d \cdot \exp \{ -  R^2/2L \}.
	\end{align}
\end{proposition}

\subsection{Discrete potential theory}\label{subsection_capacity}

  Let us now recall some elements of discrete potential theory. In particular, we define the capacity $\capacity(K)$ of a finite subset $K$ of $\mathbb{Z}^d$, then recall a variational characterization of $\capacity(K)$
   and relate $\capacity(K)$ to the probability that a simple random walk hits $K$.
     In Section \ref{subsection_capacity} we assume $d \geq 3$.

\begin{definition}[Dirichlet energy]
	\label{def:energy}
	If $\mu_1$ and $\mu_2$ are measures on $\mathbb{Z}^d$ then their \emph{mutual Dirichlet energy} is
	\begin{equation}\label{def_eq_energy}
		\energy(\mu_1, \mu_2) := \sum_{x \in \mathbb{Z}^d} \sum_{y \in \mathbb{Z}^d} g(x,y) \mu_1(x) \mu_2(y),
	\end{equation}
	where $g(\cdot,\cdot)$ is the Green function (cf.\ \eqref{green}). When $\mu_1=\mu_2=\mu$ then $\energy(\mu) := \energy(\mu, \mu)$ is called the \emph{Dirichlet energy} of $\mu$.
\end{definition}
By bilinearity, for any measure $\mu$ and constant $a > 0$, the energy of the measure $a \cdot \mu$ is
\begin{equation}\label{dir_bil}
\energy( a \cdot \mu ) = a^2 \cdot \energy(\mu).
\end{equation}

The \emph{capacity} of the set $K \fsubset \Z^d$ is defined to be
\begin{equation}
	\label{capacity_with_energy}
	\capacity(K) =  \sup \big\{\, \energy(\nu)^{-1} \, : \, \nu \text{ is a probability measure supported on } K \, \big\}.
\end{equation}
We define the capacity of the empty set to be $0$. In order to give a lower bound on $\capacity(K)$, one just puts a probability measure $\nu$ on $K$ and gives an upper bound on  $\energy(\nu)$.

 Using \eqref{capacity_with_energy}, one obtains that the capacity is monotone
\begin{equation}\label{cap_mon}
	 \capacity(K_1) \leq \capacity(K_2) \qquad \qquad \forall \, K_1 \subseteq K_2 \fsubset \Z^d,
\end{equation}
and for any vertex $x \in \Z^d$ we have
\begin{equation}
	\label{obs:capacity_of_one_point}
	\capacity(\{ x \}) \stackrel{ \eqref{capacity_with_energy} }{=} 1 \slash g(o,o) \stackrel{\eqref{green} }{\leq} 1 .
\end{equation}

 Given $K \fsubset \Z^d$,  the equilibrium measure $e_K$ of $K$ is defined by
\begin{equation}\label{equilib_measure}
e_K(x) := P_x ( \tilde{T}_K = \infty ) \cdot \ind [ x \in K ], \qquad x \in \mathbb{Z}^d.
\end{equation}
It follows from \cite[Lemma 2.3]{jain} that the capacity of $K \fsubset \Z^d$ is equal to the total mass of the equilibrium measure of $K$, i.e.
\begin{equation}
	\label{capacity_with_eqmeasure}
	\capacity(K) = \sum_{x \in \Z^d} e_K(x).
\end{equation}
This characterization yields that the capacity is subadditive, i.e.,
\begin{equation}\label{cap_subadditive}
	\capacity(K_1 \cup K_2) \leq \capacity(K_1) + \capacity(K_2), \qquad K_1, \, K_2 \fsubset \Z^d.
\end{equation}
\begin{lemma}[Trimming the fat]\label{lemma_trimming_the_fat} For any finite connected subset $K$ of $\mathbb{Z}^d$, any $x \in K$ and any $0 \leq a \leq \capacity(K)$ there exists a connected subset
$K'$ of $K$  satisfying $x \in K'$ and
\begin{equation}
  a \leq \capacity(K') \leq a+1.
\end{equation}
\end{lemma}
\begin{proof} One may order the vertices $x_1,\dots,x_n$ of $K$ in a way that $x_1=x$ and if we define $K_i:=\{x_1,\dots,x_i\}$ for any $1\leq i \leq n$ then the set $K_i$ is connected. We have $\capacity(K_1) \leq 1$ by \eqref{obs:capacity_of_one_point}, moreover  it follows from \eqref{cap_mon}, \eqref{obs:capacity_of_one_point} and \eqref{cap_subadditive} that   $\capacity(K_{i}) \leq \capacity(K_{i+1}) \leq \capacity(K_{i})+1 $ holds for any $1 \leq i \leq n-1$, thus there exists $1 \leq i \leq n$ such that $a \leq \capacity(K_i) \leq a+1$.
\end{proof}

Our next proposition follows from \cite[Lemma 1.12]{drs}.
\begin{proposition}[Capacity and hitting probability]
	\label{prop:LED}
	For any $x \in \Z^d$ and $K \fsubset \Z^d$
	\begin{equation}
	\label{LED}
	\capacity(K) \cdot \min_{y \in K} g(x,y) \leq P_x (T_K < \infty) \leq \capacity(K) \cdot \max_{y \in K} g(x,y).
	\end{equation}
\end{proposition}

\begin{lemma}[Hitting a set within a given time frame]
	\label{lemma:two_birds_one_stone}
	There exists $c > 0$ such that for any $z \in \Z^d$ and $K \fsubset \Z^d$ we have
	\begin{equation}
		\label{eq:lemma:two_birds_one_stone}
		P_z \left( T_K \leq 2 \cdot D^2 \right) \geq c \cdot \capacity(K) \cdot (D + 1)^{ 2 - d }, \quad \text{where} \quad
D = \max_{x \in K} |x - z|.
	\end{equation}
\end{lemma}

\begin{proof}
	Let us first note that   the result obviously holds when $K = \{ z \}$ since $\capacity(\{z\})>0$ by \eqref{obs:capacity_of_one_point}, hence we can focus on the case when $D > 0$. In this case \eqref{eq:lemma:two_birds_one_stone} can be proved as follows:
	\begin{align*}
		& P_z \left( T_K \leq 2 \cdot D^2 \right)
		\stackrel{ \eqref{last_visit} }{\geq}
		P_z \left( L_K \leq 2 \cdot D^2 \right)
		\stackrel{(*)}{=}
		\sum_{t = 0}^{ 2 D^2 } \sum_{x \in K} p_t(z,x) \cdot e_K(x) \geq
		\\ & \,
		\sum_{x \in K} e_K(x)  \sum_{s = \lceil D^2 \slash 2 \rceil }^{ D^2 } \left[ p_{2s-1}( z, x ) + p_{2s}(z,x) \right]
		\stackrel{ (**)}{ \geq }
		\capacity(K) \sum_{s = \lceil D^2 \slash 2 \rceil}^{D^2} c \cdot s^{- d \slash 2}
		\geq c \cdot \capacity(K) \cdot D^{2 - d}.
	\end{align*}
	Here, in $(*)$ we used \eqref{rw_transition_probab}, \eqref{equilib_measure} and the Markov property at $t$, moreover in $(**)$ we used \eqref{capacity_with_eqmeasure} and \eqref{heat_kernel_lower_bound}, noting that by the definition of $D$, we have $|z - x| \leq D \leq \sqrt{2s}$ for all $x \in K$.
\end{proof}

The next proposition follows from \cite[Proposition 6.5.2]{LawlerLimic2010}.
\begin{proposition}[Capacity of a ball]
	\label{capacity_of_ball}
	There exist constants $c > 0$ and $C < \infty$ such that
	\begin{equation}
	\label{eq:capacity_of_ball}
	c \cdot R^{d - 2} \leq \capacity( \ball(R) ) \leq C \cdot R^{d - 2}, \qquad R >0.
	\end{equation}
\end{proposition}

\subsubsection{Discrete potential theory on the trace of a random walk}\label{subsection_pt_rw}

In Section \ref{subsection_pt_rw} we recall/prove  results that allow us to estimate the capacity of the trace of a random walk
 of fixed finite length.
 Our next proposition follows from \cite[Theorem 2]{jain_orey_68}.
\begin{proposition}[Law of large numbers for the capacity of the range of random walk]\label{prop_lln_cap_rw}
Let $d \geq 5$. There exists a constant $e_\infty=e_\infty(d)\in(0,+\infty)$ such that
\begin{equation}
\lim_{n \to \infty} \frac{1}{n} \capacity \left( \cup_{t=0}^{n-1} \{ X(t) \} \right) = e_\infty \qquad P_o - \text{almost surely.}
\end{equation}
\end{proposition}

We will estimate the capacity of (the union of the traces of some) worms using the Dirichlet energy  of the so-called
local time measure of a worm: the  measure of a vertex $x \in \mathbb{Z}^d$ is equal to the number of visits to $x$ by the worm.
Our next result follows from \cite[Lemma 3]{RathSapozhnikov2012}.
\begin{proposition}[Expected  Dirichlet energy of the local time measure of a random walk]
	\label{upper:exp_energy_of_traj}	
	If $d \geq 5$ then there exists $C < \infty$ such that for every $n \in \N$ and $x \in \Z^d$ we have
	\begin{equation}
		\label{eq:upper:exp_energy_of_traj}
		E_x \left[ \sum_{t = 0}^{n} \sum_{t' = 0}^{n} g(X(t), X(t')) \right] \leq C \cdot n.
	\end{equation}
\end{proposition}

Recall the notation $E_{x,x'}$ from Section \ref{subsection_random_walks}.

\begin{lemma}[Expected mutual Dirichlet energy of the local times of two random walks]
	\label{upper:exp_mutual_energy_of_trajs} If $d \geq 3$ then
 	there exists $C < \infty$ such that for every $n, \, n' \in \N$ and $x, x' \in \Z^d$ we have
	\begin{equation}
		\label{eq:upper:exp_mutual_energy_of_trajs}
			E_{x,x'} \left[	\sum_{t = n }^{2 n} \sum_{t' = n' }^{2  n'}  g( X(t), X'(t') ) \right]
			\leq
C \cdot n \cdot n' \cdot \min  \left\{ |x - x'|^{2-d}  , \,    (n + n')^{1-d/2}     \right\} .
	\end{equation}
\end{lemma}

\begin{proof}
	By the linearity of expectation, we can rewrite the l.h.s.\ of \eqref{eq:upper:exp_mutual_energy_of_trajs} as
	\begin{equation}\label{energy_lin_expect_green}
		 \sum_{t = n}^{ 2 n} \sum_{t' = n'}^{ 2 n' } E_{x,x'}\left[ g( X(t), X'(t') ) \right].
	\end{equation}
	We want to bound the generic term of the above sum.	Let us observe that
	\begin{equation}\label{egx1t1x2t2}
		E_{x,x'} \left[ g(X(t), X'(t') ) \right]
		\stackrel{ \eqref{green} }{ = }
		\sum_{s=0}^{\infty} E_{x,x'} \left( p_s(X(t),X'(t') )  \right)
		\stackrel{(*)}{=}
		\sum_{s=0}^{\infty} p_{s+t+t'}(x, x'),
	\end{equation}
	where $(*)$ follows from the Chapman-Kolmogorov equations for the transition kernel \eqref{rw_transition_probab}.		

	We will bound the r.h.s.\ of \eqref{egx1t1x2t2} in two different ways:
	\begin{align*}
		& \sum_{s=0}^{\infty} p_{s+t+t'}(x, x') \stackrel{\eqref{green} }{\leq} g(x, x') \stackrel{ \eqref{green_bounds} }{\leq} C \cdot | x - x' |^{ 2 - d },\\
		& \sum_{s=0}^{\infty} p_{s+t+t'}(x, x')  \stackrel{\eqref{dumb_heat_kernel_bound} }{\leq } C \cdot \sum_{s=0}^{\infty} (s+t+t')^{-d/2} \leq C \cdot (t +t')^{1-d/2}
		\stackrel{(**)}{\leq} C \cdot   (n+n')^{1-d/2},
	\end{align*}
	where in $(**)$ we used that $t \geq n$ and $t' \geq n'$ holds for the terms of the sum \eqref{energy_lin_expect_green}. Putting together the above bounds, we obtain \eqref{eq:upper:exp_mutual_energy_of_trajs}.
\end{proof}

\subsection{Sums of random walk hitting probabilities}\label{subsection_rw_estimates}

In Section \ref{subsection_rw_estimates} we consider a subset $K$ of a box of scale $r$. We state and prove some results that will allow us to estimate
 the expected number of worms of length $\ell$ that hit $K$: Lemma \ref{upper:fuga_hitting_sum} gives an upper bound, while Lemma \ref{lower:fuga_hitting_sum}
provides us with a matching lower bound. Lemma \ref{lemma_upper:two_cons_hitting_sum} gives an upper bound on the expected number of worms of length $\ell$ that hit the sets $K_1$ and $K_2$.
 The reader may want to skip the proofs of these technical results at first reading.

\begin{lemma}[Upper bound on the sum of hitting probabilities]
	\label{upper:fuga_hitting_sum} Let $d\geq 3$.
	There exists $C < \infty$ such that, for any $r \geq 1$, $\ell \geq 16 \cdot r^2$ and $K \subseteq \ball(3  r)$ we have
	\begin{equation}
		\label{eq:upper:fuga_hitting_sum}
		\sum_{z \in \Z^d \, : \, | z | \geq 4  r} P_z \big( T_K \leq \ell \big) \leq C \cdot \ell \cdot \capacity(K).
	\end{equation}
\end{lemma}

\begin{proof}
	By the assumption $\ell \geq 16 \cdot r^2$ we can divide the sum into disjoint terms
	\begin{equation}\label{group_terms_in_ketto}
		\sum_{z \in \Z^d \, : \, |z| \geq 4 r} P_z( T_K \leq \ell ) =
		\sum_{4 r \leq |z| \leq \sqrt{\ell}} P_z( T_K \leq \ell ) +
		\sum_{\sqrt{\ell} < |z|} P_z( T_K \leq \ell ).
	\end{equation}
	The first term on the r.h.s.\ of \eqref{group_terms_in_ketto}  can be bounded as
	\begin{align}
		\sum_{4r \leq |z| \leq \sqrt{\ell}} P_z(T_K \leq \ell) & \leq
		\sum_{t = 4r}^{ \sqrt{\ell} } \sum_{|z| = t} P_z(T_K < \infty) \stackrel{ \eqref{LED} }{ \leq }
		\sum_{t = 4r}^{ \sqrt{\ell} } \sum_{|z| = t} \capacity(K) \cdot \max_{y \in K} g(z,y) \\ & \stackrel{(*)}{ \leq }
		C \cdot \capacity(K) \cdot \sum_{t = 4r}^{\sqrt{\ell} } t^{d-1} \cdot t^{2 - d} \leq
		C \cdot \capacity(K) \cdot \ell,
	\end{align}
	where in $(*)$ we used $\eqref{green_bounds}$, noting that $|z-y| \geq t - 3r \geq t/4 $ holds since $K \subseteq \ball(3r)$.
	
	For the second term on the r.h.s.\ of \eqref{group_terms_in_ketto}, observe that if $z \in \Z^d$ satisfies $|z| = t > \sqrt{\ell}$ then
	\begin{equation}
		\label{hitting_sum_Hoeffding}
		P_z \big( T_{\ball( \sqrt{\ell} )} \leq \ell \big) \leq
		P_o \left( \max_{s = 1, \ldots, \ell}  |X(s)|  \ \geq t - \sqrt{\ell} \right) \stackrel{ \eqref{eq:staying_inside} }{ \leq }
		C \cdot \exp \left\{ - \frac{(t - \sqrt{\ell})^2}{ 2 \ell } \right\},
	\end{equation}
	 Now the upper bound on the second term on the r.h.s.\ of \eqref{group_terms_in_ketto} follows using a similar argument as before:
	\begin{multline}
		\sum_{|z| > \sqrt{\ell}} P_z( T_K \leq \ell )  \stackrel{(\diamond)}{ \leq }
		\sum_{t = \sqrt{\ell}}^{\infty} \sum_{|z| = t}
 P_z \big( T_{\ball(\sqrt{\ell})} \leq \ell \big) \cdot \max_{y \in \partial^{\text{int}} \ball( \sqrt{\ell} )} P_y (T_K < \infty) \\
		\stackrel{ \eqref{hitting_sum_Hoeffding}, \, \eqref{LED} }{ \leq }
		C \cdot \sum_{t = \sqrt{\ell}}^{\infty} \sum_{|z| = t} \exp \left\{ - \tfrac{( t - \sqrt{\ell} )^2 }{ 2 \ell } \right\} \cdot  \capacity(K) \cdot \max_{x \in K, \, y \in \partial^{\text{int}} \ball(\sqrt{\ell})} g(y,x)
  \\
		\stackrel{ (\diamond \diamond) }{ \leq }
		C \cdot \capacity(K) \cdot \ell^{1 - d \slash 2}  \sum_{t = \sqrt{\ell}}^{\infty} t^{d-1}  \exp \left\{ - \tfrac{(t - \sqrt{\ell})^2}{ 2 \ell } \right\}  \leq		C \cdot \capacity(K) \cdot \ell^{ 1 - d \slash 2 }  \int_{\sqrt{\ell}}^{\infty} t^{d-1}  \exp \left\{ - \tfrac{(t - \sqrt{\ell})^2}{  2\ell } \right\} \, \mathrm{d}t\\  \stackrel{ (\bullet) }{ = }
		C \cdot \capacity(K) \cdot \ell^{1 - d \slash 2}  \int_{1}^{ \infty } \ell^{ (d-1)\slash 2 } \cdot x^{ d - 1 } \cdot \exp \left\{ -\tfrac{1}{2} (x-1)^2  \right\} \cdot \sqrt{\ell} \, dx  \leq
		C \cdot \capacity(K) \cdot \ell,
	\end{multline}
	where in $(\diamond)$ we used the strong Markov property at $T_{\ball(\sqrt{\ell})}$, in $(\diamond \diamond)$ we used \eqref{green_bounds} together with the inequalities  $\sqrt{\ell} \geq 4 r$ and $|y-x| \geq \sqrt{\ell} - 3r \geq \sqrt{\ell}/4$,  moreover in $(\bullet)$ we made the variable change $x = t \slash \sqrt{\ell}$. Putting all of these together we obtain the desired upper bound \eqref{eq:upper:fuga_hitting_sum}.
\end{proof}

\begin{lemma}[Upper bound on the sum of double hitting probabilities]
	\label{lemma_upper:two_cons_hitting_sum} Let $d \geq 3$.
	There exists $C < \infty$ such that for any $r \geq 1$, $x \in \Z^d$ satisfying $|x| \geq 10  r$, $K_1 \subseteq \ball(3 \cdot r)$ and $K_2 \subseteq \ball(x, 3 \cdot r)$ and for any $\ell \geq 16 \cdot r^2$ we have
	\begin{align}
		\label{eq:upper:two_cons_hitting_sum}
		\sum_{z \in \mathcal{C}} P_z \big( \max \{ T_{K_1}, T_{K_2} \} \leq \ell \big) \leq C \cdot \ell \cdot \capacity(K_1) \cdot \capacity(K_2) \cdot |x|^{2 - d},
	\end{align}
	where $\mathcal{C} = \{ z \in \Z^d \, : \, |z| \geq 4 \cdot r, \, |z - x| \geq 4 \cdot r \}$.	
\end{lemma}

\begin{proof}
	We can assume that our walker visits the set $K_1$ first, since by the layout's symmetry the same proof gives the same upper bound in the other case. However, this can be done as
	\begin{multline}
		\sum_{z \in \mathcal{C}} P_z \left( T_{K_1} < T_{K_2} \leq \ell \right)
		 \stackrel{(*)}{\leq}
		\sum_{z \in \Z^d \, : \, |z| \geq 4r} P_z \left( T_{K_1} \leq \ell \right) \cdot \max_{y \in K_1} P_y \left( T_{K_2} < \infty \right)
		  \stackrel{ (**) }{ \leq } \\
		C \cdot \sum_{z \in \Z^d \, : \, |z| \geq 4r} P_z \left( T_{K_1} \leq \ell \right) \cdot |x|^{2-d} \cdot \capacity(K_2)
		 \stackrel{ \eqref{eq:upper:fuga_hitting_sum} }{ \leq }
		C \cdot \ell \cdot \capacity(K_1) \cdot |x|^{2-d} \cdot \capacity(K_2),
	\end{multline}
	where in $(*)$ we used the strong Markov property at $T_{K_1}$ and $(**)$ follows from \eqref{LED} and \eqref{green_bounds} since the distance between $K_1$ and $K_2$ is at least $|x|-6r \geq (4/10)\cdot |x| $, using $|x| \geq 10  r$.
\end{proof}

Our next result shows that the bound of Lemma \ref{upper:fuga_hitting_sum} is sharp (compare \eqref{eq:upper:fuga_hitting_sum} and \eqref{eq:lower:fuga_hitting_sum}).
For technical reasons (that will only be clarified in Section \ref{subsection_doubling_gamma}) we need to restrict our attention to random walks of length $\ell$ which (i) start from a certain subset $\mathcal{G}$ of $\ball(2R)$ which has holes of scale $r$ in it (to be defined below), (ii) hit the set $K$ (which is still a subset of a box of scale $r$) in the first $ \ell/3 $ steps and (iii) never exit a bigger ball $\ball(3R)$. In order for the
lower bound \eqref{eq:lower:fuga_hitting_sum} to hold, we require $\ell/r^2$ to be big (so that we can apply Lemma \ref{lemma:two_birds_one_stone}) and we require $\ell/R^2$ to be small (since we want the walker to stay inside $\ball(3R)$).

\begin{lemma}[Lower bound on the sum of hitting probabilities]
	\label{lower:fuga_hitting_sum}
	Let $d \geq 3$. There exist $1< \CLowerFugaHitDELTALOWER < \infty$, $0 < \CLowerFugaHitDELTAUPPER < 1$ and $\CLowerFugaHitLOWER > 0$ such that for any $2 \leq r, \ell, R \in \N$ satisfying
	\begin{equation}\label{ell_sandwich}
		\CLowerFugaHitDELTALOWER \cdot r^2 \leq \ell \leq \CLowerFugaHitDELTAUPPER \cdot  R^2,
	\end{equation}
	any $y \in 10 r \Z^d$ such that $\ball(y, r) \cap \ball(2R) \neq \emptyset$ and any $K \subseteq \ball(y, 3 \cdot r)$ we have
	\begin{equation}
		\label{eq:lower:fuga_hitting_sum}
		\sum_{z \in \mathcal{G}} P_z \big(\, T_K \leq  \ell/3 , \,  T_{\ball(3R)^c} > \ell   \, \big) \geq \CLowerFugaHitLOWER \cdot \ell \cdot \capacity(K),
	\end{equation}
	where $\mathcal{G} = \ball(2 \cdot R) \setminus \bigcup_{x \in 10 r \Z^d} \ball(x, 4 \cdot r)$.
\end{lemma}

\begin{proof}
	For any $z \in \mathcal{G}$ we have
	\begin{equation}\label{split_to_diff_of_two}
	P_z \big(\, T_K \leq  \ell/3 , \,  T_{\ball(3R)^c} > \ell   \, \big) =
	P_z \big(\, T_K \leq  \ell/3  \,\big) -
	P_z \big(\, T_K \leq  \ell/3 , \,  T_{\ball(3R)^c} \leq \ell   \, \big),
	\end{equation}
	hence Lemma \ref{lower:fuga_hitting_sum} can be proved in two steps. First we will show that if $\CLowerFugaHitDELTALOWER$ is big enough and $\CLowerFugaHitDELTAUPPER<1 $ then there exists a constant $\CPrfLowerFugaFirstLOWER > 0$ such that for all $\ell$ satisfying \eqref{ell_sandwich} we have
	\begin{equation}
	\label{eq:lowerfuga_first}
	\sum_{z \in \mathcal{G}} P_z \big(\, T_K \leq  \ell/3  \,\big) \geq \CPrfLowerFugaFirstLOWER \cdot \ell \cdot \capacity(K).
	\end{equation}
	After this, we will see that if we choose $\CLowerFugaHitDELTAUPPER$ small enough then
	\begin{equation}
	\label{eq:lowerfuga_second}
	\sum_{z \in \mathcal{G}} P_z \big(\, T_K \leq  \ell/3 , \,  T_{\ball(3R)^c} \leq \ell   \, \big) \leq \frac{\CPrfLowerFugaFirstLOWER}{2} \cdot \ell \cdot \capacity(K)
	\end{equation}
	also holds. From these, the desired lower bound \eqref{eq:lower:fuga_hitting_sum} will follow using \eqref{split_to_diff_of_two} with $\CLowerFugaHitLOWER = \frac{ \CPrfLowerFugaFirstLOWER }{2}$.
	
	We begin our proof of  \eqref{eq:lowerfuga_first} by observing that for any $z \in \mathcal{G}$ we have
	\begin{align}
		P_z \left( T_K \leq  \ell \slash 3  \right)
		& \geq
		P_z \left( T_K \leq 2 \cdot \max_{x \in K} |x - z|^2  \right) \cdot \ind \left[ \max_{x \in K}|x - z|^2 \leq  \ell \slash 6  \right]
		\notag \\ & \stackrel{(*)}{\geq}
		c \cdot \capacity(K) \cdot |y - z|^{2 - d} \cdot \ind \left[\, 4 \cdot |y - z|^2 \leq  \ell \slash 6 \, \right],
\label{TK_l_3_ind_bound}
	\end{align}
	where in $(*)$ we used  Lemma \ref{lemma:two_birds_one_stone} and also used that $K \subseteq \ball(y, 3r)$ and $|y-z| \geq 4r$ together imply
$\max_{x \in K} |x - z| \leq 2 \cdot |y-z|$.
 Summing this probability over $z \in \mathcal{G}$ we obtain a lower bound on the left-hand side of \eqref{eq:lowerfuga_first}:
	\begin{equation}
		\label{eq:first}
		\sum_{z \in \mathcal{G}} P_z \big(\, T_K \leq  \ell/3  \,\big) \stackrel{ \eqref{TK_l_3_ind_bound} }{\geq} c \cdot \capacity(K) \cdot
 \sum_{z \in \mathcal{G} \cap \ball(y, \sqrt{ \ell \slash 24 } ) } |y-z|^{2-d}  \stackrel{(\diamond)}{\geq} \CPrfLowerFugaFirstLOWER \cdot \capacity(K) \cdot \ell,
	\end{equation}
	where the inequality $(\diamond)$ requires some further explanation. Let us define a disjoint collection of annuli $\mathcal{A}(x), x \in  \mathcal{D}'(y,\ell) $, where
\begin{align}
\mathcal{D}'(y,\ell):= & 10 r \Z^d \cap \ball(2R - 5r ) \cap \ball(y, \sqrt{ \ell \slash 24 }-5r ), \\
 \mathcal{A}(x):= & \ball(x,5r-1)\setminus \ball(x,4r), \quad x \in  \mathcal{D}'(y,\ell).
\end{align}
Recalling that $\mathcal{G} = \ball(2 \cdot R) \setminus \bigcup_{x \in 10 r \Z^d} \ball(x, 4 \cdot r)$, we see that we have
\begin{equation}\label{fill}
  \bigcup_{x \in \mathcal{D}'(y,\ell)} \mathcal{A}(x) \subseteq \mathcal{G} \cap \ball(y, \sqrt{ \ell \slash 24 } ), \quad \text{moreover} \quad
  x\neq x' \in \mathcal{D}'(y,\ell) \, \implies \, \mathcal{A}(x) \cap \mathcal{A}(x')=\emptyset.
\end{equation}
 We are now ready to prove  inequality $(\diamond)$ of equation
\eqref{eq:first}:
\begin{multline}
 \sum_{z \in \mathcal{G} \cap \ball(y, \sqrt{ \ell \slash 24 } ) } |y-z|^{2-d} \stackrel{ \eqref{fill} }{\geq}
  \sum_{x \in \mathcal{D}'(y,\ell) } \sum_{z \in \mathcal{A}(x)} |y-z|^{2-d}
 \stackrel{(\circ)}{\geq} c \cdot r^d \cdot
   \sum_{x \in \mathcal{D}'(y,\ell)\setminus \{y\} }  |y-x|^{2-d}
    \stackrel{(\circ \circ)}{\geq}
    \\
c \cdot r^d \cdot   \sum_{k=1}^{\CPrfLowerFugaHitEXACTSTARUPPER \sqrt{\ell}/r -1} c' \cdot k^{d-1} \cdot  (10 \cdot r \cdot k )^{2-d}  = c \cdot r^2 \cdot \sum_{k=1}^{\CPrfLowerFugaHitEXACTSTARUPPER \sqrt{\ell}/r -1} k
   \stackrel{(**)}{\geq}
    c \cdot \ell,
 \end{multline}
where we explain the inequalities $(\circ)$, $(\circ \circ)$ and  $(**)$ below.

The inequality
 $(\circ)$ follows from the inequality $| \mathcal{A}(x) | \geq c \cdot r^d$ and the observation that if $x \in \mathcal{D}'(y,\ell)\setminus \{y\}$ then $\max_{z \in \mathcal{A}(x)} |y-z| \leq \tfrac{3}{2}|y-x|$.

 The inequality
   $(\circ \circ)$ holds with $\CPrfLowerFugaHitEXACTSTARUPPER = \frac{1}{10} (24)^{-1/2}$ using that for any $k \in \{1,\dots, \CPrfLowerFugaHitEXACTSTARUPPER  \sqrt{\ell}/r -1\}$ the number of elements $x$ of
$\mathcal{D}'(y,\ell)$ satisfying $|y-x|=10 \cdot r \cdot k$ is at least  $c' \cdot k^{d-1}$ (here we used our assumptions $y \in 10 r \Z^d$ and $\ball(y, r) \cap \ball(2R) \neq \emptyset$ together with the inequality $ \sqrt{ \ell \slash 24 } \leq 2R$, which follows from our assumptions $\CLowerFugaHitDELTAUPPER < 1$ and
$ \ell \leq \CLowerFugaHitDELTAUPPER \cdot  R^2$, cf.\    \eqref{ell_sandwich}).

The inequality $(**)$ holds if the sum on the l.h.s.\ of $(**)$ is non-empty, i.e., if $\CPrfLowerFugaHitEXACTSTARUPPER \sqrt{\ell}/r -1 \geq 1$, and this follows from the assumption $\CLowerFugaHitDELTALOWER \cdot r^2 \leq \ell$
(cf.\ \eqref{ell_sandwich}) if we choose $\CLowerFugaHitDELTALOWER \geq \left( \frac{2}{ \CPrfLowerFugaHitEXACTSTARUPPER } \right)^2= 9600$, completing the proof of \eqref{eq:lowerfuga_first}.

	Let us turn to the proof of \eqref{eq:lowerfuga_second}. We will consider the cases $ T_K < T_{\ball(3R)^c}$ and  $T_K > T_{\ball(3R)^c}$ separately. The first case is bounded as follows:
	\begin{align}
	& \sum_{z \in \mathcal{G}} P_z \left(\, T_K \leq  \tfrac{\ell}{3} , \,  T_{\ball(3R)^c} \leq \ell, \, T_K < T_{\ball(3R)^c}  \, \right)
	\stackrel{ (\diamond \diamond) }{\leq}
	\sum_{z \in \mathcal{G}} P_z \left( T_K \leq  \tfrac{\ell}{3}  \right) \cdot \max_{x \in K} P_x \left( T_{\ball(3R)^c} \leq \ell \right)
	\notag\\ & \quad \stackrel{\eqref{eq:upper:fuga_hitting_sum}}{\leq}
	C \cdot \capacity(K) \cdot \ell \cdot P_o \left(  \max_{0 \leq t \leq \CLowerFugaHitDELTAUPPER \cdot R^2}  |X(t)|  > R \right)
	\stackrel{ \eqref{eq:staying_inside} }{\leq}
	\label{eq:second:first}
	C \cdot \capacity(K) \cdot \ell \cdot e^{ - 1 \slash 2\CLowerFugaHitDELTAUPPER },
	\end{align}
	where in $(\diamond \diamond)$ we used the strong Markov property at $T_K$. The second case is bounded as follows:
	\begin{align}
	& \sum_{z \in \mathcal{G}}  P_z \left(\,  T_K \leq  \tfrac{\ell}{3} , \, T_{\ball(3R)^c} \leq \ell, \, T_{\ball(3R)^c} < T_K  \, \right)
	\notag \\ & \stackrel{(\bullet)}{\leq}
	|\mathcal{G}| \cdot \sup_{z \in \mathcal{G}} P_z \left( T_{\ball(3R)^c} \leq \ell \right) \cdot
	\max_{x \in \partial^{\text{ext}} \ball(3R)} P_x \left( T_K < \infty \right)
	\notag \\ & \stackrel{(\bullet \bullet)}{\leq}
	\label{eq:second:second}
	C \cdot R^d \cdot P_o \left( \max_{0 \leq t \leq \ell}  |X(t)|  \geq R \right) \cdot \capacity(K) \cdot R^{2-d}
	\stackrel{\eqref{eq:staying_inside}}{\leq}
	C \cdot \capacity(K)  \cdot e^{- \tfrac{1}{2} R^2 \slash \ell } \cdot R^2 ,
	\end{align}
	where in $(\bullet)$ we used the strong Markov property at $T_{\ball(3R)^c}$, $(\bullet \bullet)$ follows from $\mathcal{G} \subseteq \ball(2R)$, $|\mathcal{G}| \leq C \cdot R^d$,  \eqref{LED} and \eqref{green_bounds}. Now  we need $\exp \{ -\tfrac{1}{2} R^2 \slash \ell \} \cdot R^2 \slash \ell$ to be small if we want an upper bound as in \eqref{eq:lowerfuga_second}. This can be achieved by choosing $ \CLowerFugaHitDELTAUPPER $ small enough, since $R^2 \slash \ell \geq 1 \slash \CLowerFugaHitDELTAUPPER$.
	
	Putting \eqref{eq:second:first} and \eqref{eq:second:second} together we get the required upper bound \eqref{eq:lowerfuga_second} by choosing $\CLowerFugaHitDELTAUPPER$ small enough. Combining this result with  \eqref{eq:lowerfuga_first} we obtain \eqref{eq:lower:fuga_hitting_sum}.
\end{proof}

\subsection{On the number and location of sub-boxes visited by a walker in a box}\label{subsection_rw_sub_boxes}

In Section \ref{subsection_rw_sub_boxes}  we consider a box of scale $R$ and its sub-boxes of scale $r$. We prove matching upper (cf.\ \eqref{eq:upper:exp_numbof_visited_boxes}) and lower (cf.\ \eqref{eq:lower:no_of_visited_boxes}) bounds  on the number of
sub-boxes of scale $r$ visited by a random walk.
 In \eqref{eq:upper:weighted_exp_visited_boxes} we bound the expectation of a quantity closely related to the Dirichlet energy of the counting measure on the centers of the sub-boxes visited by the random walk.
   The reader may want to skip the proofs of these technical results at first reading.

  Note that for \eqref{eq:upper:exp_numbof_visited_boxes} below we only need to assume $d \geq 3$.
\begin{lemma}[The sub-boxes visited by a walker are few and well spread-out]
	\label{upper:numb_of_visited_boxes}
Let $d\geq 5$.	For $r \leq R \in \N$ let us define
\begin{equation}
 \mathcal{D} := \mathcal{D}(r,R) := \{\, y \in 10 r \Z^d \, : \, \ball( y, r) \cap \ball( 2 R ) \neq \emptyset \, \}, \quad
 \zeta_y := \ind \left[ T_{\ball(y, r)} < \infty \right], \quad y \in \mathcal{D}.
 \end{equation}
 There exists $C < \infty$ such that for any $r \leq R \in \N$ and $z \in \ball(R)$ we have
	\begin{align}
	\label{eq:upper:exp_numbof_visited_boxes}
&	E_z \left[ \sum_{y \in \mathcal{D}} \zeta_y  \right]  \leq C \cdot \big( R \slash r \big)^2, \\
	\label{eq:upper:weighted_exp_visited_boxes}
&	E_z \left[ \sum_{y \in \mathcal{D}} \sum_{y' \in \mathcal{D} \setminus \{ y \} } \zeta_y \cdot \zeta_{y'} \cdot  |y - y'|^{2 - d} \right]
	 \leq
	C \cdot \big( R \slash r \big)^2 \cdot r^{2-d}.
	\end{align}
\end{lemma}

\begin{proof}
	First note that for any distinguished vertex $y_0 \in \mathcal{D}$ we have
	\begin{equation}
	\label{eq:summing_y-y0_by_the_surfaces}
	\sum_{y \in \mathcal{D} \setminus \{ y_0 \} } | y - y_0 |^{2-d}
	\stackrel{(\bullet)}{\leq}
	C \cdot \sum_{s = 1}^{ \lceil 2 R \slash 10 r \rceil + 1 } s^{d-1} \cdot ( 10 \cdot r \cdot s )^{2-d}
	\leq
	C \cdot r^{2 - d} \cdot (R \slash r)^{2},
	\end{equation}
where $(\bullet)$ holds	since any $y \in \mathcal{D} \setminus \{ y_0 \}$ is located on the surface of a sphere around $y_0$ with radius $10 \cdot r \cdot s$ (for $s = 1, \ldots, \lceil 2 R \slash 10 r \rceil + 1 $) and the number of $y \in \mathcal{D}$ at distance $10 \cdot r \cdot s$ from $y_0$ is bounded by $C \cdot s^{d-1}$.

Let $y_0 \in \mathcal{D}$ denote one of the closest vertices of $\mathcal{D}$ to $z$. By $z \in \ball(R)$ and the definition of $\mathcal{D}$,
we have $|z - y_0| \leq 5 r$.

	Using these observations, the first upper bound \eqref{eq:upper:exp_numbof_visited_boxes} can be obtained as follows:
	\begin{align}
	E_z \left[ \sum_{y \in \mathcal{D}} \zeta_y \right]
	& \stackrel{(*)}{\leq}
	1 + C \cdot \sum_{y \in \mathcal{D} \setminus \{ y_0 \} } \capacity( \ball(y, r) ) \cdot \max_{x \in \ball(y, r)} g(z, x)
	\notag \\ & \stackrel{(**)}{ \leq }
	1 + C \cdot r^{d-2} \cdot \sum_{y \in \mathcal{D} \setminus \{ y_0 \}  }  |y_0 - y|^{ 2 - d }
	\stackrel{ \eqref{eq:summing_y-y0_by_the_surfaces} }{\leq}
	\label{eq:prf_of_upper:exp_numbof_visited_boxes}
	C \cdot (R \slash r)^2.
	\end{align}
	Here, the inequality $(*)$ follows from the linearity of expectation, the trivial bound $\zeta_{y_0} \leq 1$ and \eqref{LED}, in $(**)$ we used \eqref{eq:capacity_of_ball}, \eqref{green_bounds} and the fact that $(4 \slash 10) \cdot |y_0 - y| \leq |z - x|$ holds for any $x \in \ball(y, r)$ if $y \in \mathcal{D} \setminus \{ y_0 \}$.
	
	Let us turn to the proof of \eqref{eq:upper:weighted_exp_visited_boxes}. Due to the linearity of  expectation, let us first bound $E_z[\zeta_y \cdot \zeta_{y'} ]$ for $y \neq y' \in \mathcal{D} \setminus \{ y_0 \}$. However, to obtain an appropriate upper bound, it is useful to examine the following probability:
	\begin{align}
	& P_z \left( T_{ \ball(y, r) } < T_{\ball( y', r)  } < \infty \right)
	\stackrel{(\diamond)}{\leq}
	P_z \left( T_{ \ball(y, r) } < \infty \right) \cdot \max_{ \hat{x} \in \ball(y, r)} P_{\hat{x}} \left( T_{ \ball(y', r) } < \infty \right)
	\notag \\ & \quad \stackrel{\eqref{LED}}{\leq}
	C \cdot \capacity( \ball(y, r) ) \cdot \max_{ x \in \ball(y, r) } g(z, x) \cdot \capacity( \ball( y', r) ) \cdot \max_{\hat{x} \in \ball( y, r)} \max_{ x' \in \ball(y', r) } g( \hat{x}, x' )
	\notag \\ & \quad \stackrel{(\diamond \diamond)}{\leq}
	\label{eq:probability_to_bound_the_expectation}
	C \cdot r^{d-2} \cdot |y_0 - y|^{2 - d} \cdot r^{d - 2} \cdot |y - y'|^{2 - d},
	\end{align}
	where the inequality $(\diamond)$ follows from the strong Markov property applied at $T_{\ball( y, r )}$, and in $(\diamond \diamond)$ we used \eqref{eq:capacity_of_ball}, \eqref{green_bounds} and the facts that $(4 \slash 10) \cdot |y_0 - y| \leq |z - x|$ for any $x \in \ball(y, r)$ and $(8 \slash 10) \cdot |y - y'| \leq |\hat{x} - x'|$ for any $\hat{x} \in \ball(y, r)$, $x' \in \ball(y', r)$.

	By \eqref{eq:probability_to_bound_the_expectation} and the definition of the indicators $\zeta_y$ ($y \in \mathcal{D}$) we obtain
	\begin{align}
	\label{eq:exp_of_prod_taus}
	E_z \left[ \zeta_y \cdot \zeta_{y'} \right] \leq C \cdot r^{ 2d - 4 } \cdot | y - y' |^{2 - d} \cdot \left( |y_0 - y|^{2-d} + | y_0 - y' |^{2 - d} \right),
	\end{align}
	hence if we sum over the pairs of vertices of $\mathcal{D} \setminus \{ y_0 \}$, we get
	\begin{align}
	& E_z \left[ \sum_{y \in \mathcal{D} \setminus \{ y_0 \} } \sum_{y' \in \mathcal{D} \setminus \{ y_0, \, y \} }
 \zeta_y \cdot \zeta_{y'} \cdot  |y - y'|^{2 - d}  \right]
	\notag \\ & \stackrel{(\bullet)}{\leq}
	C \cdot r^{ 2d - 4 } \cdot \sum_{y \in \mathcal{D} \setminus \{ y_0 \}} |y_0 - y|^{2-d} \cdot \sum_{ y' \in \mathcal{D} \setminus \{ y_0, y \} } | y - y' |^{4 - 2d}
	\notag \\ & \stackrel{(\bullet \bullet)}{\leq}
	C \cdot r^{ 2d - 4 } \cdot \sum_{y \in \mathcal{D} \setminus \{ y_0 \} } | y_0 - y|^{ 2 - d } \cdot
	\sum_{s = 1}^{ \lceil 2 R \slash 10 r \rceil + 1 } s^{d-1} \cdot (10 \cdot r \cdot s)^{ 4 - 2d }
	\notag \\ & \stackrel{(\circ)}{\leq}
	C  \cdot
	\sum_{y \in \mathcal{D} \setminus \{ y_0 \} } | y_0 - y|^{ 2 - d }
	\stackrel{\eqref{eq:summing_y-y0_by_the_surfaces}}{\leq}
	\label{eq:when_there_are_no_ynull}
	C  \cdot r^{ 2 - d } \cdot ( R \slash r)^2.
	\end{align}
	Here, in $(\bullet)$ we used \eqref{eq:exp_of_prod_taus} and the symmetry of the summands, $(\bullet \bullet)$ follows using the same argument as in \eqref{eq:summing_y-y0_by_the_surfaces}, and in $(\circ)$ we used $d \geq 5$ and $r^{2d-4}\cdot r^{4-2d}=1$.

	To obtain the desired upper bound \eqref{eq:upper:weighted_exp_visited_boxes}, we also need to take care of the terms in which at least one of the vertices $y, y' \in \mathcal{D}$ is equal with $y_0$. However, by the trivial bound $\zeta_{y_0} \leq 1$ and a similar argument as in \eqref{eq:prf_of_upper:exp_numbof_visited_boxes}, we have
	\begin{align}
	& E_z \left[ \sum_{y \in \mathcal{D}} \sum_{y' \in \mathcal{D} \setminus \{ y \} } \zeta_y \cdot \zeta_{y'} \cdot |y - y'|^{2 - d} \cdot \ind \left[ \left\{ y = y_0 \right\} \cup \left\{ y' = y_0 \right\} \right] \right]
	\leq \notag \\ & \leq
	2 \cdot E_z \left[ \sum_{y \in \mathcal{D} \setminus \{ y_0 \}} \zeta_y \cdot  |y_0 - y|^{2 - d} \right] \leq  C  \cdot r^{ 2 - d }.
	\end{align}
	Since $R \geq r$, adding this to \eqref{eq:when_there_are_no_ynull} gives the upper bound \eqref{eq:upper:weighted_exp_visited_boxes} stated in the lemma.	
\end{proof}

Our next lemma gives a lower bound on the number of sub-boxes of scale $r$ in a box of scale $R$ visited by a random walk.
 It shows that \eqref{eq:upper:exp_numbof_visited_boxes} is in some sense sharp. For technical reasons (that will only be clarified in Section \ref{subsection_doubling_gamma}) we need to restrict our attention to the sub-boxes visited before the walker either exits $\ball(2 R)$ or performs $R^2 - r^2$ steps.
   The statement of Lemma \ref{lower:no_of_visited_boxes} seems rather obvious, however we did not find a result in the
 literature of random walks which is uniform in the parameters $r$ and $R$ in a way that fits our requirements, therefore we decided to include the full proof.

\begin{lemma}[Lower bound on the number of sub-boxes visited by a random walk]
	\label{lower:no_of_visited_boxes} Let $d \geq 3$.
For $r \leq R \in \N$ let us define
\begin{align}
 \mathcal{D} &:= \mathcal{D}(r,R) := \{\, y \in 10 \, r  \Z^d \, : \, \ball( y, r) \cap \ball( 2 R ) \neq \emptyset \, \}, \\
 \chi_y &:= \mathds{1}\left[\, T_{\ball(y, r )} <  T_{\ball(2 R)^c} \wedge (R^2 - r^2)\,  \right], \quad y \in \mathcal{D}.
 \end{align}
There exist constants $\CLowerNOVisitedBoxesLOWER > 0$ and $ \CLowerNOVisitedBoxesRATIOLOWER  < \infty$ such that if $R/r \geq \CLowerNOVisitedBoxesRATIOLOWER$ then for any $z \in \ball(R)$ we have
	\begin{equation}
	\label{eq:lower:no_of_visited_boxes}
	P_z \left( \sum_{y \in \mathcal{D}} \chi_y \geq \CLowerNOVisitedBoxesLOWER \cdot R^2 \slash r^2  \right) \geq 0.997
	\end{equation}
\end{lemma}

\subsubsection{Lower bound on the number of sub-boxes visited by a random walk}\label{subsub_low_subbox_rw}

In Section \ref{subsub_low_subbox_rw} we prove Lemma \ref{lower:no_of_visited_boxes}. The proof relies on two simple ideas.
The first idea is to run the walker only for $\varepsilon \cdot R^2$ steps
to make sure that it does not exit $\ball(2 R)$ with high probability. The second idea is that instead of counting the number of visits
to the boxes $\ball(y, r )$ centered at $ y \in 10r\mathbb{Z}^d$, we will only count the  number of visits to boxes $\ball(y, r )$  centered at  $ y \in \beta r\mathbb{Z}^d$ (where $\beta$ is a big multiple of $10$): the relatively bigger distance between the boxes helps us to control the number of annoying revisits to the boxes already visited.

 For any $\veps \in (0, 1 \slash 2)$ and any $y \in 10 \, r  \Z^d$, let us introduce
\begin{equation}
	\chi_y^{\veps} := \ind \left[ \, T_{\ball(y, r)} \leq \veps \cdot R^2 \, \right],  \qquad
	A ( \veps) := \{ \, T_{\ball(2 R)^c} > \veps \cdot R^2 \, \}.
\end{equation}
Let us assume  $R/r \geq 2$, so that we have $\veps \cdot R^2 \leq R^2 - r^2$.

 For any $\veps \in (0, 1 \slash 2)$, $\beta \in 10 \N$ and $\CLowerNOVisitedBoxesLOWER \in \mathbb{R}_+$ we have
\begin{equation}\label{simplify_chi_sum}
P_z \left( \sum_{y \in \mathcal{D}} \chi_y \geq \CLowerNOVisitedBoxesLOWER \cdot R^2 \slash r^2  \right) \geq
P_z \left( \sum_{y \in \beta r \Z^d} \chi_y^{\veps} \geq  \CLowerNOVisitedBoxesLOWER \cdot R^2 \slash r^2 \right) - P_z(A(\veps)^c),
\end{equation}
since  $A (\veps)$ implies  $\sum_{y \in 10 r \Z^d} \chi_y^{\veps} \leq \sum_{y \in \mathcal{D}} \chi_y$,  moreover $\beta$ is a positive multiple of $10$.

We have $ A(\veps)^c \subseteq \{ \max_{1 \leq t \leq \veps \cdot R^2} |X(t) - z| \geq R \}$ by  $z \in \ball(R)$, thus we can use \eqref{eq:staying_inside} to  choose (and fix) an $\veps > 0$ such that $P_z ( A( \veps )^c ) \leq 10^{-3}$ holds for any $R \geq 1$.
 Thus, in order to conclude the proof of \eqref{eq:lower:no_of_visited_boxes}, it is enough to show that there exist constants $\beta \in 10 \mathbb{N}$, $\CLowerNOVisitedBoxesLOWER > 0$ and $\CLowerNOVisitedBoxesRATIOLOWER < \infty$ such that if $R/r \geq \CLowerNOVisitedBoxesRATIOLOWER$ then for any $z \in \ball(R)$ the first term on the right-hand side of \eqref{simplify_chi_sum} is at least $0.998$.

Our next goal is to write $\sum_{y \in \beta r \Z^d} \chi_y^{\veps}$ in a more manageable form, see \eqref{manageable} below.
 Let $\widetilde{\mathcal{B}}^{\beta} := \bigcup_{y \in \beta r \Z^d} \ball(y, r)$.
 Let us introduce the sequence of stopping times $\tau_k^{\beta}, k \in \mathbb{N}$ and the corresponding sequence of indices $y_k^{\beta} \in \beta r \Z^d, k \in \N$:
\begin{align}
	\tau_1^{\beta} & := \min \left\{ t \geq 0 \, : \, X(t) \in \widetilde{\mathcal{B}}^{\beta} \right\},
	\qquad
	y_1^{\beta} := \arg \min_{ y \in \beta r \Z^d } \left| y - X( \tau_1^{\beta} ) \right|; \\
	\tau_k^{\beta} & := \min \left\{ t \geq \tau_{k-1}^{\beta} \, : \, X(t) \in \widetilde{\mathcal{B}}^{\beta} \setminus \ball(y_{k-1}^{\beta}, r) \right\},
	\qquad
	y_k^{\beta} := \arg \min_{ y \in \beta r \Z^d } \left| y - X( \tau_k^{\beta} ) \right|.
\end{align}
In words, the stopping times $\{ \tau_k^{\beta} \}_{k \geq 1}$ are the consecutive times when our walker visits a box of form $\ball(y, r), y \in \beta r \Z^d$ other than the one visited last, and the sequence $\{ y_k^{\beta} \}_{k \geq 1}$ consists of the centers of the corresponding boxes, hence for any $k \geq 1$ we have $y_k^{\beta} \neq y_{k+1}^{\beta}$.

 Let $\tau_0^{\beta}:=0$ and
\begin{equation}\label{n_beta_def_eq}
 N^{\beta} := \max \{\, k \in \N_0 \, : \, \tau_k^{\beta} \leq \veps \cdot R^2 \, \}.
\end{equation}
For any $1 \leq k \leq N^{\beta}$ let us define the indicator
\begin{equation}\label{eta_indicator_def}
	 \eta_k :=
	 \ind \left[\, \exists \, \ell \in \left\{ k+1, \ldots, N^{\beta} \right\} \, : \, y_{\ell}^{\beta} = y_k^{\beta}\, \right]
\end{equation}
of the event that the visit to  $\ball(y_{k}^{\beta},r)$ at time $\tau_k^{\beta}$ is not the last one before time $\varepsilon \cdot R^2$.

Using this notation and
noting that $1-\eta_k=\ind \left[ y_{\ell}^{\beta} \neq y_k^{\beta}, \, \ell = k+1, \ldots, N^{\beta} \right]$, we obtain
\begin{equation}\label{manageable}
	\sum_{y \in \beta r \Z^d} \chi_y^{\veps} =
	\sum_{k = 1}^{ N^{\beta} } \left( 1 - \eta_k \right).
\end{equation}
Thus, in order to conclude the proof of Lemma \ref{lower:no_of_visited_boxes}, it is enough to show that there exist constants $\beta \in 10 \mathbb{N}$, $ \CLowerNOVisitedBoxesLOWER  > 0$ and $\CLowerNOVisitedBoxesRATIOLOWER < \infty$ such that if $R/r \geq \CLowerNOVisitedBoxesRATIOLOWER$ then for any $z \in \mathbb{Z}^d$
\begin{equation}
	\label{eq:reformalized_thin_numb_of_visited_boxes}
	P_z \left( \sum_{k = 1}^{ N^{\beta} } ( 1 - \eta_k ) \geq \CLowerNOVisitedBoxesLOWER \cdot R^2 \slash r^2  \right) \geq 0.998.
\end{equation}
We will use the following claim when we lower bound  $N^{\beta}$.

\begin{claim}
	\label{claim:consecutive_visits}
	 There exists $\CCLAIMConsecutiveVisits < \infty$ such that for any $k,r \geq 1$ and any $\beta \in 10 \mathbb{N}$ we have
	\begin{equation}
		\label{eq:lemma:consecutive_visits}
		\mathbb{E} \left[ \tau_{k}^{\beta} - \tau_{k-1}^{\beta} \right] \leq \CCLAIMConsecutiveVisits \cdot r^2 \cdot \beta^{d}.
	\end{equation}
\end{claim}

\begin{proof}
	Let $\widehat{\mathcal{B}}^{\beta} := \bigcup_{y \in 2 \beta r \Z^d} \ball(y, r)$ and $M^\beta:=\max_{z' \in \mathbb{Z}^d } E_{z'}( T_{\widehat{\mathcal{B}}^{\beta}} )$. Let us first note that it is enough to prove that  there exists $\CCLAIMConsecutiveVisits < \infty$ such that for any $\beta \in 10 \mathbb{N}$ we have $M^\beta \leq \CCLAIMConsecutiveVisits \cdot r^2 \cdot \beta^{d}$. Indeed, one obtains $\mathbb{E} [ \tau_{k}^{\beta} - \tau_{k-1}^{\beta} ] \leq M^\beta $ by applying the strong Markov property at $\tau_{k-1}^{\beta}$ and using that $\widetilde{\mathcal{B}}^{\beta} \setminus \ball(y_{k-1}^{\beta}, r)$ contains a translate of $\widehat{\mathcal{B}}^{\beta}$.

	Let $p^\beta:=\min_{z' \in \mathbb{Z}^d } P_{z'}( T_{\widehat{\mathcal{B}}^{\beta}} \leq 8 (\beta r)^2 )$. The inequality $M^\beta \leq 8 (\beta r)^2+ (1-p^\beta)M^\beta$ follows by an application of the Markov property at time $8 (\beta r)^2$, from which $M^{\beta} \leq 8 ( \beta r)^2 /p^\beta $ follows by rearrangement, thus we only need to show  $p^\beta \geq c \beta^{2-d}$ (with some $c>0$) to conclude the proof.

	For any $z' \in \mathbb{Z}^d$, let $y'$ denote one of the vertices in $2 \beta r \Z^d$ for which $\left| y' - z'  \right|$ is minimal. Noting that $\max_{x \in \ball(y', r)} |x - z'|\leq 2 \beta r$, we may apply Lemma \ref{lemma:two_birds_one_stone} to deduce the desired bound:
	\begin{equation*}
		P_{z'}( T_{\widehat{\mathcal{B}}^{\beta}} \leq 8 (\beta r)^2 )
		\geq
		P_{z'}( T_{ \ball( y' , r) } \leq 2 (2 \beta r)^2 )
		\stackrel{ \eqref{eq:lemma:two_birds_one_stone} }{\geq }
		c \cdot \capacity( \ball( y', r) ) \cdot (2 \beta r)^{2-d}
		\stackrel{ \eqref{eq:capacity_of_ball} }{\geq}
		c \beta^{2-d}.
	\end{equation*}
\end{proof}

With Claim \ref{claim:consecutive_visits} at hand, let us consider the probability of the complement of the event that appears in \eqref{eq:reformalized_thin_numb_of_visited_boxes}. If we introduce the parameter $K \in \N$, whose value will be declared later, then by partitioning this event depending on whether $N^{\beta} \geq K$ or not, we have
\begin{equation}
	\label{eq:reformalized_and_divided_thin_boxes}
	P_z \left( \sum_{k = 1}^{N^{\beta}} ( 1 - \eta_k ) < \CLowerNOVisitedBoxesLOWER \cdot R^2 \slash r^2 \right) \leq
	P_z \left( \sum_{k = 1}^{K} ( 1 - \eta_k ) < \CLowerNOVisitedBoxesLOWER \cdot R^2 \slash r^2 \right) + P_z \left( N^{\beta} < K \right).
\end{equation}
The second term on the right-hand side of \eqref{eq:reformalized_and_divided_thin_boxes} can be bounded from above using Markov's inequality and Claim \ref{claim:consecutive_visits} as follows:
\begin{align}
	\label{eq:N_beta_big}
	P_z \left( N^{\beta} < K \right)
	\stackrel{ \eqref{n_beta_def_eq} }{=}
	P_z \left( \tau_K^{\beta} > \veps  R^2 \right)
	\leq
	\frac{ \mathbb{E} \left[ \tau_K^{\beta} \right] }{ \veps  R^2 }
	=
	\frac{ \sum_{k = 1}^{K} \mathbb{E} \left[ \tau_{k}^{\beta} - \tau_{k-1}^{\beta}  \right] }{\veps  R^2}
	\stackrel{\eqref{eq:lemma:consecutive_visits}}{ \leq }
	\frac{ K \cdot \CCLAIMConsecutiveVisits \cdot \beta^{d} }{ \veps  R^2 \slash r^2 }.
\end{align}
We will later choose the value of $\beta$ and $K$ in a way that makes the r.h.s.\ of \eqref{eq:N_beta_big}  small. Before that,  let us
perform some preparatory calculations that will help us deal with the first term on the right-hand side of \eqref{eq:reformalized_and_divided_thin_boxes}.
  For any $1 \leq k \leq N^{\beta}$ we have
\begin{align}
	\mathbb{E} \left[ \eta_k \right]
	& \stackrel{( * )}{\leq}
	\mathbb{P} \left( \exists \, \ell \geq k+2 \, : \, y_{\ell}^{\beta} = y_k^{\beta} \right)
	\stackrel{( ** )}{=}
	\mathbb{E}\left( P_{X(\tau^{\beta}_{k+1})}\left( T_{\ball(y^\beta_k, r)}<\infty   \right) \right)
	\notag \\ &
	\label{eq:easy_LED_bound_on_exp_etaks}
	\stackrel{(\diamond)}{\leq}
	\mathbb{E}\left( C |y^{\beta}_{k+1}-y^{\beta}_k|^{2-d} \capacity ( \ball(y^{\beta}_k, r) )  \right)
	\stackrel{( \diamond \diamond )}{\leq}
	\CPrfEasyLEDBoundOnExpETAKS (\beta r)^{2-d} r^{d-2}
	=
	\CPrfEasyLEDBoundOnExpETAKS \cdot \beta^{ 2 - d },
\end{align}
where in $(*)$ we used \eqref{eta_indicator_def} and  $y_{k+1}^{\beta} \neq y_k^{\beta}$, in $( ** )$ we used the strong Markov property at $\tau^{\beta}_{k+1}$, in  $(\diamond)$ we used \eqref{LED} and \eqref{green_bounds} together with the fact  that the distance of $X(\tau^{\beta}_{k+1})$ and $\ball(y^{\beta}_k, r)$  is at least $\frac{8}{10}|y^{\beta}_{k+1}-y^{\beta}_k|$, and finally in $( \diamond \diamond)$ we used \eqref{eq:capacity_of_ball} and that $|y^{\beta}_{k+1}-y^{\beta}_k| \geq \beta r $.

Let us now choose (and fix) the value of $\beta$ so big that we have $4 \cdot \CPrfEasyLEDBoundOnExpETAKS \cdot \beta^{2 - d} \leq  10^{-3}$, where $\CPrfEasyLEDBoundOnExpETAKS$ appears on the right-hand side of \eqref{eq:easy_LED_bound_on_exp_etaks}. Recall that we introduced the constant $\CCLAIMConsecutiveVisits$ in Claim \ref{claim:consecutive_visits}.
  Having already fixed the value of $\varepsilon$ and $\beta$, we can fix the values of $K$ and $\CLowerNOVisitedBoxesLOWER$ as follows:
\begin{equation}
	\label{value_of_K_and_hat_c}
	K := \left \lfloor \CCLAIMConsecutiveVisits^{-1} \cdot 10^{-3} \cdot \veps  \cdot \beta^{-d} \cdot ( R \slash r )^2 \right \rfloor \geq 1,
	\qquad \CLowerNOVisitedBoxesLOWER := \frac{1}{4} \cdot \CCLAIMConsecutiveVisits^{-1} \cdot  10^{-3} \cdot \veps  \cdot \beta^{-d},
\end{equation}
where $K \geq 1$ is guaranteed if  $(R\slash r)^2 \geq \CCLAIMConsecutiveVisits \cdot 10^3 \cdot \veps^{-1} \cdot \beta^d$. We will therefore  fix the constant $\CLowerNOVisitedBoxesRATIOLOWER$ that appears in the statement of Lemma \ref{lower:no_of_visited_boxes} to be
$\CLowerNOVisitedBoxesRATIOLOWER:= \sqrt{\CCLAIMConsecutiveVisits \cdot 10^3 \cdot \veps^{-1} \cdot \beta^d}$.

  With this choice of $K$, using \eqref{eq:N_beta_big}, we obtain $P_z ( N^{\beta} < K ) \leq 10^{-3}$.

We can now use \eqref{eq:easy_LED_bound_on_exp_etaks}  to bound the first term on the r.h.s.\ of \eqref{eq:reformalized_and_divided_thin_boxes} as
\begin{align}
	\label{eq:FIRST_TERM_reformalized_and_divided_thin_boxes}
	P_z \left( \sum_{k = 1}^{K} ( 1 - \eta_k ) < \CLowerNOVisitedBoxesLOWER \cdot \frac{R^2}{ r^2} \right)
	=
	P_z \left( \sum_{k = 1}^{K} \eta_k > K - \CLowerNOVisitedBoxesLOWER \cdot \frac{R^2}{r^2} \right)
	\stackrel{( \bullet)}{\leq}
	\frac{K \cdot \CPrfEasyLEDBoundOnExpETAKS \cdot \beta^{2 - d}}{ K - \CLowerNOVisitedBoxesLOWER \cdot R^2 \slash r^2 } \stackrel{( \bullet  \bullet  )}{\leq}
	10^{-3},
\end{align}
where in $(\bullet)$ we used Markov's inequality, the linearity of expectation and the upper bound obtained in \eqref{eq:easy_LED_bound_on_exp_etaks}, and
$( \bullet \bullet )$ holds since $K - \CLowerNOVisitedBoxesLOWER \cdot R^2 \slash r^2 \geq K/4$ follows from \eqref{value_of_K_and_hat_c} and $4 \cdot \CPrfEasyLEDBoundOnExpETAKS \cdot \beta^{2 - d} \leq  10^{-3}$ holds by our choice of $\beta$. Putting these bounds together we obtain \eqref{eq:reformalized_thin_numb_of_visited_boxes} and hence the proof of Lemma \ref{lower:no_of_visited_boxes} is complete.

\section{Point processes of worms}\label{section_ppp_worms}

In Section \ref{section_ppp_worms} we introduce some notation and state some basic results about Poisson point
processes on the space of worms that  we will use in later sections.

Let us define the space of $\Z^d$-valued nearest neighbour paths of length $\ell \in \mathbb{N}$:
\begin{equation}
	\label{def:space_of_worms_of_length_l}
	W^\ell:= \left\{ \, w = \left( w(0), \dots, w(\ell-1)\right) \, : \,  w(j) \sim w(j-1), \, 1 \leq j \leq \ell-1 \, \right\}.
\end{equation}
We call $W^\ell$ the set of \emph{worms} of length $\ell$.
Let us stress that a worm of length $\ell$ performs $\ell-1$ steps and thus a worm of length $1$ solely consists of its starting point $w(0)$.

We define the \emph{space of worms} as the disjoint union of $W^\ell$, $\ell \in \mathbb{N}$:
\begin{equation}
	\label{def:space_of_worms}
	\widetilde{W}:= \bigcup_{\ell=1}^{\infty} W^\ell.
\end{equation}
Note that the set $\widetilde{W}$ is countable.

If $w \in \widetilde{W}$ then we denote by $L(w)$ the length of the worm $w$, i.e.,
\begin{equation}
	\label{nota:length_of_a_worm}
	L(w):=\ell \; \iff \; w \in W^{\ell}.
\end{equation}

\noindent We will use the notation of the entrance time $T_K$ (cf.\ \eqref{entrance_time}) for worms as well: if $w(j) \notin K $ for all $0\leq j\leq L(w)-1$ then we declare $T_K(w)=+\infty$ (thus $T_K(w)\geq L(w) \iff T_K(w)=+\infty$).

\noindent Given $w \in \widetilde{W}$ and $0 \leq t_1 \leq t_2 \leq L(w)-1$, let us denote by $w[t_1,t_2]$ the shorter worm
\begin{equation}
	\label{nota:worm_surgery}
	w[t_1,t_2]:= (w(t_1),\dots, w(t_2)).
\end{equation}
In words, we say that $w[t_1,t_2]$ is a \emph{sub-worm} of $w$.

Let us denote by $\widetilde{\mathbf{W}}$ the \emph{space of point measures on} $\widetilde{W}$, i.e.,
\begin{equation}
	\mathcal{X} \in \widetilde{\mathbf{W}} \, \iff \,
				\mathcal{X} = \sum_{i \in I} \delta_{w_i}
\end{equation}
where $I$ is a finite or countably infinite set of indices, $w_i \in \widetilde{W}$ for all $i \in I$ and  $\delta_{w}$ denotes the Dirac mass concentrated on $w \in \widetilde{W}$. Let us denote by $\mathcal{X}(A)$  the $\mathcal{X}$-measure of $A \subseteq \widetilde{W}$. We can alternatively think about an $\mathcal{X}\in \widetilde{\mathbf{W}}$ as a multiset of worms, where $\mathcal{X}(\{ w \})$ denotes the number of copies of the worm $w$ contained in $\mathcal{X}$.

If $\mathcal{X}  = \sum_{i \in I} \delta_{w_i} \in \widetilde{\mathbf{W}}$ and $A \subseteq \widetilde{W}$, let us denote by
\begin{equation}
	\label{nota:restrict_pointproc_of_worms}
	\mathcal{X} \ind[A] := \sum_{i \in I} \delta_{w_i} \mathds{1}[ w_i \in A ]
\end{equation}
the restriction of $\mathcal{X}$ to the set $A$ of worms. Note that $ \mathcal{X} \ind[A]$ is also an element of $\widetilde{\mathbf{W}}$.

Later on we want to consider random variables defined on the space $\widetilde{\mathbf{W}}$ of form $\mathcal{X} (A)$ for any $A \subseteq \widetilde{W}$, hence we define $\boldsymbol{\widetilde{\mathcal{W}}}$ to be the sigma-algebra on $\widetilde{\mathbf{W}}$ generated by such variables.

Given $A \subseteq \widetilde{ W }$, if $\varphi : A \to  \widetilde{W}$ is a  function  and $\mathcal{X}= \sum_{i \in I} \delta_{w_i} \in \widetilde{\mathbf{W}}$ is a point measure satisfying $\mathcal{X} = \mathcal{X} \ind[ A ]$, then let
\begin{equation}
	\label{nota:map_of_point_process}
	\varphi(\mathcal{X}):= \sum_{i \in I} \delta_{\varphi(w_i)}
\end{equation}
denote the image of $\mathcal{X}$ under $\varphi$. Note that $\varphi(\mathcal{X})$ is also an element of $\widetilde{\mathbf{W}}$. An example for such a function $\varphi$ is the following. Given  $w \in \widetilde{W}$, $\mathcal{X} = \sum_{i \in I} \delta_{w_i} \in \widetilde{\mathbf{W}}$ and $x \in \mathbb{Z}^d$, let us define the shifted objects $w+x \in \widetilde{W}$ and $\mathcal{X}+x \in  \widetilde{\mathbf{W}} $ by
\begin{equation}
	\label{nota:shift_of_worms}
	w+x:=(w(0)+x,\dots,w(\ell-1)+x) \; \text{ if } \; L(w) =\ell, \qquad \mathcal{X}+x:= \sum_{i \in I} \delta_{w_i+x}.
 \end{equation}

Given  $w \in \widetilde{W}$ and $\mathcal{X} = \sum_{i \in I} \delta_{w_i} \in \widetilde{\mathbf{W}}$, we define the \emph{traces} $\mathrm{Tr}(w) \subset \mathbb{Z}^d$ and $\mathrm{Tr}\left( \mathcal{X} \right) \subseteq \mathbb{Z}^d $ by
\begin{equation}
	\label{def:trace}
	\mathrm{Tr}(w):= \bigcup_{j=0}^{L(w)-1} \{w(j)\}, \qquad   \mathrm{Tr}\left( \mathcal{X} \right):= \bigcup_{i \in I} \mathrm{Tr}(w_i),
\end{equation}
thus $\mathrm{Tr}(w)$ is the set of sites visited by  $w$ and $\mathrm{Tr}\left( \mathcal{X} \right)$ is the set of sites visited by  $w_i, i \in I$.

Given  $w \in \widetilde{W}$ and $\mathcal{X} = \sum_{i \in I} \delta_{w_i} \in \widetilde{\mathbf{W}}$, we define the measures $\mu^w$ and $\mu^{\mathcal{X}}$ on $\mathbb{Z}^d$ by
\begin{equation}
	\label{def:local_time}
	\mu^w := \sum_{j=0}^{L(w)-1} \delta_{w(j)}, \qquad \mu^{\mathcal{X}}:= \sum_{i \in I}  \mu^{w_i},
\end{equation}
thus if $y \in \mathbb{Z}^d$ then $\mu^w(\{y\})$ equals the number of visits to $y$ by $w$ and $\mu^{\mathcal{X}}(\{y\})$ equals the total number of visits to $y$ by $w_i, i\in I$. We call $\mu^{\mathcal{X}}(\{y\})$ the \emph{local time} of $\mathcal{X}$ at $y$.
  Given some  $\mathcal{X} = \sum_{i \in I} \delta_{w_i} \in \widetilde{\mathbf{W}}$, we define the total length $\Sigma^{\mathcal{X}}$ of the worms $w_i, i \in I $ by
\begin{equation}
	\label{def:total_length}
	\Sigma^{\mathcal{X}}:=  \mu^{\mathcal{X}}(\mathbb{Z}^d)= \sum_{i \in I}  L(w_i).
\end{equation}
We also call $\Sigma^{\mathcal{X}}$ the \emph{total local time} of the point measure $\mathcal{X}$ of worms.

Recall the notion of the length distribution mass function $m: \mathbb{N} \to \mathbb{R}_+$ from Definition \ref{def_worms_nu}.

\begin{definition}[Law $\mathcal{P}_v$ of Poisson point process $\widehat{\mathcal{X}}$ of worms at level $v$] \label{def_full_worm_ppp} Given some $v \in \mathbb{R}_+$ and a probability mass function $m: \mathbb{N} \to \mathbb{R}_+$,
	a random element $\widehat{\mathcal{X}}$ of $\widetilde{ \mathbf{W} }$ has law $\mathcal{P}_v$ if $\widehat{\mathcal{X}}$ is a Poisson point process (PPP) on $\widetilde{W}$ with intensity measure $v \cdot m( L(w) ) \cdot (2 d)^{ 1 - L(w) }$.
\end{definition}

\begin{claim}[Old and new definitions of $\mathcal{S}^v$ coincide] \label{claim_ppp_S_v}
 The random length worms set $\mathcal{S}^v$ at level $v$ introduced in Definition \ref{def_worms_nu}
  can be alternatively defined as $\mathcal{S}^v:=\mathrm{Tr}( \widehat{\mathcal{X}} )$, where $ \widehat{\mathcal{X}} \sim \mathcal{P}_v$.
\end{claim}

\begin{lemma}[Expectation and variance of linear functionals of the worm PPP]
	\label{prop:one_worm_ppp_formulas}
	Given $\widehat{\mathcal{X}} = \sum_{i \in I} \delta_{w_i} \sim \mathcal{P}_v$ and a  function
$f \, : \, \widetilde{ W } \longrightarrow \R^+$, we have
	\begin{align}
		\label{eq:one_worm_ppp_expectation_formula}
		\mathbb{E} \left[ \sum_{i \in I} f(w_i) \right] & =
		v \cdot \sum_{ x \in \Z^d } \sum_{\ell = 1}^{\infty}  m(\ell) \cdot E_x \big[  f \left( X[0, \ell-1] \right)  \big], \\
		\label{eq:one_worm_ppp_variance_formula}
		\Var \left( \sum_{i \in I} f(w_i)  \right) & =
		v \cdot \sum_{x \in \Z^d} \sum_{\ell = 1}^{\infty} m(\ell) \cdot E_x \big[  f^2 \left( X[0, \ell - 1] \right)  \big].
	\end{align}
\end{lemma}
\begin{proof} Denote by $\mu$ the intensity measure of the PPP $\widehat{\mathcal{X}}$. The random variables $\widehat{\mathcal{X}}(\{w\}), w \in \widetilde{W}$ are independent and
$\widehat{\mathcal{X}}(\{w\}) \sim \mathrm{POI}(\mu\{w\})$ for each worm  $w \in \widetilde{W}$. From this we obtain
\begin{equation}\label{E_var_mu_formulas}
  \mathbb{E} \left[ \sum_{i \in I} f(w_i) \right]=\sum_{w \in \widetilde{W}}\mu(\{w\})f(w), \qquad   \Var \left( \sum_{i \in I} f(w_i)  \right)= \sum_{w \in \widetilde{W}}\mu(\{w\})f^2(w).
\end{equation}
By Definition \ref{def_full_worm_ppp} we may write $\mu=v \cdot \sum_{x \in \Z^d} \sum_{\ell = 1}^{\infty} m(\ell) P^\ell_x$, where $P^\ell_x$ denotes the law of $(X(0),\dots,X(\ell-1))$ under $P_x$, i.e.,  the law of the first $\ell-1$ steps of a random walk starting from $x$. From this and \eqref{E_var_mu_formulas}
the formulas \eqref{eq:one_worm_ppp_expectation_formula} and \eqref{eq:one_worm_ppp_variance_formula} follow.
\end{proof}

Recall the notation of $P_{x,x'}$ and $E_{x,x'}$ from Section \ref{subsection_random_walks}.

\begin{lemma}[Expectation of  bilinear functionals of the worm Poisson point process]
	\label{prop:two_worms_joint_ppp_formula}
	Given $\widehat{\mathcal{X}} = \sum_{i \in I} \delta_{w_i} \sim \mathcal{P}_v$ and a function $f \, : \, \widetilde{ W } \times \widetilde{ W }  \longrightarrow \R^+$, we have
	\begin{align}
		\label{eq:two_worms_joint_ppp_expectation_formula}
		\mathbb{E} \left[ \sum_{i, j \in I}  f( w_i, w_j ) \right] & =
		v^2 \cdot \sum_{x,x' \in \Z^d} \sum_{\ell, \ell' = 1}^{\infty}
		m(\ell) \cdot m(\ell') \cdot
		E_{x,x'} \left[ f \left( X[0, \ell-1], X'[0, \ell'-1] \right) \right]
		\notag \\ & \qquad +
		v \cdot \sum_{x \in \Z^d} \sum_{\ell = 1}^{\infty} m(\ell) \cdot E_x \left[ f \left( X[0, \ell-1], X[0, \ell-1] \right)  \right]
	\end{align}
\end{lemma}
\begin{proof}  Let us begin by noting that
$
 \sum_{i, j \in I}  f( w_i, w_j ) =
 \sum_{w, w' \in \widetilde{W}} \mathcal{X}(\{w\})\mathcal{X}(\{w'\}) f( w, w' ).
$

 Taking the expectation of both sides (and using the notation introduced in the proof of Lemma \ref{prop:one_worm_ppp_formulas}), we obtain
\begin{multline}
\mathbb{E} \left[ \sum_{i, j \in I}  f( w_i, w_j ) \right]= \sum_{w, w' \in \widetilde{W}} \mathbb{E} \left[\mathcal{X}(\{w\})\mathcal{X}(\{w'\})\right] f( w, w' )= \\
\sum_{w, w' \in \widetilde{W}}  \mu(\{w\})\mu(\{w'\})f( w, w' ) +\sum_{w \in \widetilde{W}} \mu(\{w\})f( w, w ).
\end{multline}
From this the formula \eqref{eq:two_worms_joint_ppp_expectation_formula} follows using $\mu=v \cdot \sum_{x \in \Z^d} \sum_{\ell = 1}^{\infty} m(\ell) P^\ell_x$.
\end{proof}

\section{Percolation of worms using multi-scale recursion}\label{section_perco_worms_multiscale}

From now on we assume  $d \geq 5$.

In Section \ref{subsection_good_seq_of_scales} we introduce the notion of a \emph{good sequence} $( R_n )_{n = 0}^{N+1}$ \emph{of scales} and formulate Theorem \ref{thm_worm_perco} which
states that the existence of a good sequence of scales implies that the random length worms percolation model is supercritical.
Theorem \ref{thm_worm_perco} is more complicated to state than our main result Theorem \ref{main_thm}, but it is stronger: in Section \ref{subsection_good_seq_of_scales}
we show that Theorem \ref{thm_worm_perco} implies Theorem \ref{main_thm}.

In Section \ref{subsection_input_packages} we introduce some notation and state Lemma \ref{lemma_doubling_gamma}, one of the main ingredients of the proof of
Theorem \ref{thm_worm_perco} (in Sections \ref{subsub_further_ideas_main_thm} and \ref{sub_structure_of_paper} we referred to Lemma \ref{lemma_doubling_gamma} as the ``capacity doubling lemma''). Loosely speaking, Lemma \ref{lemma_doubling_gamma} is the induction step that allows us to move one scale higher in a good sequence of scales.
Lemma  \ref{lemma_doubling_gamma} will be proved in Section \ref{subsection_doubling_gamma}.

In Sections  \ref{subsection_exploration} and \ref{subsec_coarse} we prove  Theorem \ref{thm_worm_perco} using Lemma \ref{lemma_doubling_gamma}. This argument involves
a \emph{coarse graining} of $\mathbb{Z}^d$ and \emph{dynamic renormalization}. We will  subdivide $\mathbb{Z}^d$ into boxes of scale $R_{N+1}$ (the highest scale on a good sequence of scales) and dynamically explore these boxes, starting from the origin: some of these boxes are ``good'', while some of them are ``bad'', and if the exploration of good boxes never terminates then the worms percolate.

In Section  \ref{subsection_exploration} we recall the notation and results of \cite{GrimmettMarstrand1990} about dynamic renormalization.

In Section  \ref{subsec_coarse} we define our coarse graining and show that the dynamic renormalization scheme  can be applied to prove percolation of worms (i.e., we prove Theorem \ref{thm_worm_perco}).

\subsection{Good sequence of scales}\label{subsection_good_seq_of_scales}

Recall the definition of the length variable $\mathcal{L}$ and its  mass function $m$ from Definition \ref{def_worms_nu}.

\begin{definition}[Good sequence of scales]\label{def_seq_of_good_scales} $ $
 Let $R^*_0,\gamma_0, \Dlow, \Dup, \underline{\alpha},  \psi, s, \Lambda, v \in (0,+\infty)$.
We say that an increasing sequence $( R_n )_{n = 0}^{N+1}$ of positive integers is a $(R^*_0,\gamma_0, \Dlow, \Dup, \underline{\alpha},  \psi, s, \Lambda, v )$-good sequence of scales for a probability measure $m$ on $\mathbb{N}$ if
	\begin{align}
\label{condition_Rzero}  R_0 & \geq R^*_0, &  \\
\label{condition_R_n_R_n1}		 \sum_{\ell = \Dlow \cdot R_n^2}^{ \Dup \cdot R_{n+1}^2 } \ell^2 \cdot m(\ell) & \geq \underline{\alpha}\, \slash  v, \quad & 0 \leq n \leq N,  \\
\label{condition_psi}		2^n \cdot \gamma_0 & \leq \psi \cdot R_n^{d-4}, \quad & 0 \leq n \leq N,  \\
 \label{condition_shoot}  2^{N+1} \cdot \gamma_0 \cdot  s \cdot v \cdot R_{N+1}^4 \cdot \sum_{\ell=\Lambda \cdot R_{N+1}^2}^{\infty} m(\ell)   & \geq 2.
	\end{align}
\end{definition}

\begin{remark} Note that condition \eqref{condition_psi} is in practice redundant, as we now explain. If we assume that
 $\Dlow/\Dup \geq 4$ holds then \eqref{condition_R_n_R_n1} implies $R_{n+1}/R_n \geq 2$, which in turn implies $2^n \cdot R^*_0 \leq R_n$, thus
 by choosing  $R^*_0$ big enough, we can conclude that  \eqref{condition_psi} holds. We decided to keep the condition \eqref{condition_psi} because we wanted to make the role of the constant $\psi$ explicit.
\end{remark}

\begin{remark} Let us compare the conditions of Definition \ref{def_seq_of_good_scales} with the ``ideal'' condition $\mathbb{E}[\mathcal{L}^2]=+\infty$.
 If
 $\mathbb{E}[\mathcal{L}^2]=+\infty$ then we can define $( R_n )_{n = 0}^{\infty}$ so that $R_n <+\infty$  and  \eqref{condition_R_n_R_n1} and satisfied  for each $n$, but
  $\mathbb{E}[\mathcal{L}^2]=+\infty$ is not enough for \eqref{condition_shoot} to hold for some $N$, see Remark \ref{remark_if_epsilon_negative}.
\end{remark}
Recall from Section \ref{subsection_methods} our plan which involved \emph{fattening} and \emph{target shooting}: we will use condition \eqref{condition_shoot}
when we perform the target shooting in Section \ref{subsec_coarse} and we will use condition \eqref{condition_R_n_R_n1} when we perform the fattening in order to achieve \emph{capacity doubling} in Section \ref{subsection_doubling_gamma}.

\begin{theorem}[Good sequence of scales implies supercritical  percolation]\label{thm_worm_perco} Let $d \geq 5$.
 There exist constants $R^*_0, \gamma_0, \Dlow, \Dup, \underline{\alpha}, \psi, s, \Lambda \in (0,+\infty)$
such that for any $v \in (0,+\infty)$ and any probability measure $m$ on $\mathbb{N}$ the following implication holds. If there exists an
 $(R^*_0, \gamma_0, \Dlow, \Dup, \underline{\alpha}, \psi, s, \Lambda, v )$-good sequence $( R_n )_{n = 0}^{N+1}$ of scales for $m$ then
 \begin{equation}
 \mathbb{P}(\, \mathcal{S}^v \text{ percolates}\, )=1.
 \end{equation}
\end{theorem}

Before we prove Theorem \ref{thm_worm_perco}, let us derive Theorem \ref{main_thm} from it. We only need:

\begin{lemma}[A distribution with a good sequence of  scales]\label{lemma_fav_distr_good_seq_of_scales} Let $\varepsilon>0$ and $\ell_0 \geq e^e$. If
\begin{equation}\label{loglog_epsilon}
  m(\ell)= c \frac{\ln(\ln(\ell))^{\varepsilon}}{\ell^3 \ln(\ell)} \mathds{1}[\ell \geq \ell_0], \qquad \ell \in \mathbb{N},
\end{equation}
 then for any $R^*_0,\gamma_0, \Dlow, \Dup, \underline{\alpha},  \psi, s, \Lambda, v \in (0,+\infty)$
 there exists an
 $(R^*_0, \gamma_0, \Dlow, \Dup, \underline{\alpha}, \psi, s, \Lambda, v )$-good sequence $( R_n )_{n = 0}^{N+1}$ of scales for $m$.
\end{lemma}

\begin{proof} Let us fix  $R^*_0,\gamma_0, \Dlow, \Dup, \underline{\alpha},  \psi, s, \Lambda, v \in (0,+\infty)$.
 The values of these parameters will be treated as constants for the rest of this proof.
 Let us fix any $ \delta \in (0,\varepsilon)$ and  define
\begin{equation}\label{def_R_n_expexp}
\widetilde{R}_n:= \exp\left( \exp\left( n^{1/(1+\delta)} \right) \right), \qquad n \in \mathbb{N}_0.
\end{equation}
 Let us also define $m(x)=c \frac{\ln(\ln(x))^{\varepsilon}}{x^3 \ln(x)} \mathds{1}[x \geq \ell_0]$ for any
$x \in \mathbb{R}_+$ and choose $\ell^*_0 \geq \ell_0$ so that $x \mapsto x^2 \cdot m(x)$ is a decreasing function on $[\ell^*_0,+\infty)$.
For any $\ell^*_0 \leq a \leq b \in \mathbb{N}$ we have
\begin{equation}\label{sum_int_bound_loglog}
  \sum_{\ell=a}^{b}  \ell^2 \cdot m(\ell) \geq \int_a^b x^2 \cdot m(x)\, \mathrm{d} x = \frac{1}{1+\varepsilon}\left(  \ln(\ln(b))^{1+\varepsilon}- \ln(\ln(a))^{1+\varepsilon}  \right).
\end{equation}
Observe that $ \lim_{n \to \infty} \frac{\widetilde{R}_n}{2^n}=+\infty$ by \eqref{def_R_n_expexp} and
 $ \lim_{n \to \infty} \sum_{\ell = \Dlow \cdot \widetilde{R}_n^2}^{ \Dup \cdot \widetilde{R}_{n+1}^2 } \ell^2 \cdot m(\ell)=+\infty$  by \eqref{def_R_n_expexp}, \eqref{sum_int_bound_loglog} and the fact that $\delta<\varepsilon$, therefore we can choose (and fix) $n_0$ big enough so that if we define $R_n:= \widetilde{R}_{n+n_0}$ then
 \eqref{condition_Rzero} holds, moreover \eqref{condition_R_n_R_n1}	 and \eqref{condition_psi} hold for all $n \in \mathbb{N}$.

 For any $\ell^*_0 \leq a \in \mathbb{N}$ we have
\begin{equation}\label{sum_m_tail_lower_bound}
\sum_{\ell=a}^{\infty} m(\ell) \stackrel{ \eqref{loglog_epsilon} }{\geq} c \sum_{\ell=a}^{2a}   \frac{1}{\ell^3 \ln(\ell)} \geq \frac{c}{\ln(2a)} \int_a^{2a} \frac{1}{x^3} \, \mathrm{d}x \geq  \frac{c'}{a^2 \ln(a) }.
\end{equation}
Now we are ready to check that \eqref{condition_shoot} holds if $N$ is large enough. First let us bound
\begin{equation}\label{shooting_bound_for_fav}
2^{n} \cdot  R_{n}^4 \cdot \sum_{\ell=\Lambda \cdot R_{n}^2}^{\infty} m(\ell) \stackrel{\eqref{sum_m_tail_lower_bound}}{\geq} 2^{n} \cdot  R_{n}^4 \cdot
 \frac{c'}{(\Lambda \cdot R_{n}^2)^2 \ln(\Lambda \cdot R_{n}^2) } \stackrel{\eqref{def_R_n_expexp} }{\geq} c \frac{2^n}{\exp\left( (n+n_0)^{1/(1+\delta)} \right) }.
\end{equation}
The r.h.s.\ of \eqref{shooting_bound_for_fav} goes to infinity as $n \to \infty$ since $0<\delta$, thus \eqref{condition_shoot} holds if $N$ is large enough. In conclusion, $( R_n )_{n = 0}^{N+1}$ is an $(R^*_0, \gamma_0, \Dlow, \Dup, \underline{\alpha}, \psi, s, \Lambda, v )$-good sequence  of scales for $m$ by Definition \ref{def_seq_of_good_scales}. The proof of
Lemma \ref{lemma_fav_distr_good_seq_of_scales}  is complete.
\end{proof}

Theorem \ref{main_thm} follows from Lemma \ref{lemma_fav_distr_good_seq_of_scales} and Theorem \ref{thm_worm_perco}. It remains to prove Theorem \ref{thm_worm_perco}.

\begin{remark}\label{remark_if_epsilon_negative} Note that if we consider a probability mass function $m(\cdot)$ of form \eqref{loglog_epsilon} with $\varepsilon \in (-1,0)$
then $\sum_{\ell=1}^{\infty} \ell^2 \cdot m(\ell)=+\infty$ still holds by \eqref{sum_int_bound_loglog}, but a slight modification of the proof of Lemma \ref{lemma_fav_distr_good_seq_of_scales} shows that for any $R^*_0,\gamma_0, \Dlow, \Dup, \underline{\alpha},  \psi, s, \Lambda, v \in (0,+\infty)$ one
can choose $\ell_0$ big enough so that there is no $(R^*_0, \gamma_0, \Dlow, \Dup, \underline{\alpha}, \psi, s, \Lambda, v )$-good sequence  of scales for $m$.
\end{remark}

The rest of Section \ref{section_perco_worms_multiscale} is devoted to the proof of Theorem \ref{thm_worm_perco}.

\subsection{Input packages}\label{subsection_input_packages}

The goal of Section \ref{subsection_input_packages} is to state Lemma \ref{lemma_doubling_gamma},  one of the main ingredients of the proof of Theorem \ref{thm_worm_perco}.
 Lemma \ref{lemma_doubling_gamma} only involves (i) a PPP $\worms$ of worms  shorter than  $R^2$ which emanate from a box of spatial scale $R$ and (ii) a distinguished worm $\pinkwalk$ from a larger length scale. Together, the pair $(\worms, \pinkwalk)$ is called an \emph{input package}.

Recall from Definition \ref{def_full_worm_ppp} the notion of the law $\mathcal{P}_v$ of the PPP $\widehat{ \mathcal{X} }$ on the space of worms.

\begin{definition}[Natural law on input packages]
	\label{def:law_on_input_packages}
	Let $y \in \mathbb{Z}^d$, $R \in \mathbb{N}$, $z \in \ball(y,R)$, $v \in \mathbb{R}_+$. We say that a random element $(\worms, \pinkwalk)$ of $\widetilde{\mathbf{W}} \times \widetilde{W}$ has law $\mathcal{P}_{y,R,z,v}$ if
	\begin{enumerate}[(i)]
		\item $\mathcal{X}$ has the law of $\widehat{\mathcal{X}} \ind[\, w(0) \in \ball(y, 2 R),\, L(w) \leq R^2\, ]$, where $\widehat{ \mathcal{X} } \sim \mathcal{P}_v$
(i.e., $\mathcal{X}$ is a PPP on $\widetilde{W}$ with intensity measure $v \cdot m( L(w) ) \cdot (2 d)^{ 1 - L(w) }\ind[\, w(0) \in \ball(y, 2 R),\, L(w) \leq R^2\, ]$).
		\item $\pinkwalk=(Z(t))_{0\leq t \leq T}$ is an independent $d$-dimensional simple symmetric random walk starting from $Z(0)=z$ that we run up to time  $T=\min\{ R^2, T_{\ball(y,2R)^c}(\pinkwalk) \}$, where $T_{\ball(y,2R)^c}(\pinkwalk)$ denotes the first time the path $(Z(t))_{t \geq 0}$ exits from the box $\ball(y,2R)$ (thus $L(\pinkwalk)=T+1$).
	\end{enumerate}
\end{definition}

The law $\mathcal{P}_{y,R,z,v}$ is invariant under the translations of $\mathbb{Z}^d$ in the following sense:
\begin{equation}\label{shift_invar_package}
\text{if $(\worms, \pinkwalk) \sim \mathcal{P}_{y,R,z,v}$ then $(\worms+x, \pinkwalk+x) \sim \mathcal{P}_{y+x,R,z+x,v}$ for any $x \in \mathbb{Z}^d$.}
\end{equation}

\begin{definition}[$(y,R,z,\gamma)$-good input package]
	\label{def:gamma_good_input_package}
	Let $y \in \mathbb{Z}^d$,  $R \in \mathbb{N}$, $z \in \ball(y,R)$ and $\gamma \in \mathbb{R_+}$. We say that $(\worms, \pinkwalk) \in \widetilde{\mathbf{W}}\times \widetilde{W}$ is $(y,R,z,\gamma)$-good if there is a set $H \subset \mathbb{Z}^d $ which satisfies the following properties:
	\begin{equation}
	\begin{array}{ll}  	 \text{(i) } H \subseteq \ball(y,3R) \cap (\mathrm{Tr}(\worms)\cup \mathrm{Tr}(\pinkwalk)), \qquad
		& \text{(ii) } H \text{  is connected,} \\
		 \text{(iii) }  z \in H, \qquad
		& \text{(iv) } R^2 \cdot \gamma \leq \capacity(H) \leq R^2 \cdot \gamma + 1.
\end{array}
	\end{equation}
We call such a set $H$ an $(y,R,z,\gamma)$-\emph{good set} for $(\worms, \pinkwalk)$.
We call the parameter $\gamma$ the \emph{capacity-to-length ratio} (since $R^2$ is the maximal length of $\pinkwalk$).
\end{definition}

Due to  translation invariance, we have
\begin{equation}
	\mathcal{P}_{y, R, z,v} ( (\worms, \pinkwalk) \text{ is } (y,R,z,\gamma) \text{-good})=\mathcal{P}_{o, R, z-y,v} ( (\worms, \pinkwalk) \text{ is } (o,R,z-y,\gamma) \text{-good}).
\end{equation}

Our next lemma initializes our recursive construction of good sets.
Let us stress that it works for any $v \geq 0$ (even in the case $v=0$), because we  only use $\pinkwalk$ (i.e., we do not use $\worms$).

\begin{lemma}[Initializing the $\gamma$ of good input packages]\label{lemma_initializing_gamma}
There exist dimension-dependent positive finite constants $R^*_0$ and $\gamma_0$ such that for any $R \geq R^*_0$ and any $v \geq 0$  we have
\begin{equation}
	\label{initializing_gamma}
	\min_{z \in  \ball(R)}
	\mathcal{P}_{o,R,z,v}\left(\, (\worms, \pinkwalk) \text{ is $(o,R,z,\gamma_0)$-good}\,  \right) \geq
	3 \slash 4.
	\end{equation}
\end{lemma}

\begin{proof} Let $(\worms, \pinkwalk) \sim \mathcal{P}_{o,R,z,v}$.
 Recall from Definition \ref{def:law_on_input_packages} that  $T=\min\{ R^2, T_{\ball(y,2R)^c}(\pinkwalk) \}$.
First note that $H:=\mathrm{Tr}(\pinkwalk)$ satisfies conditions (i), (ii) and (iii) of Definition \ref{def:gamma_good_input_package}.
 By  Proposition \ref{staying_inside} and our assumption $z \in \ball(R)$ we can  choose (and fix) an $\veps > 0$ such that $\mathbb{P}(\veps \cdot R^2 \leq T) \geq 7/8$  holds for any $R \in \mathbb{N}$.
 By Proposition \ref{prop_lln_cap_rw} there exists $R^*_0 \in \mathbb{N}$ such that
\begin{equation}
\mathbb{P}\left[ \,\capacity\left( \mathrm{Tr}(Z[0,\varepsilon \cdot R^2])\right) < \tfrac{1}{2} \cdot e_\infty \cdot \veps \cdot R^2\, \right] \leq 1/8, \qquad R \geq R^*_0.
\end{equation}
Let $\gamma_0:=\tfrac{1}{2} \cdot e_\infty \cdot \veps$.
Putting together the above  with the monotonicity  of capacity (cf.\ \eqref{cap_mon}),  we obtain that if $R \geq R^*_0$ then $H:=\mathrm{Tr}(\pinkwalk)$ satisfies $\mathbb{P}(\capacity(H)\geq R^2 \cdot \gamma_0) \geq 7/8-1/8=3/4$. In order to guarantee that $H$ also satisfies  condition
(iv) of Definition \ref{def:gamma_good_input_package}, we may use Lemma \ref{lemma_trimming_the_fat} to replace $H$ with an appropriate subset of $H$.
 The proof of \eqref{initializing_gamma} is complete.
\end{proof}

The next lemma is the main ingredient of our recursive multi-scale construction: given that we are able to construct good sets with a capacity-to-length ratio $\gamma$ at a lower scale $\Rs$ then
we can use fattening to construct good sets with a capacity-to-length ratio $2\gamma$ at a higher scale $\Rb$.

\begin{lemma}[Doubling the $\gamma$ of good input packages by fattening]
	\label{lemma_doubling_gamma}
	There exist dimension-dependent positive finite parameters
	\begin{equation}
		\label{parameter_assumptions}
		\psi < 1, \quad \Dup < 1, \quad \Dlow > 1, \quad \underline{\alpha} > 1
	\end{equation}
	such that for any $\Rs \leq \Rb \in \mathbb{N}$, any $\gamma \in \mathbb{R}_+$  satisfying
\begin{equation} \label{gamma_ineq_assumptions}
 \gamma_0 \leq \gamma \leq \psi \cdot \Rs^{d-4}
\end{equation}
(where $\gamma_0$ appears in Lemma \ref{lemma_initializing_gamma}), any choice of $v \in \mathbb{R}_+$ and any choice of the probability mass function $m$ of $\mathcal{L}$ satisfying the inequality
	\begin{equation}
		\label{def:alpha}
		v \cdot \mathbb{E} \left[\, \mathcal{L}^2 \, \ind \left[ \Dlow \cdot \Rs^2 \leq \mathcal{L} \leq \Dup \cdot \Rb^2 \right]\, \right] = v \cdot \sum_{\ell = \Dlow \cdot \Rs^2}^{ \Dup \cdot \Rb^2 } \ell^2 \cdot m(\ell) \geq \underline{\alpha},
	\end{equation}
	the following implication holds: if
	\begin{equation}
	\label{induction_hypotheses}
	\min_{z \in  \ball(\Rs)}
	\mathcal{P}_{o,\Rs,z,v}\left(\, (\worms, \pinkwalk) \text{ is $(o,\Rs,z,\gamma)$-good}\,  \right) \geq
	3 \slash 4,
	\end{equation}
	then
	\begin{equation}
	\label{2gamma_want}
	\min_{z \in  \ball(\Rb)}
	\mathcal{P}_{o,\Rb,z,v}\left(\, (\worms, \pinkwalk) \text{ is $(o,\Rb,z,2\gamma)$-good} \, \right) \geq
	3 \slash 4.
	\end{equation}
\end{lemma}
The proof of Lemma \ref{lemma_doubling_gamma} is postponed to Section \ref{subsection_doubling_gamma}.

\begin{remark} The choice of the parameters $\psi,  \Dup,  \Dlow , \underline{\alpha}$  will be explained in Section \ref{subsection_doubling_gamma}, but let us already clear some of the mystery surrounding the values of $\Dup$ and $\Dlow$: these two parameters are in fact determined by the requirements on the random walk's length from Section \ref{subsection_rw_estimates}
(see Lemmas \ref{upper:fuga_hitting_sum}, \ref{lemma_upper:two_cons_hitting_sum}, \ref{lower:fuga_hitting_sum})
except for one additional constraint
 $\underline{\Delta} \geq \CLowerNOVisitedBoxesRATIOLOWER \vee (300/ \CLowerNOVisitedBoxesLOWER ) $ (where $ \CLowerNOVisitedBoxesLOWER $ and $ \CLowerNOVisitedBoxesRATIOLOWER $ appeared in the statement of Lemma \ref{lower:no_of_visited_boxes})
 that will be required in  Lemma \ref{lower:no_of_visited_good_boxes}.
\end{remark}

In the remainder of Section \ref{section_perco_worms_multiscale} we prove  Theorem \ref{thm_worm_perco} using Lemma \ref{lemma_doubling_gamma} to go one scale higher
in a good sequence of scales (i.e., we will choose
$\Rs=R_n$ and $\Rb=R_{n+1}$). However, the proof of Theorem \ref{thm_worm_perco} also requires further ingredients that we discuss in Section \ref{subsection_exploration}.

\subsection{Dynamic renormalization}\label{subsection_exploration}

Let us recall some notation from  \cite[Section 3]{GrimmettMarstrand1990}. Note that the $\mathbb{Z}^d$ that appears in Section \ref{subsection_exploration}
should be imagined as a coarse grained lattice (where the coarse graining will only be defined in Section \ref{subsec_coarse}):  each good/bad vertex that we define in
Section \ref{subsection_exploration} will correspond to a good/bad box in Section \ref{subsec_coarse}, the scale of boxes being $R_{N+1}$, where
$( R_n )_{n = 0}^{N+1}$ is a good sequence of scales.

Let us  fix an ordering $e_1 \prec e_2 \prec \ldots$ on the set of nearest neighbour edges of the lattice $\Z^d$. We are going to define an exploration process which produces a random connected subset of  \emph{good} vertices that contains the origin. Some vertices will turn out to be \emph{bad}: our exploration process is not allowed to use such vertices as a stepping-stone for further exploration. The goal of the construction (see Proposition \ref{lemma:Grimmett_Marstrand} below) is to give a sufficient condition which guarantees that the good connected cluster of the origin is infinite with positive probability.

We will  define a collection of $\{ 0, 1 \}$-valued random variables $\{ \sigma(t) \, : \, t \in \mathbb{N}_0 \}$, where $\sigma(t)$ is the indicator of the event that the vertex that we explore at the $t$'th step is good.

Let us define $(S_t)_{t \in \mathbb{N}_0}$, where $S_t=(G_t,B_t)$ and $G_t \subseteq \mathbb{Z}^d$ denotes the \emph{set of good vertices explored by time} $t$, while
$B_t \subseteq \mathbb{Z}^d$ denotes the \emph{set of bad vertices explored by time} $t$.

For the initial step, let us examine the origin and define $S_0 := (G_0, B_0)$ accordingly, as
\begin{align*}
	S_0 :=
	\begin{cases}
		( \{ o \}, \emptyset ), & \text{if} \quad \sigma(0)  = 1, \\
		(\emptyset, \{ o \}), & \text{if} \quad \sigma(0) = 0.
	\end{cases}
\end{align*}
If $S_0, S_1, \ldots, S_t$ are already given up to some $t \in \mathbb{N}_0$,  let us define $S_{t+1}$ as follows. Consider the edges of $\Z^d$ that have one endpoint in $G_t$ and the other in $\Z^d \setminus (G_t \cup B_t)$. If there are no such edges then $\partial^{\text{ext}} G_t \subseteq B_t$ and we define $S_{t+1} := S_t$. If there are such edges, let us pick the minimal such edge $e_{t+1}$ with respect to the ordering $\prec$.  Denote by $x_{t+1}$ the endpoint of $e_{t+1}$ in $\Z^d \setminus (G_t \cup B_t)$ and denote by $x'_{t+1}$  the endpoint of $e_{t+1}$ in $G_t$. Let
\begin{align}\label{S_t_plus_1_def}
	S_{t+1} :=
	\begin{cases}
		( G_t \cup \{ x_{t+1} \}, B_t ), & \text{if} \quad \sigma(t+1)  = 1, \\
		( G_t, B_t \cup \{ x_{t+1} \}), & \text{if} \quad \sigma(t+1) = 0.
	\end{cases}
\end{align}
 Note that $S_t$, $e_{t+1}$, $x_{t+1}$ and  $x'_{t+1}$ are determined by $\sigma(0),\dots,\sigma(t)$.

 The sequences  $\{ G_t\}_{t \geq 0}$ and $\{ B_t \}_{t \geq 0}$ of sets are non-decreasing w.r.t.\ set inclusion, hence we can define the  limits
\begin{equation*}
	G_{\infty} := \lim_{t \to \infty} G_t,
	\qquad \qquad
	B_{\infty} := \lim_{t \to \infty} B_t.
\end{equation*}
Note that $|G_{\infty}|<+\infty$ if and only if either $G_0=\emptyset$ or there exists $t \in \mathbb{N}_0$ such that $\partial^{\text{ext}} G_t \subseteq B_t$.

Recall that we denote by $ p_c^{\text{site}}(\Z^d)$ the critical threshold of Bernoulli site percolation on the nearest neighbour lattice $\mathbb{Z}^d$.

The following proposition is \cite[Lemma 1]{GrimmettMarstrand1990}.

\begin{proposition}[Sufficient condition for eternal exploration]
	\label{lemma:Grimmett_Marstrand}
	If there exists a constant $0 < \CGrimmettMarstrand < 1$ such that $\CGrimmettMarstrand > p_c^{\text{site}}(\Z^d)$, moreover
$\mathbb{P} \left( \sigma(0) = 1  \right) \geq \CGrimmettMarstrand$ and
	\begin{equation}\label{condition_bernoulli_dom}
		\mathbb{P} \left( \sigma(t+1) = 1 \, | \, \sigma(0), \sigma(1), \ldots, \sigma(t) \right) \geq \CGrimmettMarstrand
	\end{equation}
	hold for all $t \in \mathbb{N}_0$,  then we have $\mathbb{P}( | G_{\infty}| = \infty ) > 0$.
\end{proposition}

\subsection{Good sequence of scales implies supercritical  percolation}\label{subsec_coarse}

For the rest of Section \ref{section_perco_worms_multiscale}  we assume that we are given $\widehat{\worms} \sim \mathcal{P}_v$ (cf. Definition \ref{def_full_worm_ppp}), for some $v > 0$ and $m$.
Recall from Claim \ref{claim_ppp_S_v} that $\mathcal{S}^v=\mathrm{Tr}( \widehat{\mathcal{X}} )$. We want to apply Proposition \ref{lemma:Grimmett_Marstrand} in order to prove Theorem \ref{thm_worm_perco}. Loosely speaking, we will subdivide $\mathbb{Z}^d$ into boxes of scale $R$ for some $R \in \mathbb{N}$ and we will also subdivide the PPP $\widehat{\mathcal{X}}$ of worms into independent packages, where a particular package contains the worms emanating from a particular box of scale $R$.

We will build a  cluster of good boxes where each good box contains a \emph{seed}  (i.e., a connected subset of $\mathcal{S}^v$  with big capacity), moreover these seeds will all be subsets of the same connected component of $\mathcal{S}^v$.
In our construction the indicator $\sigma(t+1)$ that appears in Proposition \ref{lemma:Grimmett_Marstrand} will denote the indicator of the event that the new box  explored in step $t+1$ contains an appropriate seed which is also connected to the seed of an already explored neighbouring good box by worms emanating from the new box.

In order to make this plan more precise, we need to introduce some notation.

Let us introduce the scale $R \in \N$ and a length truncation parameter $\beta \in \mathbb{R}_+$, the values of which will be fixed later on.
 Given a vertex $y \in 10 \cdot R \cdot \Z^d$, let us introduce the subset $B[y] \subseteq \widetilde{W}$ of worms
\begin{equation}\label{def_B_y_package_worms}
B[y] := \left\{ \, w(0) \in \ball(y, 2 R), \, L(w) > (2 \beta + 1) R^2 \, \right\}.
\end{equation}
Given some $y \in 10 \cdot R \cdot \Z^d$ and a set $H \subseteq \mathbb{Z}^d$ (where in later applications $H$ will be a realization of the seed set in a good box indexed by
some $y' \in 10 \cdot R \cdot \Z^d $ satisfying $|y-y'|=10 \cdot R$), let us define the subset $B[y,H] \subseteq \widetilde{W}$ of worms by
\begin{equation}\label{back_and_forth}
B[y,H] :=  B[y] \cap \left\{  \,
T_{H}(w) \leq \beta R^2, \,
T_{\ball(y, R)}( w[ T_{H}(w), L(w)-1 ] ) \leq \beta R^2 \,
\right\}.
\end{equation}
In words: a worm $w \in  B[y]$ belongs to $B[y,H]$ if it hits $H$ in at most $\beta R^2$ steps and after that it returns to $\ball(y, R)$ in at most   $\beta R^2$
steps, and therefore it still has at least $R^2$ more steps to spare after hitting $\ball(y, R)$, since by \eqref{def_B_y_package_worms} the length of $w$ is at least $\beta R^2+\beta R^2+R^2$.
 This final segment of $w$ of length $R^2$ will serve as a nucleus that will later be fattened with the aim of producing a  seed with big enough capacity in the box indexed by $y$. Also note that $B[y,H]$ contains worms that start in the vicinity of $y$,  hit the target $H$ and then return to the vicinity of $y$ again (just like a boomerang).

Our next lemma provides us with conditions under which $\widehat{\mathcal{X}}\mathds{1}[ B[y,H]]$ is non-empty with high enough probability.

\begin{lemma}[Target shooting with a boomerang]
	\label{lemma:good_shooting}
	There exist $s > 0$ and $\beta \in \N$ such that for every $R \in \N$, $y \in \mathbb{Z}^d$ and every $H \subseteq \ball(y, 13 R)$ satisfying
	\begin{equation}
	\label{eq:good_shooting_capacity_lower_bound}
	\capacity(H) \cdot  s \cdot v \cdot R^2 \cdot \mathbb{P} \left( \mathcal{L} \geq (2 \beta + 1) R^2 \right)   \geq 2,
	\end{equation}
	we have
	\begin{equation}
	\mathcal{P}_v \left( \widehat{\mathcal{X}}( B[y,H] ) > 0 \right) \geq 3 \slash 4.
	\end{equation}
\end{lemma}
We will prove Lemma \ref{lemma:good_shooting} in Section \ref{subsection_boomerang}.

In order to implement the fattening that we mentioned in the paragraph under equation \eqref{back_and_forth}, we need to define an input package  (cf.\ Definition \ref{def:law_on_input_packages}).

\begin{notation}[Input package of scale $R$, centered at $y$,  connected to $H$]
\label{nota:shooting_steps}
Given $y \in 10\cdot R \cdot \Z^d$,  let us define the point process $\mathcal{X}^{y}$ of worms by
	\begin{equation}\label{new_X_y}
		\mathcal{X}^{y} := \widehat{\mathcal{X}} \ind \left[\, w(0) \in \ball(y, 2 R), \, L(w) \leq R^2 \, \right].
	\end{equation}
Furthermore, given a set $H \subseteq \mathbb{Z}^d$, on the event $\{ \widehat{ \mathcal{X} }( B[y, H] ) > 0 \}$ there exist at least one worm satisfying the properties described in \eqref{back_and_forth}, therefore we may use an independent source of randomness to choose one such  worm uniformly at random. Let us denote this chosen worm by $\widehat{\pinkwalk}^{y,H}$ and introduce
	\begin{align}
		\widehat{T}^{y,H} & := T_H\left(\widehat{\pinkwalk}^{y,H}\right)+ T_{\ball(y,R)} \left( \widehat{\pinkwalk}^{y,H} \left[ T_H\left(\widehat{\pinkwalk}^{y,H}\right), L \left( \widehat{\pinkwalk }^{y,H} \right) - 1 \right] \right), \\
	\label{z_y_H}	z^{y,H} & := \widehat{\pinkwalk}^{y,H} \left( \widehat{T}^{y,H} \right), \\
		T^{y, H} & := R^2 \wedge \min \left\{ t \, : \, \widehat{ \pinkwalk }^{y,H} \left( \widehat{T}^{y, H} + t \right) \notin \ball(y, 2 R) \right\}, \\
	\label{pinkwalk_for_exploration}	\pinkwalk^{y, H} & := \widehat{ \pinkwalk }^{y,H} \left[ \widehat{T}^{y, H}, \widehat{T}^{y,H} + T^{y,H} \right].
	\end{align}
\end{notation}
\noindent In words: $\widehat{T}^{y,H}$ is the first time $\widehat{\pinkwalk}^{y,H}$ hits $\ball(y,R)$ after hitting $H$, $z^{y,H}$ is the location where it hits
$\ball(y,R)$, $T^{y, H}$ is the minimum of $R^2$ and the time that $\widehat{\pinkwalk}^{y,H}$ spends in $\ball(y, 2 R)$ after $\widehat{T}^{y,H}$,
and $\pinkwalk^{y, H}$ is the sub-worm of $\widehat{\pinkwalk}^{y,H}$ of length $ T^{y,H}$ that starts at time $\widehat{T}^{y,H}$.
 Let us stress that  $\pinkwalk^{y, H}$ is indeed a sub-worm of $\widehat{\pinkwalk}^{y,H}$,  since
$ \widehat{T}^{y,H} + T^{y,H}$ is less than or equal to the length of $\widehat{\pinkwalk}^{y,H}$, since $\widehat{\pinkwalk}^{y,H} \in B[y,H]$ (see the discussion below \eqref{back_and_forth}).

The pair $(\mathcal{X}^{y},\pinkwalk^{y, H})$ will play the role of the input package which will be used to produce a seed  in the box of scale $R$ indexed by $y$.

\begin{assumption}[Good sequence of scales exist for $m$] \label{assumption_good_seq_scales}
 Let us choose the constants $R^*_0, \gamma_0 \in (0,+\infty)$ as in Lemma \ref{lemma_initializing_gamma}.
Let us choose the constants $\Dlow, \Dup, \underline{\alpha}, \psi \in (0,+\infty) $ as in Lemma \ref{lemma_doubling_gamma}. Let us choose the constants
  $s, \beta \in (0,+\infty)$ as in Lemma \ref{lemma:good_shooting}. Let $\Lambda:=2\beta+1$. Given our
   $v \in (0,+\infty)$ and our probability mass function $m$, let us assume that there exists an
 $(R^*_0, \gamma_0, \Dlow, \Dup, \underline{\alpha}, \psi, s, \Lambda, v )$-good sequence $( R_n )_{n = 0}^{N+1}$ of scales for $m$ (see Definition \ref{def_seq_of_good_scales}).
\end{assumption}

For Definitions \ref{notation_sigma_null} and \ref{notation_sigma_tplus1}, Lemma \ref{lemma_sigma_onehalf} and the proof of Theorem \ref{thm_worm_perco} below we assume
that $R=R_{N+1}$, where  $( R_n )_{n = 0}^{N+1}$ is the good sequence of scales of Assumption \ref{assumption_good_seq_scales}.

We are ready to define the indicators $\sigma(t), t \in \mathbb{N}_0$ for which we will apply Proposition \ref{lemma:Grimmett_Marstrand}.

\begin{notation}[Indicator $\sigma(0)$ of the presence of an initial good seed]\label{notation_sigma_null}
Let us define the set $H_*:=\ball( 3 R_{N+1})$ and  the indicator variable
\begin{equation}\label{sigma_null_def}
 \sigma(0):=\mathds{1}\left[\, \widehat{\mathcal{X}}( B[o, H_* ] ) > 0 \, \right] \cdot
 \mathds{1}\left[\, \left( \mathcal{X}^{o}, \pinkwalk^{o, H_*}  \right) \text{is } (o, R_{N+1}, z^{o,H_*}, 2^{N+1}\cdot \gamma_0  ) \text{-good}\, \right].
\end{equation}
If $\sigma(0)=1$, let us denote by $H_o$  the $(o, R_{N+1}, z^{o,H_*}, 2^{N+1}\cdot \gamma_0  )$-good set induced by the good input package $\left( \mathcal{X}^{o}, \pinkwalk^{o, H_*}  \right)$ (see Definition \ref{def:gamma_good_input_package}). We call $H_o$ the seed indexed by the origin.
\end{notation}
\noindent In words: $\sigma(0)$ is the indicator of the event that we can construct a seed $H_o$ in the $R_{N+1}$-neighbourhood of the origin.

Let us now assume that we have already defined the indicators $\sigma(0),\dots, \sigma(t)$ for some $t \in \mathbb{N}_0$. Our goal is to define $\sigma(t+1)$.
Given the values of the random variables $\sigma(0),\dots, \sigma(t)$, the exploration process described in Section \ref{subsection_exploration} determines the good set $G_t$ and the bad set $B_t$.

\begin{notation}[Indicator $\sigma(t+1)$ of the presence of the next seed  connected to an earlier seed]\label{notation_sigma_tplus1} If  $\partial^{\text{ext}} G_t \subseteq B_t$ then we know from Section \ref{subsection_exploration} that $G_\infty=G_t$ (i.e., the exploration terminates in finitely many steps), in which case we define $\sigma(t+1):=1$.

  Assume now that this is not the case, i.e., there is an edge connecting $G_t$ with $\Z^d \setminus (G_t \cup B_t)$. Recall that we denote by $e_{t+1}$  the minimal such edge  with respect to the ordering $\prec$, moreover we denote $e_{t+1}=\{x'_{t+1},x_{t+1}\}$, where $x'_{t+1} \in  G_t$ and
$x_{t+1} \in \Z^d \setminus (G_t \cup B_t)$.
Let us denote
\begin{equation}
y'_{t+1}:=10 \cdot R_{N+1}\cdot x'_{t+1}, \qquad  y_{t+1}:=10 \cdot R_{N+1}\cdot x_{t+1}.
\end{equation}
 Let us also assume that we have already defined the seed $H_{y'_{t+1}}$ indexed by $y'_{t+1}$, where
 $H_{y'_{t+1}}$ is an  $(y'_{t+1}, R_{N+1}, z, 2^{N+1}\cdot \gamma_0  )$-good subset of $\mathcal{S}^v$ for some $z \in \ball(y'_{t+1},R_{N+1}) $.
Now we can define the indicator
\begin{multline}\label{sigma_tplus1_def}
 \sigma(t+1):=\mathds{1}\left[ \,   \widehat{\mathcal{X}}( B[y_{t+1},H_{y'_{t+1}}] ) > 0  \, \right] \cdot \\
 \mathds{1}\left[ \, \left( \mathcal{X}^{y_{t+1}}, \pinkwalk^{y_{t+1}, H_{y'_{t+1}} }  \right) \text{is } (y_{t+1}, R_{N+1}, z^{y_{t+1},H_{y'_{t+1}}}, 2^{N+1}\cdot \gamma_0  ) \text{-good}\, \right].
\end{multline}
If $\sigma(t+1)=1$, let $H_{y_{t+1}}$  denote the  $(y_{t+1}, R_{N+1}, z^{y_{t+1},H_{y'_{t+1}}}, 2^{N+1}\cdot \gamma_0  ) $-good set induced by the good input package $\left( \mathcal{X}^{y_{t+1}}, \pinkwalk^{y_{t+1}, H_{y'_{t+1}} }  \right)$. We call $H_{y_{t+1}}$ the seed indexed by $y_{t+1}$ (where $y_{t+1}=10 \cdot R_{N+1}\cdot x_{t+1}$ and  $x_{t+1} \in G_{t+1}$ by \eqref{S_t_plus_1_def}).
\end{notation}
\noindent In words: $\sigma(t+1)$ is the indicator of the event that we can construct a new seed $H_{y_{t+1}}$ in the $R_{N+1}$-neighbourhood of $y_{t+1}$ which is also connected to
the earlier seed $H_{y'_{t+1}}$.

\begin{lemma}[The chance of finding something good in the next step is at least $1/2$] \label{lemma_sigma_onehalf}  Under Assumption \ref{assumption_good_seq_scales}, we have
\begin{align}
\mathbb{P} \left(\sigma(0)=1\right)  &\geq 1/2,\\
\label{sigma_condprob_geq_1dash2}
\mathbb{P} \left( \sigma(t+1) = 1 \, | \, \sigma(0), \sigma(1), \ldots, \sigma(t) \right) & \geq 1/2, \quad t \in \mathbb{N}_0,
\end{align}
where $\sigma(0)$ and $\sigma(t), t \in \mathbb{N}$ are introduced in Definitions \ref{notation_sigma_null} and \ref{notation_sigma_tplus1}, respectively.
\end{lemma}

\noindent We will prove Lemma \ref{lemma_sigma_onehalf} in Section \ref{subsub_sigma_onehalf}. We are ready to prove the main result of Section \ref{section_perco_worms_multiscale}.

\begin{proof}[Proof of Theorem \ref{thm_worm_perco}]
Let us fix the constants $R^*_0,\gamma_0, \Dlow, \Dup, \underline{\alpha}, \psi, s, \Lambda \in (0,+\infty)$ as in Assumption \ref{assumption_good_seq_scales}.
Given some
   $v \in (0,+\infty)$ and a probability measure $m$ on $\mathbb{N}$, we  assumed that there exists an
 $(R^*_0, \gamma_0, \Dlow, \Dup, \underline{\alpha}, \psi, s, \Lambda, v )$-good sequence $( R_n )_{n = 0}^{N+1}$ of scales for $m$. Let us define the indicators
 $\sigma(t), t \in \mathbb{N}_0$ as in Definitions \ref{notation_sigma_null} and \ref{notation_sigma_tplus1}. Given these indicators,
  let us define $S_t=(G_t,B_t)$ as in Section \ref{subsection_exploration}.

Let us now argue that $|G_\infty|=+\infty$ implies that $\mathcal{S}^v$ percolates. First note that
 it follows from our definitions that the seed sets of form $H_y, y \in  10 \cdot R_{N+1} \cdot G_\infty$ satisfy $H_y \subseteq \mathcal{S}^v$, moreover  every such seed set $H_y$ is connected
(cf.\ Definition \ref{def:gamma_good_input_package}).
Next we show by induction on $t$ that for each $t \in \mathbb{N}_0$
 the sets $H_y, y \in  10 \cdot R_{N+1} \cdot G_t$ are in the same connected component of $\mathcal{S}^v$. Indeed,
by \eqref{back_and_forth} and  Definitions \ref{nota:shooting_steps},  \ref{notation_sigma_null} and  \ref{notation_sigma_tplus1}, for each $t \geq 0$ satisfying $\sigma(t+1)=1$,
the old seed  $H_{y'_{t+1}}$  and the newly added seed $H_{y_{t+1}}$ are connected by the worm
$\widehat{\pinkwalk}^{y_{t+1},H_{y'_{t+1}}}$  (the trace of which is a subset of $\mathcal{S}^v$). Therefore, the sets $H_y, y \in  10 \cdot R_{N+1} \cdot G_\infty$ are in the same connected component of $\mathcal{S}^v$, thus $|G_\infty|=+\infty$ indeed implies that $\mathcal{S}^v$ percolates.

 By Proposition \ref{lemma:Grimmett_Marstrand} and Lemma \ref{lemma_sigma_onehalf} we obtain that $\mathbb{P}( | G_{\infty}| = \infty ) > 0$ holds, noting that
 for $d \geq 5$ we have $p_c^{\text{site}}(\Z^d) < 1 \slash 2$ (this is due to \cite{CampaninoRusso1985}, where the authors proved that $p_c^{\text{site}}( \Z^3 ) < 1 \slash 2$ and to the fact that $p_c^{\text{site}}(\Z^{d+1}) \leq p_c^{\text{site}} (\Z^d)$ for all $d \geq 1$).

From the above argument we obtain $ \mathbb{P}(\, \mathcal{S}^v \text{ percolates}\, )>0$, from which $ \mathbb{P}(\, \mathcal{S}^v \text{ percolates}\, )=1$ follows using ergodicity (cf.\ \eqref{eq_ergodic}).
\end{proof}

\subsubsection{Target shooting with a boomerang}\label{subsection_boomerang}

\begin{proof}[Proof of Lemma \ref{lemma:good_shooting}]
	Let us consider some $R \in \N$, $y \in \mathbb{Z}^d$ and $H \subseteq \ball(y, 13  R)$.
   The random variable $\widehat{\mathcal{X}}( B[y, H] )$ follows Poisson distribution with parameter $\mathbb{E} [ \, \widehat{\mathcal{X}}( B[y, H] ) \,  ]$ by  Definition \ref{def_full_worm_ppp}.  Next we bound this expectation from below:
   	\begin{align}
	\mathbb{E} [ \, \widehat{\mathcal{X}}( B[y, H] ) \,  ]
	& \stackrel{(\bullet)}{ = }
	v \cdot \sum_{\ell = (2 \beta + 1) R^2}^{\infty} m(\ell) \cdot \sum_{x \in \ball(y, 2R)} P_x \left( T_H \leq \beta R^2, \, T_{ \ball(y, R) }(X[T_H, \ell-1]) \leq \beta R^2 \right)
	\notag \\ & \stackrel{(*)}{\geq}
	v \cdot \sum_{\ell = (2 \beta + 1) R^2}^{\infty} m(\ell) \cdot \sum_{x \in \ball(y, 2 R)} P_x \left( T_H \leq \beta R^2 \right) \cdot \inf_{x' \in H} P_{x'} \left( T_{\ball(y,R)} \leq \beta R^2 \right)
	\notag \\ & \stackrel{(**)}{ \geq}
	v \cdot c \cdot \sum_{\ell = (2 \beta + 1) R^2}^{\infty} m(\ell) \cdot  R^d \cdot (\capacity(H) \cdot R^{2-d}) \cdot (R^{d - 2} \cdot R^{2 - d})
	\stackrel{(\diamond)}{\geq}
	2, \label{constant_appears_here}
	\end{align}
	where we now explain $(\bullet)$, $(*)$,  $(**)$ and $(\diamond)$.
The identity $(\bullet)$ follows from equation \eqref{eq:one_worm_ppp_expectation_formula} of Lemma \ref{prop:one_worm_ppp_formulas} with $f(w) = \ind[ w \in B[y,H] ]$ (see \eqref{back_and_forth} for the definition of $B[y,H]$). The inequality
 $(*)$ follows from the strong Markov property at $T_H$. The inequality $(**)$ follows if we apply \eqref{eq:lemma:two_birds_one_stone} twice (choosing $\beta = 2 \cdot 16^2$, so Lemma \ref{lemma:two_birds_one_stone} is applicable) and  the lower bound \eqref{eq:capacity_of_ball} on the capacity of $\ball(y,R)$.
 The inequality $(\diamond)$ follows from  \eqref{eq:good_shooting_capacity_lower_bound}  (where we defined $s$ to be equal to the constant $c$ that appears in \eqref{constant_appears_here}). Therefore, the parameter of the Poisson random variable $\widehat{\mathcal{X}}( B[y, H] )$ is bounded from below by $2$, thus we have $\mathbb{P} (\widehat{\mathcal{X}}( B[y,H] ) = 0) \leq \exp \{ -2 \} \leq 1 \slash 4$, completing the proof of  Lemma \ref{lemma:good_shooting}.
\end{proof}

\subsubsection{The chance of finding something good in the next step is at least $1/2$}\label{subsub_sigma_onehalf}

The goal of Section \ref{subsub_sigma_onehalf} is to prove Lemma \ref{lemma_sigma_onehalf}, but before we can do so, we need to state and prove two auxiliary results: Lemmas \ref{lemma_going_up_the_ladder}  and \ref{lemma_cond_law_given_explored_stuff}.

Recall the notion of the natural law $\mathcal{P}_{y,R,z,v}$ on input packages from Definition \ref{def:law_on_input_packages}.

\begin{lemma}[Good input packages on scales up to $N+1$]\label{lemma_going_up_the_ladder} Under Assumption \ref{assumption_good_seq_scales}, we have
\begin{equation}
	\label{R_n_good_3dash4}
	\min_{z \in  \ball(R_n)}
	\mathcal{P}_{o,R_n,z,v}\left(\, (\worms, \pinkwalk) \text{ is $(o,R_n,z, 2^n \cdot \gamma_0  )$-good} \, \right) \geq
	3 \slash 4
	\end{equation}
for all $0 \leq n \leq N+1$.
\end{lemma}
\begin{proof} We prove \eqref{R_n_good_3dash4} for  $0 \leq n \leq N+1$ by induction on $n$. The $n=0$ case follows from  our assumption $R_0  \geq R^*_0$ (cf.\ \eqref{condition_Rzero}) and Lemma \ref{lemma_initializing_gamma}. Let us now assume that our induction hypothesis \eqref{R_n_good_3dash4} holds for some $n \in \{0,\dots,N\}$. In order to conclude that
\eqref{R_n_good_3dash4} also holds for $n+1$, we want to apply Lemma \ref{lemma_doubling_gamma} with $\Rs=R_n$, $\Rb=R_{n+1}$ and $\gamma=2^n \cdot \gamma_0$.
By Assumption \ref{assumption_good_seq_scales}, $( R_n )_{n = 0}^{N+1}$ is an $(R^*_0, \gamma_0, \Dlow, \Dup, \underline{\alpha}, \psi, s, \Lambda, v )$-good sequence of scales for $m$, thus
\eqref{gamma_ineq_assumptions} holds by our assumption $2^n \cdot \gamma_0  \leq \psi \cdot R_n^{d-4}$ (cf.\ \eqref{condition_psi}) and
\eqref{def:alpha} holds by our assumption
		$v \cdot \mathbb{E} \left[\, \mathcal{L}^2 \, \ind \left[ \Dlow \cdot R_n^2 \leq \mathcal{L} \leq \Dup \cdot R_{n+1}^2 \right]\, \right]  \geq \underline{\alpha}$
(cf.\ \eqref{condition_R_n_R_n1}). Thus we can indeed use Lemma \ref{lemma_doubling_gamma} to proceed from $n$ to $n+1$.
The proof of Lemma \ref{lemma_going_up_the_ladder} is complete.
\end{proof}

In order to rigorously state Lemma \ref{lemma_cond_law_given_explored_stuff}  about the conditional law of the next input package given what we have already explored, we need the following definitions.

\begin{notation}[Explored sigma-field $\mathcal{F}_t$]\label{explored_sigma_algebra} For any $t \in \mathbb{N}_0$, let $\mathcal{F}_t$ denote the $\sigma$-field generated by
 \begin{equation}\widehat{\mathcal{X}}\mathds{1}[ \, w(0) \in \ball(y,2 R_{N+1})\, ], \qquad y \in 10\cdot R_{N+1} \cdot (G_t \cup B_t). \end{equation}
\end{notation}
\noindent In words, $\mathcal{F}_t$ contains all information about all of the worms that emanate from the boxes of scale $R_{N+1}$ explored by the time we finished step $t$ of the exploration. Note that the indicators $\sigma(0), \dots, \sigma(t)$, the good set $G_t$, the bad set $B_t$, the seeds $H_y$ for all $y \in 10\cdot R_{N+1} \cdot G_t  $ and the data  $e_{t+1}$, $x_{t+1}$, $x'_{t+1}$, $y_{t+1}$, $y'_{t+1}$ describing the location that we want to explore in step $t+1$ (cf.\ Definition \ref{notation_sigma_tplus1}) are all $\mathcal{F}_t$-measurable.

\begin{notation}[Sigma-fields $\mathcal{G}_*$ and $\mathcal{G}_t$ just before fattening] \label{explored_sigm_field_long} $ $

\begin{enumerate}[(i)]
\item Recalling Definition \ref{notation_sigma_null}, let $\mathcal{G}_*$ denote the sigma-field generated by the $\mathbb{N}_0$-valued random variable
 $\widehat{\mathcal{X}}( B[o, H_* ] )$, the auxiliary randomness that we use to pick $\widehat{\pinkwalk}^{o, H_*} $ if $\widehat{\mathcal{X}}( B[o, H_* ] )>0$ (cf.\ Definition \ref{nota:shooting_steps}) and
 $\widehat{\pinkwalk}^{o, H_*}[0,\widehat{T}^{o, H_*}]$.
\item Recalling Definition \ref{notation_sigma_tplus1}, for any $t \in \mathbb{N}_0$, let $\mathcal{G}_t$ denote the sigma-field generated by $\mathcal{F}_t$, the $\mathbb{N}_0$-valued random variable $ \widehat{\mathcal{X}}( B[y_{t+1},H_{y'_{t+1}}] )$, the auxiliary randomness that we use to pick
    $\widehat{\pinkwalk}^{y_{t+1}, H_{y'_{t+1}} }$ if $\widehat{\mathcal{X}}( B[y_{t+1},H_{y'_{t+1}}] )>0$ and
     $\widehat{\pinkwalk}^{y_{t+1}, H_{y'_{t+1}} }[0, \widehat{T}^{y_{t+1}, H_{y'_{t+1}} }]$.
\end{enumerate}
\end{notation}
\noindent Recalling \eqref{z_y_H},  note that $z^{o,H_*}$ is $\mathcal{G}_*$-measurable and $z^{y_{t+1}, H_{y'_{t+1}} }$ is $\mathcal{G}_t$-measurable.

\begin{lemma}[Conditional distribution of input packages]\label{lemma_cond_law_given_explored_stuff} Let $t \in \mathbb{N}_0$.

\begin{enumerate}[(i)]
\item \label{cond_distrr_i} Recalling \eqref{def_B_y_package_worms},  conditionally on $\mathcal{F}_t$, the point measure  $\widehat{\mathcal{X}}\mathds{1}[ B[y_{t+1}] ]$ of worms is a PPP
on $\widetilde{W}$ with intensity measure $v \cdot m( L(w) ) \cdot (2 d)^{ 1 - L(w) }\mathds{1}[ B[y_{t+1}] ]$.
\item \label{cond_distrr_ii} The conditional distribution of $\left( \mathcal{X}^{o}, \pinkwalk^{o, H_*}  \right)$ given $\mathcal{G}_*$ is  $\mathcal{P}_{o,R_{N+1},z^{o,H_*},v}$
(cf.\ Definition \ref{def:law_on_input_packages}).
	\item \label{cond_distrr_iii} The conditional distribution of $\left( \mathcal{X}^{y_{t+1}}, \pinkwalk^{y_{t+1}, H_{y'_{t+1}} }  \right)$ given $\mathcal{G}_t$ is
$\mathcal{P}_{y_{t+1},R_{N+1},z^{y_{t+1}, H_{y'_{t+1}} } ,v}$.
\end{enumerate}
\end{lemma}
\begin{proof} $ $

\begin{enumerate}[(i)]
  \item \label{condindep_i} The ball $\ball(y_{t+1}, 2 R_{N+1})$ is disjoint from  $ \cup_{x \in G_t \cup B_t} \ball( 10\cdot R_{N+1} \cdot x , 2 R_{N+1}) $, therefore
  the set $ B[y_{t+1}]$ of worms is disjoint from the set $ \cup_{x \in G_t \cup B_t}  \{ \, w(0) \in \ball( 10\cdot R_{N+1} \cdot x ,2 R_{N+1})\, \}$ of worms. Now the desired statement follows from the fact that the restrictions of a PPP on $\widetilde{W}$ to disjoint subsets of $\widetilde{W}$ give rise to independent Poisson point processes.
    \item The proof is  analogous to the proof of the next statement.
    \item The set of worms $\{  w(0) \in \ball(y_{t+1}, 2 R_{N+1}), \, L(w) \leq R_{N+1}^2 \}$ that we use in the definition of $ \mathcal{X}^{y_{t+1}}$
    (cf.\ \eqref{new_X_y}) is disjoint from the set $ \cup_{x \in G_t \cup B_t}  \{ \, w(0) \in \ball( 10\cdot R_{N+1} \cdot x ,2 R_{N+1})\, \}$ of worms by the above reasoning, but it is also
    disjoint from  $B[y_{t+1}]$, which only contains worms $w$ with length $L(w)$ greater than $R_{N+1}^2$ (cf.\  \eqref{def_B_y_package_worms}).   Thus, conditionally on $\mathcal{G}_t$, the point measure
    $ \mathcal{X}^{y_{t+1}}$  is a PPP on $\widetilde{W}$ with intensity measure $v \cdot m( L(w) ) \cdot (2 d)^{ 1 - L(w) }\ind[\, w(0) \in \ball(y_{t+1}, 2 R_{N+1}),\, L(w) \leq R_{N+1}^2\, ]$, as required by Definition \ref{def:law_on_input_packages}.

    By an application of the strong Markov property at time $\widehat{T}^{y_{t+1}, H_{y'_{t+1}} }$, we see that conditionally on  $\mathcal{G}_t$, the worm
     $\pinkwalk^{y_{t+1}, H_{y'_{t+1}} } $ defined by  \eqref{pinkwalk_for_exploration} is
    a simple random walk started from $z^{y_{t+1}, H_{y'_{t+1}} }$ that we run until it either performs $R^2_{N+1}$ steps or it exits the ball $\ball(y_{t+1}, 2 R_{N+1})$, as required by Definition \ref{def:law_on_input_packages}.
\end{enumerate}

\end{proof}

\begin{proof}[Proof of Lemma \ref{lemma_sigma_onehalf}]
We prove $\mathbb{P} \left(\sigma(0)=1\right)  \geq 1/2$, where $\sigma(0)$ is defined by \eqref{sigma_null_def}, by showing
\begin{align}
\label{initial_long}
\mathbb{P} \left( \widehat{\mathcal{X}}( B[o, H_* ] ) > 0 \right) & \geq 3/4, \\
\label{initial_fattening}
\mathbb{P} \left(  \left( \mathcal{X}^{o}, \pinkwalk^{o, H_*}  \right) \text{is } (o, R_{N+1}, z^{o,H_*}, 2^{N+1}\cdot \gamma_0  ) \text{-good}\, \big\vert \, \mathcal{G}_* \, \right) & \geq 3/4.
\end{align}
Indeed, as soon as we prove  \eqref{initial_long} and \eqref{initial_fattening}, the inequality $\mathbb{P} \left(\sigma(0)=1\right)  \geq (3/4)^2 \geq 1/2$ follows from \eqref{sigma_null_def} by the tower property and the fact that $ \widehat{\mathcal{X}}( B[o, H_* ] )$ is $ \mathcal{G}_*$-measurable.

We want to prove \eqref{initial_long} using Lemma \ref{lemma:good_shooting}, thus we only need to check that \eqref{eq:good_shooting_capacity_lower_bound} holds, i.e., that we have $\capacity(H_*) \cdot  s \cdot v \cdot R_{N+1}^2 \cdot \mathbb{P} \left( \mathcal{L} \geq \Lambda \cdot R_{N+1}^2 \right)   \geq 2$ (recalling that we set $\Lambda=2\beta+1$ in Assumption \ref{assumption_good_seq_scales}). Since \eqref{condition_shoot} holds by Assumption \ref{assumption_good_seq_scales}, we only need to check that the inequality
$\capacity(H_*)\geq 2^{N+1} \cdot \gamma_0 \cdot R^2_{N+1}$ holds. However, this inequality follows from the $n=N+1$ case of Lemma \ref{lemma_going_up_the_ladder}, as we now explain.

 It certainly follows from \eqref{R_n_good_3dash4} that a $(o, R_{N+1}, z, 2^{N+1} \cdot \gamma_0 )$-good set $H$ \emph{exists}. On the one hand, we have $H \subseteq \ball(3R_{N+1})=H_* $, on the other hand, we have $2^{N+1} \cdot \gamma_0 \cdot R^2_{N+1} \leq \capacity(H) $ (cf.\ Definition \ref{def:gamma_good_input_package}). Thus
$\capacity(H_*)\geq 2^{N+1} \cdot \gamma_0 \cdot R^2_{N+1}$ follows using the monotonicity \eqref{cap_mon} of capacity. The proof of \eqref{initial_long} is complete.

The inequality  \eqref{initial_fattening} follows from Lemma \ref{lemma_cond_law_given_explored_stuff}\eqref{cond_distrr_ii} and the $n=N+1$ case of Lemma \ref{lemma_going_up_the_ladder}.

\medskip

Given some $t \in \mathbb{N}_0$, we also want to prove  $\mathbb{P} \left( \sigma(t+1) = 1 \, | \, \sigma(0),  \ldots, \sigma(t) \right)  \geq 1/2$ (i.e., \eqref{sigma_condprob_geq_1dash2}).

First recall that if $\partial^{\text{ext}} G_t \subseteq B_t$ then we defined $\sigma(t+1)=1$ in Definition \ref{notation_sigma_tplus1}, thus in this case
$\mathbb{P} \left( \sigma(t+1) = 1 \, | \, \sigma(0),  \ldots, \sigma(t) \right)=1$ follows from the fact that $G_t$ and $B_t$ are determined by $ \sigma(0),  \ldots, \sigma(t)$.

In the $\partial^{\text{ext}} G_t \not\subseteq B_t$ case the indicator $\sigma(t+1)$ is defined by \eqref{sigma_tplus1_def},  hence  \eqref{sigma_condprob_geq_1dash2} will follow as soon as we prove the inequalities
\begin{align}
\label{tplus1_shoot_3dash4}
\mathbb{P} \left(  \,   \widehat{\mathcal{X}}\left( B\left[y_{t+1},H_{y'_{t+1}}\right] \right) > 0  \, \big\vert \, \mathcal{F}_t  \right)\geq 3/4, \\
\label{tplus1_fatten_3dash4}
\mathbb{P}\left( \, \left( \mathcal{X}^{y_{t+1}}, \pinkwalk^{y_{t+1}, H_{y'_{t+1}} }  \right) \text{is } (y_{t+1}, R_{N+1}, z^{y_{t+1},H_{y'_{t+1}}}, 2^{N+1}\cdot \gamma_0  ) \text{-good}\, \big\vert \, \mathcal{G}_t \right) \geq 3/4.
\end{align}
Indeed, if we prove \eqref{tplus1_shoot_3dash4} and \eqref{tplus1_fatten_3dash4}, the inequality $\mathbb{P} \left( \sigma(t+1) = 1 \, | \, \sigma(0),  \ldots, \sigma(t) \right)  \geq (3/4)^2 \geq 1/2$ follows from  \eqref{sigma_tplus1_def} by a repeated application of the tower property together with the facts that $\sigma(0),\dots,\sigma(t)$ are $\mathcal{F}_t$-measurable, $\mathcal{F}_t \subseteq \mathcal{G}_t$, and
$ \widehat{\mathcal{X}}( B[y_{t+1},H_{y'_{t+1}}] )$ is $\mathcal{G}_t$-measurable.

We want to prove \eqref{tplus1_shoot_3dash4} using  Lemma \ref{lemma:good_shooting}. First note that it follows from $y'_{t+1} \in 10\cdot R_{N+1} \cdot G_t $  that $H_{y'_{t+1}}$ is an $\mathcal{F}_t$-measurable $(y'_{t+1}, R_{N+1}, z, 2^{N+1} \cdot \gamma_0 )$-good set for some $z \in \ball(y'_{t+1},R_{N+1}) $,
thus $\capacity(H_{y'_{t+1}})\geq 2^{N+1} \cdot \gamma_0 \cdot R^2_{N+1}$ holds by Definition \ref{def:gamma_good_input_package}. If we put this bound together with the assumption that $( R_n )_{n = 0}^{N+1}$ is an $(R^*_0, \gamma_0, \Dlow, \Dup, \underline{\alpha}, \psi, s, \Lambda, v )$-good sequence  of scales for $m$ (cf.\ Assumption \ref{assumption_good_seq_scales}), we obtain that condition \eqref{eq:good_shooting_capacity_lower_bound} of Lemma \ref{lemma:good_shooting} holds:
\[
\capacity(H_{y'_{t+1}}) \cdot  s \cdot v \cdot R_{N+1}^2 \cdot \mathbb{P} \left( \mathcal{L} \geq \Lambda \cdot R_{N+1}^2 \right) \geq
 2^{N+1} \cdot \gamma_0 \cdot  s \cdot v \cdot R_{N+1}^4 \cdot \sum_{\ell=\Lambda \cdot R_{N+1}^2}^{\infty} m(\ell)    \stackrel{ \eqref{condition_shoot} }{\geq} 2.
\]
Putting this together with Lemma \ref{lemma_cond_law_given_explored_stuff}\eqref{cond_distrr_i}, we see that  Lemma \ref{lemma:good_shooting} indeed gives \eqref{tplus1_shoot_3dash4}.

The inequality  \eqref{tplus1_fatten_3dash4} follows from Lemma \ref{lemma_cond_law_given_explored_stuff}\eqref{cond_distrr_iii} and the $n=N+1$ case of Lemma \ref{lemma_going_up_the_ladder} combined with the translation invariance property \eqref{shift_invar_package}.

\end{proof}

\section{Doubling the $\gamma$ of good input packages by fattening}\label{subsection_doubling_gamma}

The goal of Section \ref{subsection_doubling_gamma} is to prove Lemma \ref{lemma_doubling_gamma}: given an input package $(\worms, \pinkwalk)$ we want to define a set $H$ and we want to show that it is good with high enough probability.

In Section \ref{subsection_consr_H} we introduce some notation that will allow us to define $H$.  We want to bound $\capacity(H)$ from below using Dirichlet energy, therefore we will also define a measure $\hitonly{\mu}_{\lozenge}$ supported on $H$. However, the measure $\hitonly{\mu}_{\lozenge}$ is not a probability measure, so we will state two lemmas:
 Lemma \ref{lemma:small_energy} states that the Dirichlet energy $\energy(\hitonly{\mu}_{\lozenge})$ is not too big, while Lemma \ref{lemma:big_total_measure} states that the total mass of $\hitonly{\mu}_{\lozenge}$ is not too small.
  We will fix the parameters  $\psi, \Dup,  \Dlow , \underline{\alpha} $ (cf.\ \eqref{parameter_assumptions}) and deduce the proof of Lemma \ref{lemma_doubling_gamma} from  Lemmas \ref{lemma:small_energy} and  \ref{lemma:big_total_measure} at the end of Section \ref{subsection_consr_H}. We will prove Lemmas \ref{lemma:small_energy} and  \ref{lemma:big_total_measure} in Sections \ref{subsection_small_energy} and \ref{subsection_big_local_time}, respectively.

\subsection{Construction of the good set $H$}\label{subsection_consr_H}

Let us  assume that we are given some parameters $\psi, \Dup,  \Dlow , \underline{\alpha} $ that satisfy the inequalities \eqref{parameter_assumptions}, moreover we are given $\Rs \leq \Rb$ and  $\gamma$ satisfying the inequality \eqref{gamma_ineq_assumptions}. We further assume that we are given $v \in \mathbb{R}_+$
and a probability distribution $m$ on $\mathbb{N}$ satisfying the inequality \eqref{def:alpha}. We assume  that the hypothesis \eqref{induction_hypotheses} holds.

The goal of  Section \ref{subsection_doubling_gamma} is to conclude that \eqref{2gamma_want} holds, so let us assume that
 we are given an input package $(\worms, \pinkwalk) \sim \mathcal{P}_{o,\Rb,z,v}$ for some $z \in  \ball(\Rb)$. We want to show that
$(\worms, \pinkwalk)$  is $(o,\Rb,z,2\gamma)$-good with probability at least $3/4$, i.e., we want to construct a set $H$ which is $(o,\Rb,z,2\gamma)$-good for $(\worms, \pinkwalk)$ (cf.\ Definition \ref{def:gamma_good_input_package}) with probability at least $3/4$. Before we construct this  set $H$ in \eqref{def:fattened_set} below, let us introduce some notation.

 Recall that by our assumptions \eqref{parameter_assumptions} and \eqref{def:alpha}, the ratio of  $\Rb$ and $\Rs$ can be bounded from below as
\begin{equation}\label{big_Dlow_big_ratio}
\Rb \slash \Rs = \sqrt{ \Rb^2 \slash \Rs^2 } \geq \sqrt{ \Dlow \slash \Dup } \geq \sqrt{ \Dlow}.
\end{equation}

In order to make good use of the assumption \eqref{induction_hypotheses} of Lemma \ref{lemma_doubling_gamma}, we want to fit many boxes of scale $\Rs$ in a box of scale $\Rb$:

\begin{notation}[Indicator $\chi_y$ of a proper visit to $\ball(y, \Rs )$ by $\pinkwalk$] Let us  define the set $\mathcal{D} \subset \subset \mathbb{Z}^d$ and the indicator variables $\chi_y,  y \in \mathcal{D} $ by
	\begin{align}
		\label{def:D_index_set}
		\mathcal{D}  & := \big\{ \, y \in 10  \Rs  \Z^d \; : \; \ball(y, \Rs) \cap \ball(2\Rb) \neq \emptyset \, \big\}, \\
		\label{nota:chi_y}
		\chi_y  & := \mathds{1}\left[\, T_{\ball(y, \Rs )}(\pinkwalk) <  T_{\ball(2\Rb)^c}(\pinkwalk) \wedge (\Rb^2 - \Rs^2)\,  \right], \qquad y \in \mathcal{D}.
	\end{align}
\end{notation}
\noindent In words, $\chi_y$ denotes the indicator of the event that $\pinkwalk$ hits the box $\ball(y, \Rs )$ before it either exits the box $\ball(2\Rb)$ or it performs $\Rb^2 - \Rs^2$ steps. The reason why we have $\Rb^2 - \Rs^2$ (rather than $\Rb^2$) in \eqref{nota:chi_y} is that this  guarantees that $\pinkwalk$ (which performs $\Rb^2$ steps unless it exits $\ball(2\Rb)$) still has $\Rs^2$ steps to spare after hitting $\ball(y, \Rs )$: this segment of $\pinkwalk$ will be used to produce an input package of scale $\Rs$
in Definition \ref{notation:pink_packages}.

Recall from Definitions \ref{def_full_worm_ppp} and \ref{def:law_on_input_packages} that $\widehat{\mathcal{X}}$ denotes a PPP of worms with law $\mathcal{P}_v$, while $\mathcal{X}$ is a PPP of worms with the same law as $\widehat{\mathcal{X}} \ind[\, w(0) \in \ball(y, 2 R),\, L(w) \leq R^2\, ]$.

\begin{notation}[Input package of scale $\Rs$ centered at  $y$]\label{notation:pink_packages}
	If $\chi_y=1$ for some $y \in \mathcal{D}$, let
	\begin{align}\label{def_eq_z_y}
		z^y & := Z( T_{\ball(y, \Rs )}(\pinkwalk) ), \\
		T^y & := \Rs^2 \wedge \min\{ \, t \, : \, Z(T_{\ball(y, \Rs )}(\pinkwalk) + t) \notin \ball(y, 2\Rs) \,  \}, \\
 		\pinkwalk^y & :=  Z[ \, T_{\ball(y, \Rs )}(\pinkwalk) , \, T_{\ball(y, \Rs )}(\pinkwalk) +T^y \, ], \\
 	\label{def_eq_worms_y}	\worms^y & := \widehat{\worms} \ind \big[ \, w(0) \in \ball(y, 2\Rs), \, L(w) \leq \Rs^2 \, \big].
	\end{align}
\end{notation}
\noindent In words: $z^y$ is the location where $\pinkwalk$ hits $\ball(y, \Rs )$, $T^y$ is the minimum of $ \Rs^2$ and  the length of the time interval that $Z$ spends
in $\ball(y, 2\Rs)$ after hitting $\ball(y, \Rs )$, $\pinkwalk^y$ is the segment of length $T^y$ of $\pinkwalk$ that starts when $\pinkwalk$ hits $\ball(y, \Rs )$, and
$\worms^y $ is the point process of worms that consists of those worms of $ \widehat{\worms}$ that start from $\ball(y, 2\Rs)$ and have length at most $\Rs^2$. Let us stress that  $\pinkwalk^y$ is indeed a sub-worm of $\pinkwalk$ since
$ T_{\ball(y, \Rs )}(\pinkwalk) +T^y$ is less than or equal to the length of $\pinkwalk$  by \eqref{nota:chi_y} and the definition of $(\worms, \pinkwalk) \sim \mathcal{P}_{o,\Rb,z,v}$ (cf.\ Definition \ref{def:law_on_input_packages}).

\begin{notation}[Indicator $\xi_y$ of a good set $H_y$ inside $\ball(y, 3\Rs )$]\label{notation_H_y}
	Let us define the indicators
	\begin{equation}
		\label{nota:xi_y}
		\xi_y  := \chi_y \cdot  \ind \big[\,  (\worms^y, \pinkwalk^y) \text{ is } (y,\Rs,z^y,\gamma) \text{-good} \, \big], \quad  y \in \mathcal{D}
	\end{equation}
	and for each $y \in \mathcal{D}$ with $\xi_y = 1$, let $H^y$ denote the $(y,\Rs,z^y,\gamma)$-good set for $(\worms^y, \pinkwalk^y)$ (cf.\ Definition \ref{def:gamma_good_input_package}). If there are more than one candidates for such a set $H^y$,
 let us pick one of them according to an arbitrary (but deterministic) rule (e.g., we can fix a well-ordering  of  the set of finite subsets of $\mathbb{Z}^d$ and
 pick the least one among the candidates).
   If $y \in \mathcal{D}$ with $\xi_y = 0$,  let $H^y = \emptyset$.
\end{notation}

\begin{notation}
	\label{nota:Falg}
	Let $\Falg$ denote the $\sigma$-field generated by  $\worms \ind [L(w) \leq \Rs^2 ]$ and   $\pinkwalk$.
\end{notation}

\noindent Note that the indicators $\chi_y$, $\xi_y$, $y \in \mathcal{D}$ and the sets $H^y, y \in \mathcal{D}$ are all $\Falg$-measurable.

In order to achieve the goal of Lemma \ref{lemma_doubling_gamma} (i.e., doubling the capacity-to-length ratio $\gamma$ by going from a lower scale to a higher scale)
we will  ``fatten'' the set $\cup_{y \in \mathcal{D}}H^y$.
Loosely speaking, we will consider the union of the traces of worms of length between  $\Dlow \cdot \Rs^2$ and $\Dup \cdot \Rb^2$ that hit
$\cup_{y \in \mathcal{D}}H^y$ and show that the capacity-to-length ratio of this set is at least $2\gamma$ with high enough probability.
The precise details of this loose plan are given in the next few definitions.

Let us define the set $A_{\bullet} \subseteq \widetilde{W}$ and the point process $\mathcal{X}_{\bullet} \in \widetilde{\mathbf{W}}$ by
\begin{equation}
	\label{worms:base}
	\mathcal{X}_{\bullet}:= \mathcal{X} \ind [A_{\bullet}], \quad
	A_{\bullet} := \left\{ \,
	\Dlow \cdot \Rs^2 \leq L(w) \leq \Dup \cdot \Rb^2, \;
	w(0) \in \mathcal{G}, \;
 \mathrm{Tr}(w) \subseteq \ball(3\Rb) \,
	\right\},
\end{equation}
where the set $\mathcal{G} \subseteq \mathbb{Z}^d$ is defined by
\begin{equation}
	\label{nota:ground_for_worms}
	\mathcal{G} := \ball(2\Rb) \setminus \bigcup_{y \in \mathcal{D}} \ball(y, 4\Rs)
 \stackrel{ \eqref{def:D_index_set} }{\supseteq} \ball(2\Rb) \setminus \bigcup_{y \in  10  \Rs  \Z^d} \ball(y, 4\Rs).
\end{equation}

\noindent
Recall that in \eqref{parameter_assumptions} we assumed  $\Dlow>1$, hence
\begin{equation}\label{X_bull_indep_of_Falg}
 \text{ $\mathcal{X}_{\bullet}$ is independent of $\Falg$.}
 \end{equation}
  We will  use the worms from $\mathcal{X}_{\bullet}$ for the purpose of fattening $\cup_{y \in \mathcal{D}} H^y $.

Given some $y \in \mathcal{D}$, let us define the set $\hitonly{A}[y] \subseteq \widetilde{W}$, and  the point process $\hitonly{\mathcal{X}}[y]\in \widetilde{\mathbf{W}}$ by
\begin{equation}
	\label{worms:hit_Hy_only}
	\hitonly{\mathcal{X}}[y]  := \mathcal{X}_{\bullet} \ind [\hitonly{A}[y] ], \;
	\hitonly{A}[y] := \left\{\, T_{H^y}(w) < \left \lfloor \tfrac{L(w)}{3} \right \rfloor, \; \forall\; \widetilde{y} \in \mathcal{D}\setminus \{y\} \, : \,
   T_{H^{ \widetilde{y} }}(w)\geq L(w)\, \right\}.
\end{equation}
In words: $\hitonly{\mathcal{X}}[y]$ consists of those worms from $ \mathcal{X}_{\bullet}$ that only hit $H^y$ (i.e., they do not hit $H^{ \widetilde{y} }$ if
$\widetilde{y} \in \mathcal{D}\setminus \{y\}$), moreover they hit $H^y$ with their first third segment.

Note that $\hitonly{A}[y] \cap \hitonly{A}[y'] = \emptyset$ if $y\neq y' \in \mathcal{D}$, therefore
\begin{equation}\label{hitonly_X_y_indep}
  \text{the Poisson point processes $\hitonly{\mathcal{X}}[y], \, y \in \mathcal{D}$ are conditionally independent given $\Falg$. }
\end{equation}

Given some $y \in \mathcal{D}$, we define  $\varphi^y  :  \hitonly{A}[y] \rightarrow \widetilde{W}$ and  the point process $\hitonly{\mathcal{X}}_{\lozenge}[y]  \in \widetilde{\mathbf{W}}$ by
\begin{equation}
	\label{last_one_third_pp}
	\hitonly{\mathcal{X}}_{\lozenge}[y]  \stackrel{ \eqref{nota:map_of_point_process} }{:=} \varphi^{y}( \hitonly{\mathcal{X}}[y] ), \quad
	\varphi^{y}(w):= w\left[\, T_{H^y}(w)+\lfloor L(w)/3\rfloor, T_{H^y}(w)+2\lfloor L(w)/3\rfloor\, \right].
\end{equation}
In words: $\hitonly{\mathcal{X}}_{\lozenge}[y]$ consists of the segments of the worms $w$ of $\hitonly{\mathcal{X}}[y]$  that start $\lfloor L(w)/3\rfloor$ steps
after hitting $H^y$ and have length $\lfloor L(w)/3\rfloor$. Note that $\varphi^{y}(w)$ is indeed a sub-worm of $w$, since $ T_{H^y}(w) <  \lfloor L(w)/3  \rfloor$ if $w$ is a worm from $\hitonly{\mathcal{X}}[y]$.

 We are ready to define our candidate $H$ for the role of the $(0,\Rb,z,2\gamma)$-good set for $(\worms, \pinkwalk)$ (cf.\ Definition \ref{def:gamma_good_input_package})
 with the aim of proving  \eqref{2gamma_want}. Let
  \begin{equation}
	\label{def:fattened_set}
	H  := \mathrm{Tr}(\pinkwalk) \cup \bigcup_{y \in \mathcal{D}} H^y  \cup \bigcup_{y \in \mathcal{D}} \mathrm{Tr}(\hitonly{\mathcal{X}}[y]).
\end{equation}
\noindent
 We want to lower bound the capacity of $H$ using Dirichlet energy (cf.\ \eqref{capacity_with_energy}), thus we will define a measure $\hitonly{\mu}_{\lozenge}$  supported on $H$.

Recalling the notion of the local time of point process of worms from \eqref{def:local_time},
 we define
\begin{equation}
	\label{def:mu_of_the_fattened_set}
	\hitonly{\mu}^y_{\lozenge}  := \mu^{\hitonly{\mathcal{X}}_{\lozenge}[y]}, \qquad
	\hitonly{\mu}_{\lozenge} := \sum_{y \in \mathcal{D}} \hitonly{\mu}^y_{\lozenge}, \qquad
	\hitonly{\Sigma}_{\lozenge} := \hitonly{\mu}_{\lozenge}(\mathbb{Z}^d).
\end{equation}
In words: $\hitonly{\mu}^y_{\lozenge}$ is the collective local time measure of the worms from $\hitonly{\mathcal{X}}_\lozenge[y]$, $\hitonly{\mu}_{\lozenge}$ is the
sum of the measures $\hitonly{\mu}^y_{\lozenge}, y \in \mathcal{D}$ and $\hitonly{\Sigma}_{\lozenge} $ is the total mass of  $\hitonly{\mu}_{\lozenge}$.

Let us define
\begin{equation}
	\label{assumption:alpha_EQ}
\alpha:=  v \cdot \sum_{\ell = \Dlow \cdot \Rs^2}^{ \Dup \cdot \Rb^2 } \ell^2 \cdot m(\ell).
\end{equation}
Note that we have $\alpha \geq \underline{\alpha}$ by \eqref{def:alpha}.

Heuristically, our goal will be  to show that $\capacity(H)$ is comparable to  $|H|$. More precisely, we want to show
that the Dirichlet energy $\energy(\hitonly{\mu}_{\lozenge})$ of  $\hitonly{\mu}_{\lozenge}$ is comparable to its total mass $\hitonly{\Sigma}_{\lozenge} $
(and then the lower bound on $\capacity(H)$ that we obtain using \eqref{dir_bil} and \eqref{capacity_with_energy} will be comparable to $\hitonly{\Sigma}_{\lozenge}$). The next two lemmas
give an upper bound on the Dirichlet energy and lower bound on the total mass of $\hitonly{\mu}_{\lozenge}$, respectively. The reader might find the following formulas easier to parse if we already mention that we will find that $\hitonly{\Sigma}_{\lozenge}$ is comparable to $\Rb^2 \cdot \gamma \cdot \alpha$.

Recall the notion of the parameter $\psi$ introduced in the statement of Lemma \ref{lemma_doubling_gamma}.

\begin{lemma}[Small energy]
	\label{lemma:small_energy}
	There exists $\CSmallEnergyUPPER < \infty$ such that if $\Dlow \geq 16$, then for any $K > 0$
	\begin{equation}
		\label{eq:lemma:small_energy}
		\mathbb{P} \big( \energy(\hitonly{\mu}_{\lozenge}) \geq K \big) \leq \frac{\CSmallEnergyUPPER \cdot \big( 1 +  \psi \cdot \alpha \big) \cdot (\Rb^2 \cdot \gamma \cdot \alpha) }{K}.
	\end{equation}
\end{lemma}
We will prove Lemma \ref{lemma:small_energy} in Section \ref{subsection_small_energy}.

\begin{lemma}[Big total local time]
	\label{lemma:big_total_measure}
	There exist $ \CBigTotalEnergyDENOMINATOR, \CBigTotalEnergyPARAMETERUPPER > 0$ and $\CBigTotalEnergyUPPER, \CBigTotalEnergyPARAMETERLOWER < \infty$ such that if $\Dup \leq \CBigTotalEnergyPARAMETERUPPER$ and $\Dlow \geq \CBigTotalEnergyPARAMETERLOWER$, then for any $K$ satisfying $ \CBigTotalEnergyDENOMINATOR \cdot (\Rb^2 \cdot \gamma \cdot \alpha) - 3K \slash 2 >0$ we have
	\begin{equation}
		\label{eq:lemma:big_total_measure}
		\mathbb{P} \big( \hitonly{\Sigma}_{\lozenge} \leq K \big) \leq 0.02 + \frac{ \CBigTotalEnergyUPPER \cdot  ( \Rb^2 \cdot \gamma \cdot \alpha ) \cdot \psi}{K} +
		\frac{ \CBigTotalEnergyUPPER \cdot \Rb^2 \cdot (\Rb^2  \cdot \gamma \cdot \alpha) }{ \big( \CBigTotalEnergyDENOMINATOR \cdot (\Rb^2 \cdot \gamma \cdot \alpha)  - 3 K \slash 2 \big)^2}
	\end{equation}
\end{lemma}
We will prove Lemma \ref{lemma:big_total_measure} in Section \ref{subsection_big_local_time}.

Now we are ready to prove Lemma \ref{lemma_doubling_gamma} using the results of Lemmas \ref{lemma:small_energy} and \ref{lemma:big_total_measure}.

\begin{proof}[Proof of Lemma \ref{lemma_doubling_gamma}] Recall our setup from the beginning of Section \ref{subsection_consr_H}, including that $\Rb \geq \sqrt{\Dlow} \Rs$ holds
by \eqref{big_Dlow_big_ratio}.
Our goal is to prove that \eqref{2gamma_want} holds.

 Note that the measure $\hitonly{\mu}_{\lozenge}$ (cf.\ \eqref{def:mu_of_the_fattened_set}) is supported on the set $H$ (cf.\ \eqref{def:fattened_set}) and $H$ satisfies criteria (i), (ii) and (iii) of a $(0,\Rb,z,2\gamma)$-good set for $(\worms, \pinkwalk)$ (cf.\ Definition \ref{def:gamma_good_input_package}). Indeed:

 (i) $H \subseteq \ball(y,3\Rb) \cap (\mathrm{Tr}(\worms)\cup \mathrm{Tr}(\pinkwalk))$  holds if $\Dlow \geq 9$, since $H$ is defined by \eqref{def:fattened_set} and
 $H^y \subseteq \ball(y,3\Rb) \cap \mathrm{Tr}(\worms)$ for all $y  \in \mathcal{D}$ by \eqref{def:D_index_set} and Definition \ref{notation_H_y},
 $\mathrm{Tr}(\hitonly{\mathcal{X}}[y]) \subseteq \ball(y,3\Rb) \cap \mathrm{Tr}(\worms)$ by \eqref{worms:base}, \eqref{worms:hit_Hy_only} and \eqref{last_one_third_pp}, moreover
 $\mathrm{Tr}(\pinkwalk) \subseteq \ball(y,2\Rb) $ by Definition \ref{def:law_on_input_packages} and  $(\worms, \pinkwalk) \sim \mathcal{P}_{o,\Rb,z,v}$.

  (ii) $H$ is connected since $\mathrm{Tr}(\pinkwalk)$ is connected and  $\pinkwalk$ hits all of the sets $H^y$ for which $\xi_y=1$, moreover $H^y$ is connected if $\xi_y=1$ and  $\mathrm{Tr}(\hitonly{\mathcal{X}}[y])$ consists of worms that hit $H^y$.
  	
(iii)  $z \in H$  since $z \in \mathrm{Tr}(\pinkwalk)$ and $\mathrm{Tr}(\pinkwalk) \subseteq H$.

In order to show that $H$ also satisfies criterion (iv)  of a $(0,\Rb,z,2\gamma)$-good set for $(\worms, \pinkwalk)$, we only need to show that
$2 \cdot \gamma \cdot \Rb^2 \leq \capacity(H)$, because $  \capacity(H) \leq  2 \cdot \gamma \cdot \Rb^2 +1 $ can be achieved by throwing away some
of the points of $H$ using Lemma \ref{lemma_trimming_the_fat} (and noting that the resulting thinned version of $H$ still satisfies (i), (ii) and (iii)).

Therefore the statement \eqref{2gamma_want} of Lemma \ref{lemma_doubling_gamma} will follow if we show that
	\begin{equation}
		\label{eq:small_capacity_GOAL}
		\mathbb{P} \left( \capacity(H)  < 2 \cdot \gamma \cdot \Rb^2 \right) < 1 \slash 4
	\end{equation}
	holds, where $H$ is the set introduced in \eqref{def:fattened_set}.
Note that  $\hitonly{\mu}_{\lozenge}/ \hitonly{\Sigma}_{\lozenge}$ (cf.\ \eqref{def:mu_of_the_fattened_set}) is a probability measure supported on $H$, thus
	\begin{equation*}
		\mathbb{P} \left( \capacity(H) < 2 \cdot \gamma \cdot \Rb^2 \right)
		\stackrel{\eqref{capacity_with_energy}}{\leq} \mathbb{P} \left( \energy( \hitonly{\mu}_{\lozenge} / \hitonly{\Sigma}_{\lozenge} )^{-1} < 2 \cdot \gamma \cdot \Rb^2 \right)
\stackrel{\eqref{dir_bil}}{=}		\mathbb{P} \left( \tfrac{ \hitonly{\Sigma}_{\lozenge}^2 }{ \energy(\hitonly{\mu}_{\lozenge}) } < 2\cdot \gamma \cdot \Rb^2 \right).
	\end{equation*}
	For any $K' > 0$ the r.h.s.\ of the previous line can be bounded from above as
	\begin{equation}
		\label{eq:small_capacity_two_terms}
		\mathbb{P} \left( \tfrac{\hitonly{\Sigma}_{\lozenge}^2}{\energy(\hitonly{\mu}_{\lozenge})} < 2\cdot \gamma \cdot \Rb^2 \right)
		\leq
		\mathbb{P} \left( \energy(\hitonly{\mu}_{\lozenge}) \geq \tfrac{K'}{2 \cdot \gamma \cdot \Rb^2} \right) + \mathbb{P} \left( \hitonly{\Sigma}_{\lozenge} \leq \sqrt{K'} \right),
	\end{equation}
	hence to obtain \eqref{eq:small_capacity_GOAL}, it is enough to consider of the two terms separately.
	
	First of all, let us assume that with the constants introduced in Lemmas \ref{lemma:small_energy} and \ref{lemma:big_total_measure}, $\Dlow \geq \max \{ 16, \CBigTotalEnergyPARAMETERLOWER \}$ and $\Dup \leq \CBigTotalEnergyPARAMETERUPPER$ hold. Let us choose
	\begin{equation}\label{K_prime_choice}
		K' := \CPrfDoublingGammaAUX \cdot (\Rb^2 \cdot \gamma \cdot \alpha)^2,
	\end{equation}
where $\CPrfDoublingGammaAUX = ( \CBigTotalEnergyDENOMINATOR /3)^2$ and $ \CBigTotalEnergyDENOMINATOR $ is the constant that appears in Lemma \ref{lemma:big_total_measure}.
Next we bound the second term  on the r.h.s.\ of \eqref{eq:small_capacity_two_terms} using Lemma \ref{lemma:big_total_measure} with $K=\sqrt{K'}$, noting that with our choice \eqref{K_prime_choice} the inequality $ \CBigTotalEnergyDENOMINATOR  \cdot \Rb^2 \cdot \gamma \cdot \alpha  - 3 K \slash 2= 3 \sqrt{K'} \slash 2 >0$ holds:
	\begin{equation}
		\label{eq:small_capacity_SECONDTERM}
		\mathbb{P} \left( \hitonly{\Sigma}_{\lozenge} \leq \sqrt{K'} \right)=	
		\mathbb{P} \left( \hitonly{\Sigma}_{\lozenge} \leq \sqrt{ \CPrfDoublingGammaAUX } \cdot (\Rb^2 \cdot \gamma \cdot \alpha) \right)
		\stackrel{\eqref{eq:lemma:big_total_measure}}{\leq}
		0.02 + \frac{\CBigTotalEnergyUPPER \cdot \psi}{\sqrt{ \CPrfDoublingGammaAUX }} +
\frac{4 \CBigTotalEnergyUPPER}{ 9  \CPrfDoublingGammaAUX  \cdot \gamma \cdot \alpha},
	\end{equation}
where the constant $\CBigTotalEnergyUPPER$ appears in  Lemma \ref{lemma:big_total_measure}.

The first term on the r.h.s.\ of \eqref{eq:small_capacity_two_terms} can be bounded as follows:
	\begin{equation}
		\label{eq:small_capacity_FIRSTTERM}
		\mathbb{P} \left( \energy(\hitonly{\mu}_{\lozenge}) \geq \tfrac{K'}{2 \cdot \gamma \cdot \Rb^2} \right)
		\stackrel{\eqref{eq:lemma:small_energy}, \eqref{K_prime_choice}  }{\leq} \frac{2 \CSmallEnergyUPPER }{ \CPrfDoublingGammaAUX } \left(
		\alpha^{-1} +  \psi \right).
	\end{equation}

	Putting  \eqref{eq:small_capacity_SECONDTERM} and \eqref{eq:small_capacity_FIRSTTERM} back into \eqref{eq:small_capacity_two_terms} we obtain
	\begin{equation}\label{we_need_this_to_be_less_than_quarter}
		\mathbb{P} \left( \capacity(H) < 2 \cdot \gamma \cdot \Rb^2 \right)
		\leq
		0.02 + \frac{\CBigTotalEnergyUPPER  \cdot \psi}{\sqrt{ \CPrfDoublingGammaAUX }} +
 \frac{4 \CBigTotalEnergyUPPER}{ 9  \CPrfDoublingGammaAUX  \cdot \gamma \cdot \alpha}
 +
\frac{2 \CSmallEnergyUPPER }{ \CPrfDoublingGammaAUX } \left( \alpha^{-1} +  \psi  \right).
	\end{equation}
	
	Finally, in order to obtain \eqref{eq:small_capacity_GOAL}, we fix $\Dup$ and $\Dlow$ as above,
   fix  $\psi$ so that $\frac{\CBigTotalEnergyUPPER  \cdot \psi}{\sqrt{ \CPrfDoublingGammaAUX }}  +  \frac{2 \CSmallEnergyUPPER }{ \CPrfDoublingGammaAUX } \cdot \psi \leq 0.01$ and recall that $\gamma \geq \gamma_0$ by the assumption \eqref{gamma_ineq_assumptions}, thus
 we can choose $\underline{\alpha}$ so big that $\frac{4 \CBigTotalEnergyUPPER}{ 9  \CPrfDoublingGammaAUX  \cdot \gamma_0 \cdot \underline{\alpha}} \leq 0.01$ and $ \frac{2 \CSmallEnergyUPPER }{ \CPrfDoublingGammaAUX } \cdot \underline{\alpha}^{-1} \leq 0.01 $ both hold. By \eqref{def:alpha} and \eqref{assumption:alpha_EQ} we have $\alpha \geq \underline{\alpha}$, therefore the r.h.s.\ of \eqref{we_need_this_to_be_less_than_quarter} is less than $1/4$.
The proof of Lemma \ref{lemma_doubling_gamma} is complete.
\end{proof}

\subsection{Small energy}\label{subsection_small_energy}

The goal of Section \ref{subsection_small_energy} is to prove Lemma \ref{lemma:small_energy}.

Since $\mathcal{X}_{\bullet}$ is independent from $\Falg$ (cf.\ \eqref{X_bull_indep_of_Falg}), conditioning on this sigma-algebra only means that everything which is $\Falg$-measurable (e.g. the sets $H^y, y \in \mathcal{D}$) can be treated as if they were deterministic, yet the law of $\mathcal{X}_{\bullet}$ remains unchanged by the conditioning on  $\Falg$. Let us thus denote by $P_{x}(. \, | \, \Falg)$
     the law of  a  nearest neighbour random walk $X$ (starting from $x$) that is independent of $\Falg$, and by  $E_{x}[. \, | \, \Falg]$  the corresponding expectation.
  Similarly,  let
   $P_{x,x'}(. \, | \, \Falg)$ denote the joint law of two independent  nearest neighbour random walks $X$ and $X'$ (starting from $x$ and $x'$, respectively) that are independent of $\Falg$ and let  $E_{x,x'}[. \, | \, \Falg]$ denote the corresponding expectation.

Before we state the first lemma of Section \ref{subsection_small_energy}, let us  note that for any $y \in \mathcal{D}$ we have
\begin{equation}\label{capa_H_y_uuper_c_star}
\capacity(H^y) \stackrel{(\bullet)}{\leq} \gamma \cdot \Rs^2 + 1 \stackrel{(\bullet \bullet)}{\leq} \CCapaHyUUPER \cdot \gamma \cdot \Rs^2,
\end{equation}
where $(\bullet)$ follows
  from Definitions \ref{notation_H_y} and \ref{def:gamma_good_input_package}, moreover $(\bullet \bullet)$ holds with $ \CCapaHyUUPER =\tfrac{1}{\gamma_0}+1$, since
    $\gamma \geq \gamma_0$ holds by \eqref{gamma_ineq_assumptions}.

We will prove Lemma \ref{lemma:small_energy}  by giving an upper bound on $\mathbb{E} \left[ \, \energy  \left( \hitonly{\mu}_{\lozenge} \right)\, \right]$.
Recalling from \eqref{def:mu_of_the_fattened_set} that 	$\hitonly{\mu}_{\lozenge} := \sum_{y \in \mathcal{D}} \hitonly{\mu}^y_{\lozenge}$, the combination
of the next lemma and the bilinearity of the Dirichlet energy (cf.\ \eqref{eq:cond_exp_small_energy} below) will help us in achieving this goal.

\begin{lemma}[Mutual Dirichlet energy bounds]
	\label{upper:cond_exp_energy}
	There exists $C < \infty$ such that if $\Dlow \geq 16$ then for any $y, y' \in \mathcal{D}$ we have
	\begin{equation}
	\label{eq:upper:cond_exp_energy}
	\mathbb{E} \left[ \energy \left( \hitonly{\mu}^y_{\lozenge}, \hitonly{\mu}^{y'}_{\lozenge} \right) \, \Big| \, \Falg \right] \leq
	\begin{cases}
	C \cdot   (\Rs^2 \cdot \gamma \cdot \alpha) \cdot \big( 1 + \psi  \cdot \alpha \big) \cdot \xi_y, & \text{if } y = y', \\
	C \cdot \big(   \Rs^2 \cdot \gamma  \cdot \alpha \big)^2 \cdot  |y - y'|^{2-d} \cdot \xi_y \cdot \xi_{y'}, & \text{if } y \neq y'.
	\end{cases}
	\end{equation}
\end{lemma}

\begin{proof} Recall from Definition \ref{notation_H_y} that if $\xi_y = 0$ then $H_y = \emptyset$ and thus $\hitonly{\mu}^y_{\lozenge} \equiv 0$, so
the statement of Lemma \ref{upper:cond_exp_energy} trivially holds if either $\xi_y=0$ or  $\xi_{y'}=0$.

	For any $y,y' \in \mathcal{D}$ let us introduce the function $f_{y,y'} \, : \, \widetilde{W} \times \widetilde{W} \to \R^+$
	\begin{equation}
		f_{y,y'}(w,w') := \ind \left[\, w, w' \in A_{\bullet}, \, w \in \hitonly{A}[y], \, w' \in \hitonly{A}[y']\, \right] \cdot
	\energy \left( \mu^{ \varphi^y(w) }, \mu^{ \varphi^{y'}(w') } \right),
	\end{equation}
	where  $A_{\bullet}$ and $\hitonly{A}[y]$  were introduced in \eqref{worms:base} and \eqref{worms:hit_Hy_only}, respectively,
$\varphi^y(w)$ was introduced in \eqref{last_one_third_pp} and $\mu^w$ was defined in \eqref{def:local_time}.
 Note that
\begin{equation}\label{when_f_is_zero}
\text{if $y \neq y'$, then $f_{y,y'}(w,w) = 0$ for any $w \in \widetilde{W}$, since $\hitonly{A}[y] \cap \hitonly{A}[y'] = \emptyset$.}
\end{equation}
	
	Recalling the setup of Lemma \ref{prop:two_worms_joint_ppp_formula} and Definition \ref{def:law_on_input_packages}, let us also note that the conditional law of $ \energy ( \hitonly{\mu}^y_{\lozenge}, \hitonly{\mu}^{y'}_{\lozenge} ) $ given $\Falg$ is the same as the conditional law of $\sum_{i, j \in I}  f_{y,y'}( w_i, w_j )$ given $\Falg$, where $\sum_{i \in I} \delta_{w_i} \sim \mathcal{P}_v$ is independent of $\Falg$, thus we can use \eqref{eq:two_worms_joint_ppp_expectation_formula} to obtain
	\begin{align}
	\label{eq:main_exp_energy}
	\mathbb{E} \left[ \energy \left( \hitonly{\mu}^y_{\lozenge}, \hitonly{\mu}^{y'}_{\lozenge} \right) \, \Big| \, \Falg  \right]
	=
	v \cdot \sum_{\ell = \Dlow \Rs^2}^{ \Dup \Rb^2 } m(\ell) \cdot \mathcal{E}(\ell) +
	v^2 \cdot \sum_{\ell, \, \ell' = \Dlow \Rs^2}^{ \Dup \Rb^2 } m(\ell) \cdot m(\ell') \cdot \mathcal{E}(\ell,\ell'),		
	\end{align}
	where
\begin{align}
 \mathcal{E}(\ell) & =		 \sum_{x \in \mathcal{G}}   E_x \left[ f_{y,y'} \left( X[0, \ell-1], X[0, \ell-1] \right) \, | \, \Falg \right], \\
\mathcal{E}(\ell,\ell') & =
 \sum_{x,x' \in \mathcal{G}}
		E_{x,x'} \left[ f_{y,y'} \left( X[0, \ell-1], X'[0, \ell'-1] \right)\, | \, \Falg \right].
	\end{align}
Note that if $y\neq y'$ then $\mathcal{E}(\ell)=0$ by \eqref{when_f_is_zero}.
 If $y=y'$ then
	\begin{align}
	\notag	\mathcal{E}(\ell) & \stackrel{(\circ)}{\leq}
	\sum_{x \in \mathcal{G}}
	E_x \left[
	\ind \left[ T_{H^y} \leq \lfloor \ell \slash 3 \rfloor \right] \cdot
	\sum_{t, \, t' = T_{H^y} + \lfloor \ell \slash 3 \rfloor }^{ T_{H^y} + 2 \lfloor \ell \slash 3 \rfloor }
	g \left( X(t), X(t') \right)
	\, \Bigg| \, \Falg
	\right]
	   \\ \notag
		& \stackrel{ (*) }{=}
		\sum_{x \in \mathcal{G}} P_x \left( T_{H^y} \leq \lfloor \ell \slash 3 \rfloor \, | \, \Falg \right) \cdot
		E_o \left[ \sum_{t, t' = \lfloor \ell \slash 3 \rfloor}^{ 2 \lfloor \ell \slash 3 \rfloor } g \left( X(t), X'(t')\right) \right]
		 \\ & \stackrel{\eqref{eq:upper:exp_energy_of_traj}, \eqref{eq:upper:fuga_hitting_sum}}{ \leq }
		C \cdot \ell \cdot \ell \cdot \capacity(H^y) \cdot \xi_y, \label{Eell_bound}
	\end{align}
where the inequality $(\circ)$ holds because we ignored the restriction $\mathrm{Tr}(w) \subseteq \ball(3\Rb)$ from $A_{\bullet}$ and we replaced $\hitonly{A}[y]$ by the less restrictive $\{  T_{H^{y}}(w) \leq \lfloor L(w) \slash 3 \rfloor \}$, moreover $(*)$ follows from the strong Markov property at $T_{H^y}$ and the translation invariance of the simple random walk.
	\begin{align}
		\mathcal{E}(\ell,\ell') & \stackrel{(\circ \circ)}{\leq}
\sum_{x, x' \in \mathcal{G}}
	E_{x,x'} \Bigg[\,
	\ind \left[ T_{H^y}(X) \leq \lfloor \ell \slash 3 \rfloor \right] \cdot
	\ind \left[ T_{H^{y'}}(X') \leq \lfloor \ell' \slash 3 \rfloor \right]
	\cdot  \notag \\ \notag & \qquad \qquad \qquad \cdot
	\sum_{t = T_{H^y}(X) + \lfloor \ell \slash 3 \rfloor}^{ T_{H^y}(X) + 2 \lfloor \ell \slash 3 \rfloor }
	\sum_{ t' = T_{H^{y'}}(X') + \lfloor \ell' \slash 3 \rfloor }^{ T_{H^{y'}}(X') + 2 \lfloor \ell' \slash 3 \rfloor } g \left( X(t), X'(t') \right)
	\, \Bigg| \, \Falg \,
	\Bigg].\\
		& \stackrel{(**)}{\leq}
		\sum_{x, x' \in \mathcal{G}} P_x \left( T_{H^y} \leq \lfloor \ell \slash 3 \rfloor \, | \, \Falg \right) \cdot
		P_{x'} \left( T_{H^{y'}} \leq \lfloor \ell' \slash 3 \rfloor \, | \, \Falg \right)
		\cdot \notag \\ & \qquad \qquad \qquad \cdot
		\max_{z \in H^y} \max_{ z' \in H^{y'}} E_{z,z'} \left[
		\sum_{t = \lfloor \ell \slash 3 \rfloor}^{ 2 \lfloor \ell \slash 3 \rfloor}
		\sum_{ t' = \lfloor \ell' \slash 3 \rfloor }^{ 2 \lfloor \ell' \slash 3 \rfloor } g \left( X(t), X'(t') \right) \right]
		\notag \\ & \stackrel{(\diamond)}{\leq}
		C \cdot \ell \cdot \capacity(H^y) \cdot \ell' \cdot \capacity(H^{y'}) \cdot \xi_y \cdot \xi_{y'} \cdot
		\notag \\ & \qquad \qquad \qquad \cdot
		\label{eq:diff_term}
		\ell \cdot \ell' \cdot \max_{z \in H^y} \max_{z' \in H^{y'}}   \min \left\{ |z - z'|^{2-d}, \,  ( \ell + \ell' )^{1-d/2}  \right\}
	\end{align}
	where $(\circ \circ)$ follows from similar considerations as $(\circ)$ above,  in $(**)$ we used the strong Markov property at $T_{H^y}$ and at $T_{H^{y'}}$ and the conditional independence of the simple walks $X$ and $X'$ under $P_{x,x'}(. \, | \, \Falg)$ and $(\diamond)$ follows from \eqref{eq:upper:fuga_hitting_sum} using $\Dlow \geq 16$ and \eqref{eq:upper:exp_mutual_energy_of_trajs}.

	Let us examine the last term of \eqref{eq:diff_term}. Recalling from Definition \ref{notation_H_y} that $H^y \subseteq \ball(y, 3 \Rs)$ holds for any $y \in \mathcal{D}$ we can infer that if $y \neq y'$ then for any $z \in H^{y}$ and $z' \in H^{y'}$ we have $|z - z'| \geq (4 \slash 10) \cdot |y - y'|$. So by $\ell+\ell' \geq 2 \Dlow \Rs^2 \geq \Rs^2 $ we obtain the  upper bound
	\begin{equation}\label{maxmax}
		\max_{z \in H^y} \max_{z' \in H^{y'}}   \min \left\{ |z - z'|^{2-d}, \,  ( \ell + \ell' )^{1-d/2}  \right\}   \leq
					 \min \left\{ (\tfrac{4}{10} \cdot |y-y'|)^{2-d} , \,  \Rs^{2-d} \right\},	
	\end{equation}
which holds in the case $y=y'$ as well as the case $y \neq y'$.	

	Putting \eqref{Eell_bound}, \eqref{eq:diff_term} and \eqref{maxmax} back into \eqref{eq:main_exp_energy}, summing in $\ell$ and $\ell'$, using \eqref{assumption:alpha_EQ} and grouping the $\mathds{1}[y = y']$ terms together, we get
	\begin{equation*}
	 \mathbb{E} \left[ \energy \left( \hitonly{\mu}^y_{\lozenge}, \hitonly{\mu}^{y'}_{\lozenge} \right) \, \Big| \, \Falg  \right]
	\leq \begin{cases}
	C \cdot \alpha \cdot \capacity(H^y) \cdot (1 + \alpha \cdot \capacity(H^y)  \cdot \Rs^{2- d} ) \cdot \xi_y & \text{ if }  y = y', \\
	 	C \cdot \alpha^2 \cdot \capacity(H^y) \cdot \capacity(H^{y'}) \cdot
	 |y-y'|^{2-d} \cdot \xi_y \cdot \xi_{y'} & \text{ if } y \neq y',
\end{cases}
	\end{equation*}
	which implies the desired upper bounds \eqref{eq:upper:cond_exp_energy}, since for any $y \in \mathcal{D}$ we have $\capacity(H^y) \leq \CCapaHyUUPER \cdot \gamma \cdot \Rs^2$ by \eqref{capa_H_y_uuper_c_star}.
 The inequality \eqref{capa_H_y_uuper_c_star}  also implies
 $\capacity(H^y)  \cdot \Rs^{2- d} \leq \CCapaHyUUPER \cdot \gamma \cdot \Rs^{4-d} \leq \CCapaHyUUPER \cdot \psi$, using  assumption \eqref{gamma_ineq_assumptions} of Lemma \ref{lemma_doubling_gamma} in the last inequality, giving \eqref{eq:upper:cond_exp_energy}.
\end{proof}

In order to achieve \eqref{eq:lemma:small_energy}, we will give an upper bound on $\mathbb{E} \left[ \, \energy  \left( \hitonly{\mu}_{\lozenge} \right)\, \right]$
using Lemma \ref{upper:cond_exp_energy}:
\begin{multline} \label{eq:cond_exp_small_energy}
	\mathbb{E} \left[ \, \energy \left( \hitonly{\mu}_{\lozenge} \right)  \, \big| \, \Falg \right]
	 \stackrel{\eqref{def:mu_of_the_fattened_set} }{=}
	\sum_{y \in \mathcal{D}} \mathbb{E} \left[ \energy \left( \hitonly{\mu}_{\lozenge}^y \right) \, \big| \, \Falg \right] +
	\sum_{y \in \mathcal{D}} \sum_{y' \in \mathcal{D} \setminus \{ y \} } \mathbb{E} \left[ \energy \left( \hitonly{\mu}^y_{\lozenge}, \hitonly{\mu}^{y'}_{\lozenge} \right) \, \Big| \, \Falg  \right] \stackrel{\eqref{eq:upper:cond_exp_energy} }{\leq}
	 \\
	C \cdot  ( \Rs^2 \cdot \gamma \cdot \alpha)  \cdot
	\left( \big( 1 + \psi  \cdot \alpha \big) \cdot \sum_{y \in \mathcal{D}} \xi_y  +
	 (\Rs^2 \cdot \gamma \cdot \alpha) \cdot \sum_{y \in \mathcal{D}} \sum_{ y' \in \mathcal{D} \setminus \{ y \} } \xi_y \cdot \xi_{y'} \cdot  | y -y' |^{ 2 - d } \right).
\end{multline}
We bound $\mathbb{E} \left[ \, \energy  \left( \hitonly{\mu}_{\lozenge} \right)\, \right]$ by taking the expectation of the r.h.s.\ of \eqref{eq:cond_exp_small_energy}, thus
   it only remains to bound $\mathbb{E} \left[ \sum_{y \in \mathcal{D}} \xi_y \right]$ and $\mathbb{E} \left[ \sum_{y \in \mathcal{D}} \sum_{ y' \in \mathcal{D} \setminus \{ y \} } \xi_y \cdot \xi_{y'} \cdot  | y -y' |^{ 2 - d } \right]$.
       By \eqref{nota:xi_y} and \eqref{nota:chi_y} we have $\xi_y \leq \ind \left[ T_{\ball(y, \Rs)}(\pinkwalk) < \infty \right]$ for each $y \in \mathcal{D}$, so we can use Lemma \ref{upper:numb_of_visited_boxes} with parameters $r = \Rs$ and $R = \Rb$  to obtain
\begin{equation*}
	\mathbb{E} \left[\, \energy  \left( \hitonly{\mu}_{\lozenge} \right)\, \right] \leq
	C  \cdot (\Rb^2 \cdot \gamma \cdot \alpha) \cdot
	\left(
		1 + \psi  \cdot \alpha +
		 (\Rs^2 \cdot \gamma \cdot \alpha)  \cdot  \Rs^{2-d}
	\right)	
	 \stackrel{ \eqref{gamma_ineq_assumptions} }{\leq }
	\CSmallEnergyUPPER  \cdot (\Rb^2 \cdot \gamma \cdot \alpha) \cdot \Big( 1  + \psi \cdot \alpha \Big),
\end{equation*}
from which the statement of Lemma \ref{lemma:small_energy}  follows by Markov's inequality.

\subsection{Big total local time}\label{subsection_big_local_time}
The goal of Section \ref{subsection_big_local_time} is to prove Lemma \ref{lemma:big_total_measure}. After introducing some notation, we will state three key lemmas (Lemmas \ref{prop:cond_restrict_measure_total}, \ref{lower:no_of_visited_good_boxes} and \ref{lemma_expect_bound_interactions}), then we deduce the proof of Lemma \ref{lemma:big_total_measure} from them, and finally we prove the three lemmas in Sections \ref{subsub_exp_var},\ref{subsub_xi_low_good_box} and \ref{subsub_double_visits}, respectively.

Recalling the notation of $\hitonly{\mathcal{X}}[y]$ and $\hitonly{\mathcal{X}}_{\lozenge}[y]$ from \eqref{worms:hit_Hy_only} and \eqref{last_one_third_pp}, let us define
\begin{equation}\label{some_notation_needed_for_big_total_local_time}
	\hitonly{\Sigma}_\lozenge[y] := \hitonly{\mu}_{\lozenge}^y(\mathbb{Z}^d) = \Sigma^{\hitonly{\mathcal{X}}_{\lozenge}[y]}, \qquad
	\hitonly{\Sigma}[y] := \Sigma^{\hitonly{\mathcal{X}}[y]}, \qquad y \in \mathcal{D}.
\end{equation}
In words: $\hitonly{\Sigma}_\lozenge[y]$ and $\hitonly{\Sigma}[y]$ are the total local times of the point measures $\hitonly{\mathcal{X}}_{\lozenge}[y]$ and $\hitonly{\mathcal{X}}[y]$ of worms, respectively.

By \eqref{def:mu_of_the_fattened_set} we have $\hitonly{\Sigma}_{\lozenge} = \sum_{y \in \mathcal{D}} \hitonly{\Sigma}_{\lozenge}[y]$.
Let us also note that $\hitonly{\Sigma}_\lozenge[y] \geq \frac{1}{3} \hitonly{\Sigma}[ y]$, since we have kept at least one third of each worm in $\hitonly{\mathcal{X}}[y]$ when we created $\hitonly{\mathcal{X}}_{\lozenge}[y]$ in \eqref{last_one_third_pp}. As a consequence,
\begin{equation}
	\label{nota:hitonly_Sigma}
 \text{ it is enough to prove Lemma \ref{lemma:big_total_measure} with }	\hitonly{\Sigma} :=\sum_{y \in \mathcal{D}}\hitonly{\Sigma}[y]	
 \text{ in place of $\hitonly{\Sigma}_\lozenge$.}
\end{equation}
Our goal is to show that $\mathbb{P}(\hitonly{\Sigma} \leq K)$ is small, i.e., we want $\hitonly{\Sigma}$ to be big enough with high enough probability. There are multiple technical difficulties that we have to overcome:
\begin{enumerate}[(i)]
\item \label{problem_number_chi} If we want $\hitonly{\Sigma}[y] >0$ then we need $H_y\neq \emptyset$ (cf.\ \eqref{worms:hit_Hy_only}), thus we want $\xi_y=1$ (cf.\ Definition \ref{notation_H_y}), hence we want $\chi_y=1$ (cf.\ \eqref{nota:xi_y}), therefore we certainly want to show that $\sum_{y \in \mathcal{D}} \chi_y$ is big enough with high enough probability. We will do this using Lemma \ref{lower:no_of_visited_boxes}.
\item \label{problem_crowd} $\hitonly{\Sigma}[y]$ is the total length of worms in $\hitonly{\mathcal{X}}[y]$, and the worms in $\hitonly{\mathcal{X}}[y]$ cannot hit any $H_{\widetilde{y}}$ where $\widetilde{y} \in \mathcal{D} \setminus \{ y \}$ (cf.\ \eqref{worms:hit_Hy_only}), thus we want to guarantee that the set $\{ y \in \mathcal{D} \, : \, \chi_y=1 \}$ is well spread-out, so that the total length of worms that hit multiple sets of form $H_{y'}, y' \in \mathcal{D}$ is not too big. We will do this using \eqref{eq:upper:weighted_exp_visited_boxes}.
\item \label{problem_dummy} The summands $\hitonly{\Sigma}[y], y \in \mathcal{D}$ are independent given $\Falg$ (cf.\ \eqref{hitonly_X_y_indep}), thus, after conditioning on $\Falg$, we would like to use Chebyshev's inequality  to show that their sum is big with high probability. This approach would involve an upper bound on
    $\Var \left( \hitonly{\Sigma}[y] \, \big| \, \Falg \right)$ and a lower bound on $\mathbb{E} \left[ \hitonly{\Sigma}[y] \, \big| \, \Falg \right]$. However, the latter bound is problematic, since the conditioning on $\Falg$ might put us in a situation where $H_y$ is surrounded with other sets of form $H_{\widetilde{y}}$, and the effect of such a ``crowd'' makes it hard for us to give a lower bound on the total length of worms that only hit $H_y$. We will overcome this  problem with a trick: first we introduce ``dummy'' worms (or rather, worm lengths) in \eqref{kieg_vava} to make up for  the deficit of $\mathbb{E} \left[ \hitonly{\Sigma}[y] \, \big| \, \Falg \right]$ caused by the crowd, and since each dummy worm hits multiple sets of form $H_{y'}, y' \in \mathcal{D}$, we will control their effect using the strategy outlined in \eqref{problem_crowd}.
\end{enumerate}

Now let us make the above ideas rigorous.
Recalling \eqref{worms:base}, for any $y \in \mathcal{D}$ let us define
\begin{align}
	\label{worms:hit_one}
	\mathcal{X}[y] & := \mathcal{X}_{\bullet} \ind \big[ \, T_{H^y}(w) < \lfloor L(w)/3 \rfloor  \, \big], & \Sigma[y] := \Sigma^{\mathcal{X}[y]}, \\
	\label{worms:hit_more_than_one}
	\overline{\mathcal{X}}[y] & :=  \mathcal{X}_{\bullet} \ind \big[ \, T_{H^y}(w) < \lfloor L(w)/3 \rfloor, \; \exists \, y' \in \mathcal{D}\setminus \{y\} \,: \,   T_{H^{y'}}(w)< L(w) \, \big],
	& \overline{\Sigma}[y]  :=\Sigma^{\overline{\mathcal{X}}[y]}.
\end{align}
In words: the point measure $\mathcal{X}[y]$ consists of those worms of $\mathcal{X}_{\bullet}$ whose first third segment hits $H^y$, while $\overline{\mathcal{X}}[y]$
consists of those worms of $\mathcal{X}[y]$ that also hit some other $H^{y'}$ as well. $\Sigma[y]$ and  $\overline{\Sigma}[y]$ denote the total local times of the point measures
$\mathcal{X}[y]$ and $\overline{\mathcal{X}}[y]$ of worms, respectively.

Note that $\mathcal{X}[y] = \hitonly{\mathcal{X}}[y]+\overline{\mathcal{X}}[y]$, and thus
\begin{equation} \label{sigma_y_sum_identity}
\Sigma[y]=\hitonly{\Sigma}[y]+\overline{\Sigma}[y].
\end{equation}

Recall the definition of the sigma-algebra $\Falg$ from Definition \ref{nota:Falg}.
 For any $y \in \mathcal{D}$,
\begin{equation}\label{kieg_vava}
\begin{array}{c}
\text{ let $\overline{\Sigma}_c[y]$ denote a random variable which has the same conditional distribution  } \\
\text{ (given $\Falg$) as $\overline{\Sigma}[y]$, but conditionally independent of everything else given $\Falg$.}
\end{array}
\end{equation}

Let us also define
\begin{equation}
	\label{nota:copy_Sigmay}
	\Sigma_c[y] := \hitonly{\Sigma}[y] + \overline{\Sigma}_c[y]
\end{equation}
and observe that it follows from \eqref{hitonly_X_y_indep}, \eqref{some_notation_needed_for_big_total_local_time}, \eqref{sigma_y_sum_identity} and \eqref{kieg_vava} that
\begin{equation}\label{fuggetlenito_trukk}
\begin{array}{l}  \text{  $\Sigma_c[y], \, y \in \mathcal{D}$  are conditionally independent of each other given $\Falg$ and} \\
\text{ $\Sigma_c[y]$ has the same conditional distribution as $\Sigma[ y]$ given $\Falg$ for each $y \in \mathcal{D}$.}
\end{array}
\end{equation}

\noindent With the notation introduced above, we can write
\begin{equation}
	\label{eq:hitonly_Sigma_approach}
	\hitonly{\Sigma} = \sum_{y \in \mathcal{D}} \Sigma_c[y] \cdot \xi_y - \sum_{y \in \mathcal{D}} \overline{\Sigma}_c[y] \cdot \xi_y.
\end{equation}

In order to prove Lemma \ref{lemma:big_total_measure},
 we will give a lower bound on the first sum on the r.h.s.\ of \eqref{eq:hitonly_Sigma_approach} and an upper bound on the second sum on the r.h.s.\ of \eqref{eq:hitonly_Sigma_approach}. We will now state the three key lemmas mentioned above, then deduce the proof of Lemma \ref{lemma:big_total_measure} from them.  We will prove the three lemmas in Sections \ref{subsub_exp_var},\ref{subsub_xi_low_good_box} and \ref{subsub_double_visits},  respectively.

The first  lemma will help us to give a lower bound on the first sum on the r.h.s.\ of \eqref{eq:hitonly_Sigma_approach}.

\begin{lemma}[Bounds on the expectation/variance of the total length of worms hitting $H^y$]
	\label{prop:cond_restrict_measure_total}
	There exist constants $\CCondExpRestrictMeasureTotalLOWER, \CLowerFugaHitDELTALOWER > 0$ and $\CCondVarRestrictMeasureTotalUPPER , \CLowerFugaHitDELTAUPPER < \infty$ such that the following bounds hold.
	\begin{enumerate}
		\item If $\Dlow \geq \CLowerFugaHitDELTALOWER$ and $\Dup \leq \CLowerFugaHitDELTAUPPER$ then for any $y \in \mathcal{D}$, we have
		\begin{equation}
		\label{eq:lower:cond_exp_restrict_measure_total}
		\mathbb{E} \left[ \Sigma[y] \, \big| \, \Falg \right] \geq \CCondExpRestrictMeasureTotalLOWER \cdot (\Rs^2 \cdot \gamma \cdot \alpha) \cdot \xi_y.
		\end{equation}
		\item If $\Dlow \geq 16$, then for any $y \in \mathcal{D}$ we have
		\begin{equation}
		\label{eq:upper:cond_var_restrict_measure_total}
		\Var \left( \Sigma[y] \, \big| \, \Falg \right) \leq
		\CCondVarRestrictMeasureTotalUPPER \cdot \Rb^2  \cdot (\Rs^2 \cdot \gamma \cdot \alpha) \cdot \xi_y
		\end{equation}
	\end{enumerate}
\end{lemma}

We will prove Lemma \ref{prop:cond_restrict_measure_total} in Section \ref{subsub_exp_var}.

 Let us denote by $\xi$ the number of good boxes visited by $\pinkwalk$:
  \begin{equation}\label{xi_def_eq}
  \xi:=\sum_{y \in \mathcal{D}} \xi_y.
  \end{equation}

The second lemma will also help us to  lower bound  the first sum on the r.h.s.\ of \eqref{eq:hitonly_Sigma_approach}.

\begin{lemma}[Lower bound on the number of good boxes visited by $\pinkwalk$]
	\label{lower:no_of_visited_good_boxes}
If $\underline{\Delta} \geq \CLowerNOVisitedBoxesRATIOLOWER \vee (300/ \CLowerNOVisitedBoxesLOWER ) $ (where $ \CLowerNOVisitedBoxesLOWER $ and $ \CLowerNOVisitedBoxesRATIOLOWER $ appear in the statement of Lemma \ref{lower:no_of_visited_boxes})
 then we have
	\begin{equation}
	\label{eq:lower:no_of_visited_good_boxes}
	\mathbb{P} \left( \xi \geq  \CLowerNOVisitedBoxesLOWER /2 \cdot \Rb^2 \slash \Rs^2 \right) \geq 0.98.
	\end{equation}
\end{lemma}

We will prove Lemma \ref{lower:no_of_visited_good_boxes} in Section \ref{subsub_xi_low_good_box}.

The third lemma upper bounds the expectation of the second sum on the r.h.s.\ of \eqref{eq:hitonly_Sigma_approach}.

\begin{lemma}[Bounding the total length of worms that visit multiple good sets]
\label{lemma_expect_bound_interactions} There exists $ \CExpectBoundInteractionsUPPER <+\infty$ such that
\begin{equation}\label{eq_expect_bound_interactions}
\mathbb{E} \left[ \sum_{y \in \mathcal{D}} \overline{\Sigma}_c[y] \cdot \xi_y \right]
 \leq \CExpectBoundInteractionsUPPER \cdot (\Rb^2 \cdot \gamma \cdot \alpha  )\cdot \psi.
 \end{equation}
\end{lemma}

We will prove Lemma \ref{lemma_expect_bound_interactions} in Section \ref{subsub_double_visits}.

\begin{proof}[ Proof of Lemma \ref{lemma:big_total_measure}]
 Recall from \eqref{nota:hitonly_Sigma} the notation $\hitonly{\Sigma}$  and the observation that it is enough to prove \eqref{eq:lemma:big_total_measure} using $\hitonly{\Sigma}$ instead of $\hitonly{\Sigma}_\lozenge$.
 For any  $K > 0$, by \eqref{eq:hitonly_Sigma_approach} we have
\begin{equation}
	\label{eq:prob_decomp_bound_on_star_Sigma}
	\mathbb{P} \left( \hitonly{\Sigma} \leq K \right)
	\leq
	\mathbb{P} \left( \sum_{y \in \mathcal{D}} \Sigma_c[y] \cdot \xi_y \leq 3 K \slash 2  \right) +
	\mathbb{P} \left( \sum_{y \in \mathcal{D}} \overline{\Sigma}_c[y] \cdot \xi_y \geq K \slash 2 \right),
\end{equation}
where $\overline{\Sigma}_c[y]$ was introduced in \eqref{kieg_vava} and $\Sigma_c[y]$ was introduced in \eqref{nota:copy_Sigmay}.

We bound the second term of the r.h.s.\ of \eqref{eq:prob_decomp_bound_on_star_Sigma} using Markov's inequality:
\begin{equation}
	\label{eq:prob_decomp_bound_on_star_Sigma_SECONDTERM}
	\mathbb{P} \left( \sum_{y \in \mathcal{D}} \overline{\Sigma}_c[y] \cdot \xi_y \geq K \slash 2 \right)
	\leq
	\frac{2}{K} \cdot \mathbb{E} \left[ \sum_{y \in \mathcal{D}} \overline{\Sigma}_c[y] \cdot \xi_y \right]
	\stackrel{ \eqref{eq_expect_bound_interactions} }{\leq}
	\frac{2 \cdot \CExpectBoundInteractionsUPPER \cdot (\Rb^2 \cdot \gamma \cdot \alpha) \cdot \psi}{K}.
\end{equation}
Next we will bound the first term on the r.h.s.\ of \eqref{eq:prob_decomp_bound_on_star_Sigma}. Recalling the notion of $\Falg$ from Definition \ref{nota:Falg}, we observe that $\xi$ (defined in \eqref{xi_def_eq}) is $\Falg$-measurable.
 Using \eqref{fuggetlenito_trukk} and  Lemma \ref{prop:cond_restrict_measure_total} we obtain that if $\Dlow \geq \CLowerFugaHitDELTALOWER \vee 16 $ and $\Dup \leq \CLowerFugaHitDELTAUPPER$ then
 \begin{equation}\label{E_var_xi_bounds}
 \mathbb{E}\left( \sum_{y \in \mathcal{D}} \Sigma_c[y] \cdot \xi_y   \Biggm|  \Falg \right) \geq \CCondExpRestrictMeasureTotalLOWER \cdot (\Rs^2 \cdot \gamma \cdot \alpha) \cdot \xi, \;
 \Var \left( \sum_{y \in \mathcal{D}} \Sigma_c[y] \cdot \xi_y   \Biggm|  \Falg \right) \leq
 \CCondVarRestrictMeasureTotalUPPER  \cdot \Rb^2  \cdot (\Rs^2 \cdot \gamma \cdot \alpha) \cdot \xi,
 \end{equation}

 Now  Chebysev's inequality and \eqref{E_var_xi_bounds} together give
 \begin{equation}\label{cheb_gives}
   \mathbb{P} \left( \sum_{y \in \mathcal{D}} \Sigma_c[y] \cdot \xi_y \leq \tfrac{3}{2} K \Biggm| \Falg  \right)\leq
   \frac{\CCondVarRestrictMeasureTotalUPPER  \cdot \Rb^2  \cdot (\Rs^2 \cdot \gamma \cdot \alpha) \cdot \xi}
   {\left( \left( \CCondExpRestrictMeasureTotalLOWER \cdot (\Rs^2 \cdot \gamma \cdot \alpha) \cdot \xi- 3 K \slash 2  \right) \vee 0 \right)^2}.
 \end{equation}
 Recall from Lemma \ref{lower:no_of_visited_good_boxes} that if  $\Dlow \geq \CLowerNOVisitedBoxesRATIOLOWER \vee (300/ \CLowerNOVisitedBoxesLOWER )$ then $\mathbb{P}\left( \xi <  \CLowerNOVisitedBoxesLOWER /2 \cdot \Rb^2 \slash \Rs^2 \right) \leq 0.02 $. Also note that the r.h.s.\ of \eqref{cheb_gives} is a non-increasing function of $\xi$. Putting these together with the total law of expectation and assuming that
 $ 3 K \slash 2 <  \CCondExpRestrictMeasureTotalLOWER \cdot (\Rs^2 \cdot \gamma \cdot \alpha) \cdot ( \CLowerNOVisitedBoxesLOWER /2 \cdot \Rb^2 \slash \Rs^2)  $, we obtain
\begin{equation}
	\label{eq:prob_decomp_bound_on_star_Sigma_FIRSTTERM}
	\mathbb{P} \left( \sum_{y \in \mathcal{D}} \Sigma_c[y] \cdot \xi_y \leq \tfrac{3}{2} K   \right)
	\stackrel{\eqref{eq:lower:no_of_visited_good_boxes}}{ \leq }
	0.02 +  \frac{\CCondVarRestrictMeasureTotalUPPER \cdot \Rb^2  \cdot (\Rs^2 \cdot \gamma \cdot \alpha) \cdot ( \CLowerNOVisitedBoxesLOWER /2 \cdot \Rb^2 \slash \Rs^2) }
{\left( \CCondExpRestrictMeasureTotalLOWER \cdot (\Rs^2 \cdot \gamma \cdot \alpha) \cdot ( \CLowerNOVisitedBoxesLOWER /2 \cdot \Rb^2 \slash \Rs^2) - 3 K \slash 2   \right)^2}.
	\end{equation}
Substituting \eqref{eq:prob_decomp_bound_on_star_Sigma_SECONDTERM} and \eqref{eq:prob_decomp_bound_on_star_Sigma_FIRSTTERM}  back into \eqref{eq:prob_decomp_bound_on_star_Sigma} we finish the proof of Lemma \ref{lemma:big_total_measure}.
\end{proof}

\subsubsection{On the expectation/variance of the total length of worms hitting $H^y$}\label{subsub_exp_var}

\begin{proof}[Proof of Lemma \ref{prop:cond_restrict_measure_total}] Recall \eqref{X_bull_indep_of_Falg} and the notation $P_{x}( \cdot \, | \, \Falg)$ and $E_{x}[ \cdot \, | \, \Falg]$ introduced at the beginning of Section \ref{subsection_small_energy}.
	Let us fix $y \in \mathcal{D}$ and define the function $f_y \, : \, \widetilde{W} \rightarrow \R^+$ by
	\begin{equation}
	\label{nota:fy_for_Sigmay}
	f_y (w) := \ind \left[\, w \in A_{\bullet}, \, T_{H^y}(w) < \lfloor L(w) \slash 3 \rfloor \, \right] \cdot L(w).
	\end{equation}
	By Definition \ref{def:law_on_input_packages}, the conditional law of $\Sigma[y]$ given $\Falg$ is the same as the conditional law of $\sum_{i \in I} f_y(w_i)$ given $\Falg$, where $\sum_{i \in I} \delta_{w_i} \sim \mathcal{P}_v$ is independent of $\Falg$. As a result, we can use Lemma \ref{prop:one_worm_ppp_formulas} to prove \eqref{eq:lower:cond_exp_restrict_measure_total} and \eqref{eq:upper:cond_var_restrict_measure_total}.
		Using \eqref{eq:one_worm_ppp_expectation_formula} we get
	\begin{align*}
	\mathbb{E} \left[ \Sigma[y] \, \big| \, \Falg \right]
	& =
	v \cdot \sum_{\ell = \Dlow \cdot \Rs^2}^{ \Dup \cdot \Rb^2 } \ell \cdot m(\ell) \cdot \sum_{x \in \mathcal{G}} P_x \left( T_{H^y} \leq \lfloor \ell \slash 3 \rfloor, \, T_{\ball( 3 \Rb )^c} > \ell \, \big| \, \Falg \right)
	\\ & \stackrel{(*)}{\geq}
	v \cdot \sum_{\ell = \Dlow \cdot \Rs^2 }^{ \Dup \cdot \Rb^2 } \ell \cdot m(\ell) \cdot \CLowerFugaHitLOWER \cdot \ell \cdot \capacity(H^y)
	\stackrel{ \eqref{assumption:alpha_EQ} }{=}
	\CLowerFugaHitLOWER \cdot \alpha \cdot \capacity(H^y)
	\end{align*}
	where in $(*)$ we used \eqref{nota:ground_for_worms} and Lemma \ref{lower:fuga_hitting_sum} with the values of $\CLowerFugaHitDELTALOWER$, $ \CLowerFugaHitDELTAUPPER $ and  $\CLowerFugaHitLOWER$ specified there. From this lower bound the desired inequality \eqref{eq:lower:cond_exp_restrict_measure_total} follows since $\capacity(H^y) \geq \Rs^2 \cdot \gamma \cdot \xi_y $ holds by Definitions \ref{def:gamma_good_input_package} and \ref{notation_H_y}.
		Similarly, by \eqref{eq:one_worm_ppp_variance_formula} we have
	\begin{multline*}
	\Var \left( \Sigma[y] \, \big| \, \Falg \right)  =
	v \cdot \sum_{\ell = \Dlow \Rs^2}^{ \Dup \Rb^2 } \ell^2 \cdot m(\ell) \cdot
	\sum_{x \in \mathcal{G}} P_x \left( T_{H^y} \leq \lfloor \ell \slash 3 \rfloor, \, T_{\ball(3 \Rb)^c} > \ell \, \big| \, \Falg \right)
	 \stackrel{ \eqref{eq:upper:fuga_hitting_sum} }{\leq} \\
	C \cdot v \cdot \sum_{\ell = \Dlow \Rs^2}^{ \Dup \Rb^2 } \ell^2 \cdot m(\ell) \cdot \ell \cdot \capacity(H^y)
	\stackrel{(*)}{\leq}
	C \cdot \alpha \cdot \capacity(H^y)  \cdot \Rb^2 \stackrel{ \eqref{capa_H_y_uuper_c_star} }{\leq} \CCondVarRestrictMeasureTotalUPPER \cdot \Rb^2  \cdot (\Rs^2 \cdot \gamma \cdot \alpha) ,
	\end{multline*}
	where $(*)$ follows from the trivial upper bound $\ell^3 \leq \Dup \cdot \Rb^2 \cdot \ell^2 \leq \Rb^2 \cdot \ell^2$ and \eqref{assumption:alpha_EQ}. The upper bound \eqref{eq:upper:cond_var_restrict_measure_total} follows once we observe that $\xi_y=0$ implies $H^y=\emptyset$ (cf.\ Definition \ref{notation_H_y}), which in turn implies
$\Sigma[y]=0$ (cf.\ \eqref{worms:hit_one}).
\end{proof}

\subsubsection{Lower bound on the number of good boxes visited by $\pinkwalk$}\label{subsub_xi_low_good_box}

The goal of Section \ref{subsub_xi_low_good_box} is to prove Lemma \ref{lower:no_of_visited_good_boxes}. The plan is simple:
in order to give a lower bound on the number $\sum_{y \in \mathcal{D}} \xi_y$ of good boxes visited by $\pinkwalk$,
we will (i)  give a lower bound the number $\sum_{y \in \mathcal{D}} \chi_y$ of boxes visited by $\pinkwalk$ using  Lemma \ref{lower:no_of_visited_boxes} and (ii)
show that with high probability at least half of these boxes are good using the hypothesis \eqref{induction_hypotheses} of Lemma \ref{lemma_doubling_gamma}. In order to make these ideas
rigorous, we first introduce some notation that allows us to rewrite the quantities $\sum_{y \in \mathcal{D}} \xi_y$ and $\sum_{y \in \mathcal{D}} \chi_y$ in a more manageable form (cf.\ \eqref{old_new_form_of_sum_xi} below).

\medskip

 Slightly abusing the notation of Definition \ref{def:law_on_input_packages}, let us naturally extend the worm $\pinkwalk=(Z(t))_{0\leq t \leq T}$ to $t=T+1,T+2,\dots$. More precisely, in Section \ref{subsub_xi_low_good_box} let $(Z(t))_{t \geq 0}$ denote a simple random walk starting from $z$. Given this extended version of $\pinkwalk$, let us also extend Definition \ref{notation:pink_packages} accordingly:  we may define $z^y$, $T^y$, $\pinkwalk^y$ and $\worms^y$ using the formulas \eqref{def_eq_z_y}--\eqref{def_eq_worms_y} for any $y \in 10 \Rs \mathbb{Z}^d$ satisfying
$T_{\ball(y, \Rs )}(\pinkwalk) < \infty$. We can thus define the indicators
\begin{equation}
\label{nota:pi_y}
		\pi_y  :=  \ind \big[\, T_{\ball(y, \Rs )}(\pinkwalk) < \infty, \;  (\worms^y, \pinkwalk^y) \text{ is } (y,\Rs,z^y,\gamma) \text{-good} \, \big], \quad  y \in 10  \Rs  \Z^d.
\end{equation}

	Let $\widetilde{\mathcal{B}} := \bigcup_{y \in 10\Rs \mathbb{Z}^d} \ball(y, \Rs)$ and let us recursively define a sequence of stopping times $\tau_k$ and a sequence of indices $y_k \in 10\Rs \mathbb{Z}^d$ ($k \in \N$) as follows. Let $\tau_0:=0$ and for any $k \geq 1$ let
	\begin{equation}
		\label{def_eq_tau_k_recursive}
		\tau_k  := \min \left\{ t \geq \tau_{k-1} \, : \, Z(t) \in \widetilde{\mathcal{B}} \setminus \cup_{\ell=1}^{k-1} \ball(y_{\ell}, \Rs) \right\},
		\quad
		y_k := \arg \min_{ y \in 10\Rs \mathbb{Z}^d  } \left| y - Z( \tau_k ) \right|.
	\end{equation}
\noindent In words: $\tau_k$ is the first time when $\pinkwalk$ hits a ball of form $\ball(y, \Rs), y \in 10\Rs \mathbb{Z}^d$ which is different from the ones visited at times $\tau_1,\dots,\tau_{k-1}$,
 moreover $y_k$ is the center of this ball.

Now that $(Z(t))_{t \geq 0}$ denotes a random walk, we have $\mathbb{P}(\tau_k<+\infty)=1$ for all $k \in \mathbb{N}$.

 Note that  we have $\{y_k, k \in \mathbb{N}\}=\{ y \in 10\Rs \mathbb{Z}^d \, : \, T_{\ball(y, \Rs )}(\pinkwalk) < \infty \}$, moreover
 $k\neq k' \in \mathbb{N}$ implies $y_k\neq y_{k'}$. Also note that  $T_{\ball(y_k, \Rs )}(\pinkwalk)=\tau_k$ for $k \in \mathbb{N}$.

	Let us define
	\begin{equation}
		\mathcal{K} := \max \{\, k \in \N \, : \, \tau_k \leq (\Rb^2 - \Rs^2) \wedge T_{ \ball(2 \Rb)^c }(\pinkwalk)\, \}.
	\end{equation}
	Recalling the definition of $\mathcal{D}$ from \eqref{def:D_index_set} and of $\chi_y, y \in \mathcal{D}$ from \eqref{nota:chi_y}, observe that for any $y \in \mathcal{D}$ we have $\chi_y=1$ if and only if $y=y_\ell$ for some $\ell \in \{1,\dots,\mathcal{K}\}$, thus we have
	\begin{equation}
		\label{old_new_form_of_sum_xi}
		\sum_{y \in \mathcal{D}} \chi_y = \mathcal{K}, \qquad
		\xi \stackrel{ \eqref{xi_def_eq} }{=}  \sum_{y \in \mathcal{D}} \xi_y  \stackrel{ \eqref{nota:chi_y}, \eqref{nota:xi_y}, \eqref{nota:pi_y}  }{=}  \sum_{y \in \mathcal{D}} \chi_y \cdot \pi_y = \sum_{\ell=1}^{\mathcal{K}} \pi_{y_\ell}.
	\end{equation}

 \begin{claim}\label{claim_coupling_statement}
 The sequence $\pi_{y_k}, k \in \mathbb{N}$ can be coupled to an i.i.d.\ Bernoulli  sequence
			  $\pi^*_k, k \in \N$ with  parameter $3 \slash 4$ in such a way that $\pi_{y_k} \geq \pi^*_k,  k \in \mathbb{N}$.
	\end{claim}
Before we prove Claim \ref{claim_coupling_statement}, let us derive the proof of Lemma \ref{lower:no_of_visited_good_boxes} using it.
	
\begin{proof}[Proof of Lemma \ref{lower:no_of_visited_good_boxes}]
	Using the constant $\CLowerNOVisitedBoxesLOWER > 0$ from Lemma \ref{lower:no_of_visited_boxes}, let us introduce the shorthand $K:=  \CLowerNOVisitedBoxesLOWER \cdot \Rb^2 \slash \Rs^2$.
	We will  argue that if $\underline{\Delta} \geq 300/ \CLowerNOVisitedBoxesLOWER $ (and thus $K \geq 300$ by \eqref{big_Dlow_big_ratio}) then
	\begin{equation}\label{binom_bound}
		\mathbb{P}\left( \sum_{\ell=1}^{K}  \pi_{y_\ell}  \geq \tfrac{1}{2} K \right)  \geq 0.99.
	\end{equation}
	Indeed, by Claim \ref{claim_coupling_statement}, it is enough to prove this bound with $\sum_{\ell=1}^{K}  \pi^*_\ell$ in place of $\sum_{\ell=1}^{K}  \pi_{y_\ell}$, but
	$\sum_{\ell=1}^{K}  \pi^*_\ell \sim \mathrm{BIN}(K,\tfrac{3}{4})$, thus by Chebyshev's inequality  \eqref{binom_bound}  holds if $K\geq 300$.
	
	We are now ready to prove that \eqref{eq:lower:no_of_visited_good_boxes} holds if $\underline{\Delta} \geq \CLowerNOVisitedBoxesRATIOLOWER \vee (300/ \CLowerNOVisitedBoxesLOWER ) $ (where $ \CLowerNOVisitedBoxesRATIOLOWER $ appears in the statement of Lemma \ref{lower:no_of_visited_boxes}):
	\begin{multline}
		\mathbb{P} \left( \xi \geq  \CLowerNOVisitedBoxesLOWER /2 \cdot \Rb^2 \slash \Rs^2 \right) \stackrel{ \eqref{old_new_form_of_sum_xi} }{=}
		\mathbb{P} \left( \sum_{\ell=1}^{\mathcal{K}} \pi_{y_\ell} \geq \tfrac{1}{2} K \right) \geq
		\mathbb{P} \left( \sum_{\ell=1}^{K} \pi_{y_\ell} \geq \tfrac{1}{2} K \right) - \mathbb{P}\left( \mathcal{K}<K \right) \\
		\stackrel{  \eqref{binom_bound}, \eqref{old_new_form_of_sum_xi} }{\geq}  0.99- \mathbb{P}\left( \sum_{y \in \mathcal{D}} \chi_y <K \right)
	    \stackrel{ \eqref{eq:lower:no_of_visited_boxes}  }{\geq} 0.99-0.003 \geq 0.98.
	\end{multline}
	The proof of Lemma \ref{lower:no_of_visited_good_boxes} is complete (using the statement of Claim \ref{claim_coupling_statement}).
\end{proof}

\begin{proof}[Proof of Claim \ref{claim_coupling_statement}]
 For any $k \in \mathbb{N}$ let us denote by $\mathcal{G}_k$ the sigma-field generated by $\pinkwalk[0, \tau_k]$ and the input packages $(\worms^{y_\ell}, \pinkwalk^{y_\ell}), \ell=1,\dots,k-1$. Observe that  the random variables $\pi_{y_1}, \dots, \pi_{y_{k-1}}$ are all $\mathcal{G}_k$-measurable. Moreover, by \eqref{def_eq_z_y} and \eqref{def_eq_tau_k_recursive} we have $z^{y_k}=Z(\tau_k)$, thus $z^{y_k}$ is also $\mathcal{G}_k$-measurable for any $k \geq 1$.

	Recalling Definition \ref{def:law_on_input_packages}, next we want to  deduce that for any $k  \in \mathbb{N}$,
\begin{equation}\label{condindep_package}
\text{given $\mathcal{G}_k$, the conditional law of $(\worms^{y_k}, \pinkwalk^{y_k})$ is $\mathcal{P}_{y_k,\Rs,z^{y_k},v}$.}
\end{equation}
 In order to see this, we use the strong Markov property of the simple random walk $\pinkwalk$ at time $\tau_k$ to deduce that $\pinkwalk^{y_k}$ has the appropriate conditional law given  $\mathcal{G}_k$. We also use the facts that $\worms^{y_k}=\widehat{\worms} \ind \big[ w(0) \in \ball(y_k, 2\Rs), \, L(w) \leq \Rs^2 \big]$, and that $\ball(y_k, 2\Rs)$ is disjoint from $\cup_{\ell=1}^{k-1}\ball(y_\ell, 2\Rs)$. Thus, conditionally on $\mathcal{G}_k$, $\worms^{y_k}$ is a PPP of worms with intensity measure $v \cdot m( L(w) ) \cdot (2 d)^{ 1 - L(w) }\ind[\, w(0) \in \ball(y_k, 2 \Rs),\, L(w) \leq \Rs^2\, ]$, and \eqref{condindep_package} follows.

Using \eqref{condindep_package} the inequality $\mathbb{P}( \pi_{y_k}=1 \, | \, \mathcal{G}_k )\geq 3/4$ follows using the hypothesis of Lemma \ref{lemma_doubling_gamma} (cf.\ \eqref{induction_hypotheses}) and translation invariance (cf.\ \eqref{shift_invar_package}). From this $\mathbb{P}( \pi_{y_k}=1 \, | \, \pi_{y_1},\dots,\pi_{y_{k-1}} )\geq 3/4$ also follows, which is enough to conclude the proof of Claim \ref{claim_coupling_statement}.
\end{proof}

\subsubsection{Bounding the total length of worms that visit multiple good sets}\label{subsub_double_visits}

Section \ref{subsub_double_visits} is devoted to the proof of Lemma \ref{lemma_expect_bound_interactions}.
Given $y \neq y' \in \mathcal{D}$, let us define
\begin{equation}
	\label{worms:hit_Hy_and_Hy_prime}
	\overline{\mathcal{X}}[y,y']  := \mathcal{X}_{\bullet} \ind \big[ \, T_{H^y}(w) < L(w),\, T_{H^{y'}}(w) < L(w)  \, \big], \quad
	\overline{\Sigma}[y,y']:=\Sigma^{\overline{\mathcal{X}}[y,y']}.
\end{equation}
In words,  $\overline{\Sigma}[y,y']$ denotes the total length of worms from $\mathcal{X}_{\bullet}$ that hit both $H^y$ and $H^{y'}$.

\begin{lemma}[Upper bound on total length of worms hitting both $H^y$ and $H^{y'}$]
	\label{upper:cond_exp_multi_restrict_measure_total}
	There exists $C < \infty$ such that if $\Dlow \geq 16$, then for any $y \neq y' \in \mathcal{D}$, we have
	\begin{equation}
		\label{eq:upper:cond_exp_multi_restrict_measure_total}
		\mathbb{E} \big[ \, \overline{\Sigma}[ y, y' ] \, \big| \, \Falg \, \big] \leq
		C \cdot \alpha \cdot (\Rs^2 \cdot \gamma)^2 \cdot | y - y' |^{2 - d} \cdot \xi_y \cdot \xi_{y'}.
	\end{equation}
\end{lemma}

Before we prove Lemma \ref{upper:cond_exp_multi_restrict_measure_total}, let us first deduce the proof of Lemma \ref{lemma_expect_bound_interactions} from it.

\begin{proof}[Proof of Lemma \ref{lemma_expect_bound_interactions}]
Recalling the definition of $\overline{\Sigma}[y]$ from \eqref{worms:hit_more_than_one}, let us observe that we have
\begin{align}
	\label{eq:upper:overline_Sigma_y_yprime_sum}
	\overline{\Sigma}[y] \leq \sum_{y' \in \mathcal{D} \setminus \{ y \}} \overline{\Sigma}[y, y'] \cdot \xi_{y'}, \qquad y \in \mathcal{D}.
\end{align}
Using this, we can start to bound  the l.h.s.\ of \eqref{eq_expect_bound_interactions}:
\begin{equation}
	\label{eq:hitonly_Sigma_approach_SECONDTERM}
	\mathbb{E} \left[ \sum_{y \in \mathcal{D}} \overline{\Sigma}_c[y] \cdot \xi_y \right]\stackrel{ \eqref{kieg_vava} }{=}
\mathbb{E} \left[ \sum_{y \in \mathcal{D}} \overline{\Sigma}[y] \cdot \xi_y \right] \stackrel{ \eqref{eq:upper:overline_Sigma_y_yprime_sum} }{\leq}
\mathbb{E} \left[  \sum_{y \in \mathcal{D}} \sum_{y' \in \mathcal{D} \setminus \{ y \}} \overline{\Sigma}[y, y'] \cdot \xi_y \cdot \xi_{y'} \right].
\end{equation}

\noindent Lemma \ref{upper:cond_exp_multi_restrict_measure_total} and the law of total expectation gives
\begin{align}
	& \mathbb{E} \left[ \sum_{y \in \mathcal{D}} \sum_{y' \in \mathcal{D} \setminus \{ y \}} \overline{\Sigma}[y, y'] \cdot \xi_y \cdot \xi_{y'} \right]	
	\stackrel{\eqref{eq:upper:cond_exp_multi_restrict_measure_total}}{\leq}
	C \cdot \alpha \cdot \left( \Rs^2 \cdot \gamma \right)^2 \cdot \mathbb{E} \left[ \sum_{y \in \mathcal{D}} \sum_{ y' \in \mathcal{D} \setminus \{ y \} } \xi_y \cdot \xi_{y'} \cdot |y - y'|^{2 - d} \right]
	\notag \\ \label{bbbound} & \qquad \qquad \stackrel{(*)}{\leq}
	\CExpectBoundInteractionsUPPER \cdot \alpha \cdot \left( \Rs^2 \cdot \gamma \right)^2 \cdot \left( \Rb \slash \Rs  \right)^2 \cdot \Rs^{2 - d}
	\stackrel{(**)}{\leq} \CExpectBoundInteractionsUPPER \cdot (\Rb^2 \cdot \gamma \cdot  \alpha)  \cdot \psi,
\end{align}
where $(*)$ follows from \eqref{eq:upper:weighted_exp_visited_boxes} with $r = \Rs$ and $R = \Rb$,  and in $(**)$ we used  $\gamma \cdot \Rs^{4-d} \leq \psi$ (cf.\ \eqref{gamma_ineq_assumptions}). Putting together \eqref{eq:hitonly_Sigma_approach_SECONDTERM} and \eqref{bbbound}, the proof of Lemma \ref{lemma_expect_bound_interactions} follows.
\end{proof}

\begin{proof}[Proof of Lemma \ref{upper:cond_exp_multi_restrict_measure_total}] Recall \eqref{X_bull_indep_of_Falg} and the notation $P_{x}(. \, | \, \Falg)$ and $E_{x}[. \, | \, \Falg]$ introduced at the beginning of Section \ref{subsection_small_energy}.
	Let us fix $y \neq y' \in \mathcal{D}$ and define the function $f_{y,y'} \, : \, \widetilde{W} \rightarrow \R^+$ by
	\begin{equation}
		f_{y,y'}(w) := \ind \left[ w \in A_{\bullet}, \, T_{H^y}(w) < L(w),\, T_{H^{y'}}(w) < L(w) \right] \cdot L(w).
	\end{equation} 	
	The conditional law of $\overline{ \Sigma }[y,y']$ given $\Falg$ is the same as the conditional law of $\sum_{i \in I} f_{y,y'}(w_i)$ given $\Falg$, where $\sum_{i \in I} \delta_{w_i} \sim \mathcal{P}_v$ is independent of $\Falg$, thus we can use \eqref{eq:one_worm_ppp_expectation_formula} to obtain
	\begin{align}
		\mathbb{E} \left[ \overline{\Sigma}[y, y'] \, \big| \, \Falg \right] & =
		v \cdot \sum_{\ell = \Dlow \Rs^2}^{ \Dup \Rb^2 } \ell \cdot m(\ell) \cdot
		\sum_{x \in \mathcal{G}} P_x \left( T_{H^y} < \ell, \, T_{H^{y'}} < \ell, T_{\ball( 3 \Rb )^c} > \ell  \, \big| \, \Falg \right)
		\\ & \stackrel{ \eqref{eq:upper:two_cons_hitting_sum} }{\leq}
		C \cdot v \cdot \sum_{\ell = \Dlow \Rs^2}^{ \Dup \Rb^2 } \ell \cdot m(\ell) \cdot \ell \cdot \capacity(H^y) \cdot \capacity( H^{y'} ) \cdot |y - y'|^{ 2 - d } \cdot \xi_y \cdot \xi_{y'}
		\\ & \stackrel{ \eqref{assumption:alpha_EQ} }{=}
		C \cdot \alpha \cdot \capacity(H_y) \cdot \capacity(H_{y'}) \cdot | y - y' |^{ 2 - d } \cdot \xi_y \cdot \xi_{y'}.
	\end{align}
	From this, the upper bound \eqref{eq:upper:cond_exp_multi_restrict_measure_total} follows by \eqref{capa_H_y_uuper_c_star}.
\end{proof}

\bigskip

{\bf Acknowledgements:} We thank Art\"em Sapozhnikov for drawing our attention to the papers \cite{Chang2021,Gouere2008}  and for useful discussions about the similarities and differences between loops and worms. We also thank Aernout van Enter, Olof Elias and Caio Alves for useful discussions.
 The work of B.~R\'ath was partially supported by grant NKFI-FK-123962 of NKFI (National Research, Development and Innovation Office), the Bolyai Research Scholarship of the Hungarian Academy of Sciences, the \'UNKP-20-5-BME-5 New National Excellence Program of the Ministry for Innovation and Technology, and the ERC Synergy under Grant No. 810115 - DYNASNET. The work of S.~Rokob was partially supported by the ERC Consolidator Grant 772466 ``NOISE''.


\begin{thebibliography}{99}


	
	

\bibitem{AlvesSapozhnikov2019}
		C.\ Alves and A.\ Sapozhnikov (2019)
		Decoupling inequalities and supercritical percolation for the vacant set of random walk loop soup.
		\textit{Electron. J. Probab.}
		\textbf{24}, 1--34.


\bibitem{capacity_rw_range_d}
A.\ Asselah, B.\ Schapira and P.\ Sousi (2018).
 Capacity of the range of random walk on $\mathbb{Z}^d$.
 \textit{Transactions of the American Mathematical Society},
  \textbf{370}(11), 7627--7645.

\bibitem{capacity_rw_range_4}
A.\ Asselah, B.\ Schapira and P.\ Sousi (2019).
 Capacity of the range of random walk on $\mathbb{Z}^{4}$.
 \textit{The Annals of Probability},
 \textbf{47}(3), 1447--1497.

\bibitem{biggins_shape}
J.\ D.\ Biggins (1978).
The asymptotic shape of the branching random walk.
\textit{Advances in Applied Probability}, \textbf{10}(1), 62--84.

\bibitem{Bowen2019}
L.\ Bowen (2019)
Finitary random interlacements and the Gaboriau–-Lyons problem. \textit{Geom. Funct. Anal.}
\textbf{29}, 659-–689.

\bibitem{CampaninoRusso1985}
M. Campanino and L. Russo (1985)
An upper bound on the critical percolation probability for the three- dimensional cubic lattice.
\textit{The Annals of Probability}
\textbf{13}, 478--491.


\bibitem{CaiHanYeZhang2020}
		Z.\ Cai, X.\ Han, J.\ Ye and Y.\ Zhang (2022)
		On chemical distance and local uniqueness of a sufficiently supercritical finitary random interlacement.
		\textit{to appear in Journal of Theoretical Probability}.

\bibitem{CaiProcacciaZhang2021}
Z.\ Cai, E.\ B.\ Procaccia and Y.\ Zhang (2021)
 Continuity and uniqueness of percolation critical parameters in Finitary Random Interlacements.
  \textit{arXiv:2109.11756}.


\bibitem{CaiXiongZhang2021}
Z.\ Cai, Y.\ Xiong, and Y.\ Zhang (2021)
 On (non-) monotonicity and phase diagram of finitary random interlacement.
 \textit{Entropy} \textbf{23}(1): 69.

\bibitem{CaiZhang2021b}
Z.\ Cai,  Y.\ Zhang (2021)
On the exact orders of critical value in Finitary Random Interlacements.
  \textit{arXiv:2112.01136}.

	\bibitem{ChangSapozhnikov2016}
		Y.\ Chang and A.\ Sapozhnikov (2016)
		Phase transition in loop percolation.
		\textit{Probab. Theory Relat. Fields}
		\textbf{164}, 979-–1025.

\bibitem{Chang2017}
		Y.\ Chang (2017)
		Supercritical loop percolation on $\mathbb{Z}^d$ for $d \geq 3$. \textit{Stochastic Processes and their Applications.}
		\textbf{127}, 3159--3186.

\bibitem{Chang2021}
Y.\ Chang (2021)
 Bernoulli hyper-edge percolation on $\mathbb{Z}^d$.
  \textit{arXiv:2101.06082}



\bibitem{drs}
A.\ Drewitz, B.\ R\'ath and A.\ Sapozhnikov (2014).
 \textit{An introduction to random interlacements.}
  Berlin: Springer.

\bibitem{dvoer}
A.\ Dvoretzky and P.\ Erd\H{o}s (1951).
Some problems on random walk in space.
\textit{Proc. Second Berkeley Symposium}, 353--367.


\bibitem{ErhardMartinezPoisat2017}
	D.\ Erhard, J.\ Mart\'{i}nez and J.\ Poisat (2017)
	Brownian paths homogeneously distributed in space: percolation
	phase transition and uniqueness of the unbounded cluster.
	\textit{J. Theoret. Probab.}
	\textbf{30}, 784-–812.
	
	\bibitem{ErhardPoisat2016}
	D.\ Erhard and J.\ Poisat (2016)
	Asymptotics of the critical time in Wiener sausage percolation with a
	small radius.
	\textit{ALEA Lat. Am. J. Probab. Math. Stat.}
	\textbf{13}, 417-–445.



\bibitem{GrimmettMarstrand1990}
G.\ R.\ Grimmett and J.\ M.\ Marstrand (1990)
The supercritical phase of percolation is well behaved.
\textit{Proc. R. Soc. Lond. A}
\textbf{430}, 439--457.

\bibitem{Gr99} G. R. Grimmett (1999)
 \textit{Percolation}.
Springer-Verlag Berlin (Second edition).



\bibitem{Gouere2008}
	J.--B. Gou\'er\'e (2008)
	Subcritical regimes in the Poisson Boolean model of continuum percolation.
	\textit{The Annals of Probability}
	\textbf{36}, 1209--1220.

\bibitem{HilarioUngaretti2021}
	M.\ Hil\'{a}rio and D. Ungaretti (2021)
	Euclidean and chemical distances in ellipses percolation.
	\textit{arXiv:2103.09786}.

\bibitem{jain_orey_68}
N.\ C.\ Jain and S.\ Orey (1968) On the range of random walk.
\textit{Israel J. Math.}, \textbf{6.4}, 373--380.

\bibitem{jain}
 N.\ C.\ Jain and S.\ Orey (1973) Some properties of random walk paths.
 \textit{J. Math. Anal. Appl.} \textbf{43}, 795--815.


\bibitem{LawlerLimic2010}
G. Lawler and V. Limic (2010)
\textit{Random Walk: A Modern Introduction},
Cambridge University Press.




\bibitem{LeJanLemaire2013}
		Y.\ Le Jan and S.\ Lemaire (2013)
		Markovian loop clusters on graphs.
		\textit{Illinois J. Math.}
		\textbf{57}, 525--558.

\bibitem{MeesterRoy1996}
	R.\ Meester and R.\ Roy (1996)
	\textit{Continuum Percolation},
	Cambridge University Press.

\bibitem{ProcacciaYeZhang2021}
		E.\ B.\ Procaccia, J.\ Ye and Y.\ Zhang (2021)
		Percolation for the finitary random interlacements.
		\textit{ALEA, Lat. Am. J. Probab. Math. Stat.}
		\textbf{18}, 265-–287.

\bibitem{RathSapozhnikov2012}
B. R{\'a}th and A. Sapozhnikov (2012)
Connectivity properties of random interlacement and intersection of random walks.
\textit{ALEA, Lat. Am. J. Probab. Math. Stat.}
\textbf{9}, 67--83.

\bibitem{Sznitman2010}
	A.-S.\ Sznitman (2010)
	Vacant set of random interlacements and percolation.
	\textit{Ann. Math. (2).}
	\textbf{171}, 2039--2087.	

	\bibitem{TeixeiraUngaretti2017}
	A. Teixeira and D. Ungaretti (2017)
	Ellipses percolation.
	\textit{Journal of Statistical Physics}
	\textbf{38}, 369--393.

\bibitem{TykessonWindisch2012}
	J.\ Tykesson and D.\ Windisch (2012)
	Percolation in the vacant set of Poisson cylinders.
	\textit{Probability Theory and Related Fields}
	\textbf{154}, 165--191.

\end{thebibliography}
\end{document}